\iffalse %%
\documentclass[10pt, a4paper, oneside, reqno]{article}
\usepackage[total={7in, 10.5in}]{geometry}
\else 
\documentclass[letterpaper, 10 pt, journal, twoside]{ieeetran}
\fi 
\usepackage[noadjust]{cite}
\usepackage{amsmath, amssymb,mathrsfs, dsfont} %, amsthm
\usepackage{amsmath,amssymb,bm,bbm}%,bm
\usepackage{ color, bm,bbm,authblk}
\usepackage{ graphicx,graphics, epstopdf} %,subcaption, transparent
\usepackage{algorithm, algorithmicx,multirow, titlecaps} % 
\usepackage{threeparttable}
\usepackage{tikzsymbols}
\usepackage{caption}

\DeclareMathAlphabet{\mathbbold}{U}{bbold}{m}{n}

\usepackage[english]{babel}
\usepackage[noend]{algpseudocode} 

\usepackage{tabularx}
\usepackage{booktabs}
\usepackage{multirow}
\usepackage{courier}
\usepackage{float}
\usepackage{makecell}
\setcellgapes{5pt}
\newcommand{\cmr}[1]{{\color{red}{#1}}}

\newcommand{\stiny}[1]{{\scalebox{.5}{#1}}}
\newcommand{\smtiny}[1]{{\scalebox{.63}{#1}}}

\newcommand{\argmin}{{\mathrm{argmin}}}

\newcommand*{\minOp}{\operatornamewithlimits{min}\limits}

\newcommand*{\supOp}{\operatornamewithlimits{sup}\limits}

\newcommand*{\sumOp}{\operatornamewithlimits{\sum}\limits}

\newcommand*{\limOp}{\operatornamewithlimits{lim}\limits}
\newcommand{\tolimOp}{\operatornamewithlimits{\longrightarrow}\limits}

\newcommand{\tr}{\mathsf{\scriptscriptstyle T}}

\newcommand{\zero}{\mathbf{0}}
\newcommand{\one}{\mathbf{1}}
\iftrue%\iffalse %<-<-<-<-<-<-<-<-<-<-<-<-<-<-<
\newcommand{\eye}{\mathbb{I}}
\else
\newcommand{\eye}{\mathbf{I}}
\fi%<-<-<-<-<-<-<-<-<-<-<-<-<-<-<-<-<-<-<-<-<-<

\newcommand{\vc}[1]{{ \mathrm{#1} }}
\newcommand{\mx}[1]{{ \mathrm{#1} }}
\newcommand{\drm}{\mathrm{d}}
\newcommand{\linspan}{\mathrm{span}} % span
\newcommand{\proj}{\mathrm{proj}} % projection
\newcommand{\norm}[1]{\|#1\|}
\newcommand{\inner}[2]{{ \langle {#1,#2} \rangle}}

\newcommand{\imag}[1]{\mathrm{imag}(#1)}%imaginary part
\newcommand{\real}[1]{\mathrm{real}(#1)}%real part
\newcommand{\ellone}{\ell^{1}}
\newcommand{\ellinfty}{\ell^{\infty}}

\newcommand{\Lone}{L^{1}}
\newcommand{\Linfty}{L^{\infty}}

\newcommand{\Lscrone}{\Lscr^{1}}
\newcommand{\Lscrinfty}{\Lscr^{\infty}}
\newcommand{\Lscrp}{\Lscr^{p}}

\newcommand{\Ascr}{{\mathscr{A}}}
\newcommand{\Bscr}{{\mathscr{B}}}

\newcommand{\Dscr}{{\mathscr{D}}}

\newcommand{\Fscr}{{\mathscr{F}}}
\newcommand{\Gscr}{{\mathscr{G}}}
\newcommand{\Hscr}{{\mathscr{H}}}

\newcommand{\Lscr}{{\mathscr{L}}}

\newcommand{\Pscr}{{\mathscr{P}}}

\newcommand{\Tscr}{{\mathscr{T}}}
\newcommand{\Uscr}{{\mathscr{U}}}
\newcommand{\Vscr}{{\mathscr{V}}}
\newcommand{\Wscr}{{\mathscr{W}}}
\newcommand{\Xscr}{{\mathscr{X}}}

\newcommand{\Ccal}{{\mathcal{C}}}

\newcommand{\Ecal}{{\mathcal{E}}}
\newcommand{\Fcal}{{\mathcal{F}}}
\newcommand{\Gcal}{{\mathcal{G}}}
\newcommand{\Hcal}{{\mathcal{H}}}
\newcommand{\Ical}{{\mathcal{I}}}
\newcommand{\Jcal}{{\mathcal{J}}}

\newcommand{\Lcal}{{\mathcal{L}}}

\newcommand{\Ncal}{{\mathcal{N}}}

\newcommand{\Rcal}{{\mathcal{R}}}
\newcommand{\Scal}{{\mathcal{S}}}

\newcommand{\Ycal}{{\mathcal{Y}}}

\newcommand{\Cbb}{{\mathbb{C}}}

\newcommand{\Nbb}{{\mathbb{N}}}

\newcommand{\Rbb}{{\mathbb{R}}}
\newcommand{\Sbb}{{\mathbb{S}}}
\newcommand{\Tbb}{{\mathbb{T}}}

\newcommand{\Xbb}{{\mathbb{X}}}
\newcommand{\Ybb}{{\mathbb{Y}}}
\newcommand{\Zbb}{{\mathbb{Z}}}

\newcommand{\bbh}{\mathds{h}}
\newcommand{\bbk}{\mathbbm{k}}

\usepackage{amsthm}

\newtheorem{theorem}{Theorem}
\newtheorem*{theorem*}{Theorem}
\newtheorem{definition}{Definition}
\newtheorem*{definition*}{Definition}
\newtheorem{assumption}{Assumption}
\newtheorem{proposition}[theorem]{Proposition}
\newtheorem{corollary}[theorem]{Corollary}
\newtheorem{lemma}[theorem]{Lemma}
\newtheorem{remark}{Remark}
\newtheorem{example}{Example}
\newtheorem*{example*}{Example}

\newtheorem*{claim*}{Claim}
\newtheorem{problem}{Problem}
\newtheorem*{problem*}{Problem}

\newcommand\xqed[1]{%
	\leavevmode\unskip\penalty9999 \hbox{}\nobreak\hfill
	\quad\hbox{#1}}
\iftrue %<--\iffalse
\usepackage[colorlinks=true]{hyperref}
	\usepackage{xcolor}
	\hypersetup{
		colorlinks,
		linkcolor={blue!55!black},
		citecolor={red!55!black},
		urlcolor={blue!80!black}
	}
	\allowdisplaybreaks
	
\fi
\newcommand{\nD}{n_{\stiny{$\!\Dscr$}}}
\newcommand{\nS}{n_{\text{s}}}

\newcommand{\Lu}[1]{\mx{L}^{\!\vc{u}}_{#1}}
\newcommand{\gS}{\vc{g}^{\stiny{$(\Scal)$}}}
\newcommand{\gtS}{{g}^{\stiny{$(\Scal)$}}}
\newcommand{\GS}{G^{\stiny{$(\Scal)$}} }
\newcommand{\GStilde}{G^{\stiny{$(\tilde{\Scal})$}} }
\newcommand{\GSbar}{G^{\stiny{$(\overline{\Scal})$}} }
\newcommand{\gSbar}{\vc{g}^{\stiny{$(\overline{\Scal})$}}}
\newcommand{\GSDelta}{G^{\stiny{$(\Delta\Scal)$}} }
\newcommand{\gSDelta}{\vc{g}^{\stiny{$(\Delta\Scal)$}}}

\newcommand{\geps}{\vc{g}^{(\epsilon)}}
\newcommand{\xeps}{\vc{x}^{(\epsilon)}}
\newcommand{\veps}{\vc{v}^{(\epsilon)}}
\newcommand{\Gstareps}{G^{(\star,\epsilon)}}
\newcommand{\gstareps}{\vc{g}^{(\star,\epsilon)}}
\newcommand{\xstareps}{\vc{x}^{(\star,\epsilon)}}
\newcommand{\vstareps}{\vc{v}^{(\star,\epsilon)}}
\newcommand{\gstar}{\vc{g}^{\star}}
\newcommand{\xstar}{\vc{x}^{\star}}
\newcommand{\vstar}{\vc{v}^{\star}}
\newcommand{\gstart}[1]{g_{#1}^{\star}}
\newcommand{\gstarepst}[1]{g_{#1}^{(\star,\epsilon)}}
\newcommand{\gSproj}{\vc{g}^{\stiny{$(\Scal)$}}_{\perp}}

\newcommand{\vcg}{\vc{g}}
\newcommand{\vcu}{\vc{u}}
\newcommand{\vcv}{\vc{v}}
\newcommand{\vcx}{\vc{x}}
\newcommand{\vcy}{\vc{y}}
\newcommand{\mxA}{\mx{A}}
\newcommand{\vcq}{\vc{q}}

\newcommand{\vca}{\vc{a}}
\newcommand{\vcb}{\vc{b}}
\newcommand{\vcc}{\vc{c}}

\newcommand{\mgstareps}{m_{\gstareps}}
\newcommand{\gone}{g^{\stiny{$(1)$}} }
\newcommand{\gtwo}{g^{\stiny{$(2)$}} }
\newcommand{\vcgone}{\vc{g}^{\stiny{$(1)$}} }
\newcommand{\vcgtwo}{\vc{g}^{\stiny{$(2)$}} }

\newcommand{\kernel}{\bbk}
\newcommand{\Hk}{\Hscr_{\bbk}}
\newcommand{\Vk}{\Vscr_{\bbk}}

\newcommand{\phiu}[1]{\varphi_{#1}^{\text{\rm{(u)}}}}

\newcommand{\expe}{\mathrm{e}}
\newcommand{\Jimage}{\mathrm{j}}

\newcommand{\Fw}{\Fcal_\omega}
\newcommand{\Fwr}{\Fcal_\omega^{\text{\rm{(r)}}}}
\newcommand{\Fwi}{\Fcal_\omega^{\text{\rm{(i)}}}}
\newcommand{\Fwrj}[1]{\Fcal_{\omega_{#1}}^{\text{\rm{(r)}}}}
\newcommand{\Fwij}[1]{\Fcal_{\omega_{#1}}^{\text{\rm{(i)}}}}
\newcommand{\phiwr}{\varphi_\omega^{\text{\rm{(r)}}}}
\newcommand{\phiwi}{\varphi_\omega^{\text{\rm{(i)}}}}
\newcommand{\phiwrk}[1]{\varphi_{\omega,#1}^{\text{\rm{(r)}}}}
\newcommand{\phiwik}[1]{\varphi_{\omega,#1}^{\text{\rm{(i)}}}}

\newcommand{\phiwrone}{\varphi_{\omega_1}^{\text{\rm{(r)}}}}
\newcommand{\phiwrtwo}{\varphi_{\omega_2}^{\text{\rm{(r)}}}}
\newcommand{\phiwione}{\varphi_{\omega_1}^{\text{\rm{(i)}}}}
\newcommand{\phiwitwo}{\varphi_{\omega_2}^{\text{\rm{(i)}}}}

\newcommand{\phir}[1]{\varphi_{#1}^{\text{\rm{(r)}}}}
\newcommand{\phii}[1]{\varphi_{#1}^{\text{\rm{(i)}}}}
\newcommand{\mg}{m_{\vc{g}}}
\newcommand{\mgr}{m_{\vc{g}}^{(\mathrm{r})}}
\newcommand{\mgi}{m_{\vc{g}}^{(\mathrm{i})}}

\newcommand{\ddomega}{\frac{\mathrm{d}}{\mathrm{d}\omega}}
\newcommand{\domega}{\mathrm{d}\omega}

\newcommand{\wmax}{\omega_{\mathrm{max}}}

\newcommand{\mesh}{\mathrm{mesh}}
\newcommand{\nP}{n_{\stiny{$\!\Pscr$}}}

\newcommand{\Iu}{\Ical^{\smtiny{$(\mathrm{u})$}}}

\newcommand{\IP}{\Ical^{\smtiny{$(\mathscr{P})$}}}
\newcommand{\Iall}{\Ical}
\newcommand{\Pscrbar}{\overline{\Pscr}}
\newcommand{\nPbar}
{n_{{\stiny{$\overline{\Pscr}$}}}}

\newcommand{\gstarepsbar}{\overline{\vc{g}}^{(\star,\epsilon)}}
\newcommand{\omegabar}{\overline{\omega}}
\newcommand{\IPbar}{\Ical^{\smtiny{$(\overline{\mathscr{P}})$}}}
\newcommand{\mbar}{\overline{m}}
\newcommand{\Phibar}{\overline{\Phi}}
\newcommand{\phibar}{\overline{\varphi}}
\newcommand{\vcabar}{\overline{\vc{a}}}
\newcommand{\vcbbar}{\overline{\vc{b}}}
\newcommand{\vccbar}{\overline{\vc{c}}}

\newcommand{\FDIsub}{\text{FDI}_{\text{sub}}}
\newcommand{\FDIRKHS}{\text{FDI}_{\text{RKHS}}}

\newcommand{\alphadc}{\alpha_{\text{dc}}}

\newcommand{\FBsub}{\text{FB}_{\text{sub}}}
\newcommand{\FBRKHS}{\text{FB}_{\text{RKHS}}}

\usepackage{multirow}
\usepackage{booktabs}
\usepackage{array}
\usepackage{lineno}
\title{\LARGE Kernel-Based Identification with Frequency Domain Side-Information}
\author{M.~Khosravi and R.~S.~Smith
\thanks{Mohammad Khosravi is with  Delft Center for Systems and Control, Delft University of Technology,
	Delft, The Netherlands (email: Mohammad.khosravi@tudelft.nl).}
\thanks{Roy~S.~Smith is with with Automatic Control Laboratory, ETH Z\"urich, Switzerland (email: rsmith@control.ee.ethz.ch).}
}

\begin{document}
\maketitle
%________________________________________
\begin{abstract}
In this paper, we discuss the problem of system identification when frequency domain side information is available on the system. Initially, we consider the case where the prior knowledge is provided as being the $\Hcal_{\infty}$-norm of the system bounded by a given scalar. This framework provides the opportunity of considering various forms of side information such as the dissipativity of the system as well as other forms of frequency domain prior knowledge. We propose a nonparametric identification method for estimating the impulse response of the system under the given side information. The estimation problem is formulated as an optimization in a reproducing kernel Hilbert space (RKHS) endowed with a stable kernel. The corresponding objective function consists of a term for minimizing the fitting error, and a regularization term defined based on the norm of the impulse response in the employed RKHS. To guarantee the desired frequency domain features defined based on the prior knowledge, suitable constraints are imposed on the estimation problem. The resulting optimization has an infinite-dimensional feasible set with an infinite number of constraints. We show that this problem is a well-defined convex program with a unique solution. We propose a heuristic that tightly approximates this unique solution. The proposed approach is equivalent to solving a finite-dimensional convex quadratically constrained quadratic program. The efficiency of the discussed method is verified by several numerical examples.	
\end{abstract}

\section{Introduction}\label{sec:introduction}
% ============================================
% 40-50 references
% ============================================
% From system identification to side information:
System identification problem, initially introduced in \cite{zadeh1956identification}, deals with the theory and techniques for estimating suitable mathematical models describing dynamical systems using measurement data. The topic has received substantial attention according to its broad applicability in numerous phenomena in science and technology \cite{LjungBooK2, ljung2010perspectives, schoukens2019nonlinear}. In many situations, identifying a dynamical system is beyond fitting a mathematical model to the input-output measurement data. More precisely, we may additionally need to integrate a specific attribute or a known feature of the system into the model. This side information is possibly provided from our general understanding of the behavior of the system based on its inherent physical nature, or acquired from qualitative characteristics and phenomena observed from historical or experimental data.
For example, in the identification of nonlinear dynamics  various forms of side information such as stability, region of attraction, dissipativity and many other ones are incorporated \cite{umenberger2018specialized,khosravi2021ROA,hara2019dissipativityKoopman,khosravi2021grad,ahmadi2020learning,khosravi2022Koopman}.

The incorporation of side information has been considered in the identification of linear dynamics, e.g., the low complexity of the model is imposed by considering sparsity promoting regularizations \cite{khosravi2020low, khosravi2020regularized,shah2012linear,pillonetto2016AtomicNuclearKernel,smith2014frequency}.
For penalizing the order of systems, the rank and the nuclear norm of the corresponding Hankel matrix are utilized in  \cite{smith2014frequency, pillonetto2016AtomicNuclearKernel,fazel2013hankel, mohan2010reweighted}. 
For the same purpose, the notion of  atomic transfer functions and regularization based on the atomic norm are employed \cite{shah2012linear,khosravi2020low}.
Identification with the prior knowledge involving positivity features of the system such as  compartmental structure and the internal or external positivity are discussed as well in the literature   \cite{khosravi2019positive,grussler2017identification,umenberger2016scalable, de2002identification, zheng2018positive, benvenuti2002model}.
Other forms of side information are studied in  \cite{goethals2003identification, hoagg2004first, okada1996subspace, miller2013subspace, inoue2019subspace, abe2016subspace}, e.g., including information on the location of the eigenvalues is discussed in \cite{okada1996subspace, miller2013subspace}, 
positive-realness is studied in \cite{goethals2003identification, hoagg2004first}, and
utilizing prior knowledge on the moments of the transfer function is discussed in \cite{inoue2019subspace}.
The subspace identification method is employed to include information on the steady-state behavior \cite{alenany2011improved, yoshimura2019system, lacy2003subspace} such as the  stability of the system \cite{lacy2003subspace}. %
%\cmr{For example, the subspace identification method is combined with the a priori information including stability by \cite{lacy2003subspace}, eigenvalue location by \cite{okada1996subspace}; \cite{miller2013subspace}, steady-state property by \cite{alenany2011improved}; \cite{yoshimura2019system}, moments by \cite{inoue2019subspace}, and positive-realness studied by \cite{goethals2003identification}; \cite{hoagg2004first}, more general frequency-domain property by \cite{abe2016subspace}.}

% ============================================
% Kernel based idea
% On the other hand, 
%Motivated by including the stability property in the identification problem, 
Starting from the seminal work Pillonetto and De Nicolao \cite{pillonetto2010new}, a paradigm shift known as \emph{kernel-based} approach has been emerged in system identification and also the idea of integrating prior knowledge \cite{ljung2020shift,pillonetto2014kernel,chiuso2019system,khosravi2021robust}. 
In the proposed framework, the identification problem is formulated as a regularized regression in a reproducing kernel Hilbert space (RKHS) \cite{aronszajn1950theory} where
the regularization term is defined based on the norm of RKHS to penalize the feasible solutions not following the prior knowledge. 
Indeed, by suitable choice of the kernel function or imposing appropriate constraints in the regression problem, one can incorporate various prior knowledge such as the stability of the system, the resonant frequencies, gain, the smoothness of the impulse response, and internal or external positivity of the system \cite{chen2018kernel,khosravi2021SSG, marconato2016filter,khosravi2019positive,zheng2018positive,khosravi2021POS}.
Moreover, employing a Tikhonov-like regularization in this framework leads in improvement of the bias-variance trade-off \cite{pillonetto2014kernel}.

% ============================================
% Dissipativity information and its importance
Input-output behavioral properties such as the the $\Hcal_\infty$-norm of the plant, or more generally frequency domain properties like the dissipativity of the system \cite{willems1972dissipative}, can be strongly useful in feedback controller design \cite{zames1966input, brogliato2007dissipative}.
These features are a priori known for many systems due to their inherent physical nature, e.g., the electrical circuits where the energy is dissipated by the resistors %, the viscoelastic systems where there is a loss in energy according to viscous friction, and the open thermodynamical systems where there exists a form of dissipation, and subsequently increase of entropy, due to the second law of thermodynamics 
\cite{willems1972dissipative}.
%Moreover,  for verifying these $\Hcal_{\infty}$-norm or dissipativity of the system, 
Also, they can be verified using recently developed data-driven methods %are 
\cite{muller2017stochastic, romer2017determining, romer2019one, koch2019sampling, koch2020provably}.
Accordingly, knowledge on these valuable attributes are potentially available or can be extracted to be used later as side information.
However, in the existing research studies on the identification of LTI systems with prior knowledge \cite{abe2016subspace}, the question on efficient integration of the aforementioned frequency domain side information has not suitably addressed.   
For example, in the standard implementation of subspace and kernel-based identification methods, the information on the bound of system's $\Hcal_{\infty}$-norm is not encoded in the identified model (see the example in Section \ref{sec:problem_statement}).

The main goal of this paper is to develop identification methods utilizing frequency domain side information in a numerically tractable fashion and also providing suitable theoretical guarantees. First, we consider the case where the prior knowledge is provided as being the $\Hcal_{\infty}$-norm of the system bounded by a given scalar. This framework provides the opportunity of considering various forms of side information such as the dissipativity of the system as well as other types of frequency domain prior knowledge like the DC-gain of the system.  We propose a nonparametric identification method for estimating the impulse response of the system under the given side information. Since kernel-based approach provides a powerful framework, the identification problem is formulated as a constrained regularized optimization problem regression in an RKHS endowed with a stable kernel \cite{pillonetto2014kernel,chen2018stability,khosravi2022Lut}.  The objective function of the optimization problem consists of a term for minimizing the fitting error, and a regularization term defined based on the norm of the impulse response in the employed RKHS. To guarantee the desired frequency domain prior knowledge, suitable constraints are imposed on the estimation problem.  Accordingly, we obtain a regularized optimization problem in an infinite-dimensional space with an infinite number of constraints. Following this, we show that this problem is a convex program that attains a unique solution, and consequently, it is well-posed. Then, towards deriving a tractable scheme, we consider a suitable finite set of frequencies, and subsequently, a new optimization is formulated as an approximate estimation problem with the same objective function but with constraints defined on the given finite set. We show that the new problem attains a unique solution with an especial parametric form. Subsequently, an equivalent finite-dimensional convex optimization is derived as a convex quadratically constrained quadratic program (QCQP). More precisely, by solving this optimization problem, the coefficients in the parametric form are estimated, and hence, the solution of the approximate problem is obtained. We derive suitable bounds on the tightness of this approximation. Moreover, we provide theoretical guarantees on the convergence of the approximate solution to the solution of the original problem. The efficiency of the discussed method is verified by several numerical examples.

\section{Notations and Preliminaries}
The set of natural numbers, the set of non-negative integers, the set of real numbers, the set of non-negative real numbers, the set of complex numbers, the $n$-dimensional Euclidean space, the space of $n$ by $m$ real matrices, and  the space of $n$ by $n$ real symmetric matrices are denoted by $\Nbb$, $\Zbb_+$,  $\Rbb$, $\Rbb_+$, $\Cbb$, $\Rbb^n$,  $\Rbb^{n\times m}$, and $\Sbb^n$, respectively.
For any $z\in\Cbb$, the real and imaginary part of $z$ are denoted by $\real{z}$ and $\imag{z}$, respectively. 
The inner product and norm of Hilbert space $\Hscr$ is denoted by $\inner{\cdot}{\cdot}_{\Hscr}$ and $\norm{\cdot}_{\Hscr}$, respectively, and when it is clear from the context, we drop the subscript.
To handle discrete and continuous time in the same formulation, 
$\Tbb$ denotes either $\Zbb_+$ or $\Rbb_+$, and $\Tbb_{\pm}$ is the set of scalars $t$ where either $t\in\Tbb$ or $-t\in\Tbb$. 
Given measure space $\Xscr$, the space of measurable functions $g:\Xscr\to \Rbb$ is denoted by $\Rbb^{\Xscr}$.
The element $\vc{u}\in\Rbb^{\Xscr}$ is shown entry-wise as $\vcu=(u_x)_{x\in\Xscr}$, or equivalently as $\vcu=\big(u(x)\big)_{x\in\Xscr}$.
Depending on the context of discussion, $\Lscrinfty$ refers either to $\ellinfty(\Zbb)$ or $\Linfty(\Rbb)$. 
Similarly, $\Lscrone$ is either $\ellone(\Zbb_+)$ or $\Lone(\Rbb_+)$. 
For $p\in\{1,\infty\}$, the norm in $\Lscrp$ is denoted by $\|\cdot\|_{p}$.
Let $(\Xbb,\|\cdot\|_{\Xbb})$ and $(\Ybb,\|\cdot\|_{\Ybb})$ be two normed vector spaces. The set of linear bounded (continuous) operators $\mx{A}:\Xbb\to\Ybb$, denoted by $\Lcal(\Xbb,\Ybb)$,
is a normed vector space with the norm defined as $\|\mx{A}\|_{\Lcal(\Xbb,\Ybb)}:=\sup_{\vc{x}\in\Xbb,\|\vc{x}\|_{\Xbb}\le 1}\ \|\mx{A}\vc{x}\|_{\Ybb}$.
The identity matrix/operator and the zero vector are denoted by $\eye$ and $\zero$ respectively.
Given $\Vscr\subseteq \Xbb$, the linear span of $\Vscr$, denoted by $\linspan\Vscr$, is a linear subspace of $\Xbb$ containing linear combination of the elements of $\Vscr$.
Let $\Ycal$ be a set and $\Ccal\subseteq\Ycal$. 
We define the function $\delta_{\Ccal}$ as
$\delta_{\Ccal}(y) = 0$, if $y\in\Ccal$ and $\delta_{\Ccal}(y) = \infty$, otherwise.
Similarly, function $\mathbf{1}_{\Ccal}$ is defined as $\mathbf{1}_{\Ccal}(y) = 1$, if $y\in\Ccal$ and $\mathbf{1}_{\Ccal}(y) = 0$, otherwise.

\section{System Identification with Frequency Domain Side Information}\label{sec:problem_statement}
Consider a stable LTI system $\Scal$ described with impulse response $\gS:=(\gtS_t)_{t\in\Tbb}\in\Rbb^\Tbb$ and transfer function $\GS$. 
For the case of discrete-time and the case of continuous-time, we have here respectively $\Tbb :=\Zbb_+$ and $\Tbb :=\Rbb_+$.
%where for the discrete-time case we have $\Tbb :=\Zbb_+$ and the continuous-time case corresponds to $\Tbb := \Rbb_+$.
Let the system be actuated with a bounded input signal denoted by $\vc{u}=(u_t)_{t\in \Tbb}\in \Lscr_\infty$.
Accordingly, for any $t\in\Tbb$, 
%and stable impulse response $\vcg=(g_t)_{t\in\Tbb}\in\Rbb^\Tbb$, 
one can define linear map $\Lu{t}$ over the space of stable impulse responses, for the discrete-time case, as  
\begin{equation}\label{eqn:Lut_Z}
\Lu{t}(\vcg) := \sum_{s\in \Zbb_+}g_s u_{t-s},
\end{equation}
and similarly, for the case of continuous-time, as
\begin{equation}\label{eqn:Lut_R}
\Lu{t}(\vcg) := \int_{\Rbb_+}\!\!g_s u_{t-s}\ \!\drm s.
\end{equation}
Let the output of the system be measured at time instants $\Tscr:=\{t_i\ \!|\! \ i=0,\ldots,\nD\!-\!1\}$, for a given $\nD\in\Nbb$. More precisely, define $y_t$ as 
\begin{equation}\label{eqn:output_sys_S}
y_t := \Lu{t}(\gS)+w_t, \qquad t\in\Tscr,
\end{equation}
where,  for any $t\in\Tscr$, $w_t$ denotes the measurement uncertainty.
Subsequently, let $\Dscr$ be the set of input-output pairs, i.e., $\Dscr$ is defined as
$\Dscr = \{(u_t,y_t)\ | \ t\in\Tscr \}$.
%\begin{equation}\label{eqn:DataSet}
%	\Dscr = \{(u_t,y_t)\ | \ t\in\Tscr \}.
%\end{equation} 

%\subsection{Side-Information:  $\Lscr_2$-gain or $\Hcal_{\infty}$-norm}
Let assume we know that the $\Hcal_{\infty}$-norm of system $\Scal$ is bounded by a given scalar $\rho\in\Rbb_+$. 
The question is whether this side-information is naturally encoded in the identification problem. 
%The following example elaborate this issue.
The following example elaborates this issue by demonstrating that the information on the bound of the system's $\Hcal_{\infty}$-norm is not included in the models identified by the standard identification approaches such as the subspace and the kernel-based methods.
%----------------------------------------------
\begin{example*}\label{exm:pf}\normalfont
Let $\Scal$ be a discrete-time system described by the following transfer function
\begin{equation} \label{eqn:sys_example}
	\GS(z) = \frac{1}{2z-1}\ +\ \frac{3}{100}\ \frac{z-1}{z^2+z+0.9}.
%	\GS(z) = \frac{\frac12z^{-1}}{1-\frac12z^{-1}}\ +\ \frac{3}{100}\ \frac{z-1}{z^2+z+0.9}
%	\GS(z) = \frac{0.5\ \!z^{-1}}{1-0.5\ \!z^{-1}}\ +\ \frac{3}{100}\ \frac{z^{-1}-z^{-2}}{1+z^{-1}+0.9z^{-2}}
\end{equation}
This can be a model of a low-pass filter where its frequency response has a deformity in high frequencies potentially due to the impact of aging, or possible mistakes in design and implementation. To obtain set of data $\Dscr$, % in \eqref{eqn:DataSet}, 
we actuate the system with a random white Gaussian signal of length $\nD=150$, and then, the output of system is measured where SNR $= 14.5$ [dB] (see Figure \ref{fig:example}).
%We use \textsc{Matlab}'s \textsc{System Identification Toolbox} \cite{ljung2012version}  to identify the system. Accordingly, $\hat{G}_1$ and $\hat{G}_2$ are obtained using \texttt{impulseest} and \texttt{n4sid}, respectively.
Additionally, let assume we are given the side-information $\|\GS\|_{\Hcal_\infty}\le 1$ which can be due to the filter nature of $\Scal$. 
We employ \textsc{Matlab}'s \textsc{System Identification Toolbox} \cite{ljung2012version}  to estimate models $\hat{G}_1$ and $\hat{G}_2$ for the system using \texttt{impulseest} and \texttt{n4sid}, respectively.
The results are shown in Figure \ref{fig:example} where we have $\|\hat{G}_1\|_{\Hcal_\infty}=1.24$ and $\|\hat{G}_2\|_{\Hcal_\infty}=1.38$. 
%Accordingly, one can see that the estimated models do not follow the side-information.
One can see that the models estimated by the mentioned standard identification methods do not comply the side-information.
\end{example*}
\begin{figure}[t]
	\centering
	\iffalse
	\includegraphics[width =0.45\textwidth]{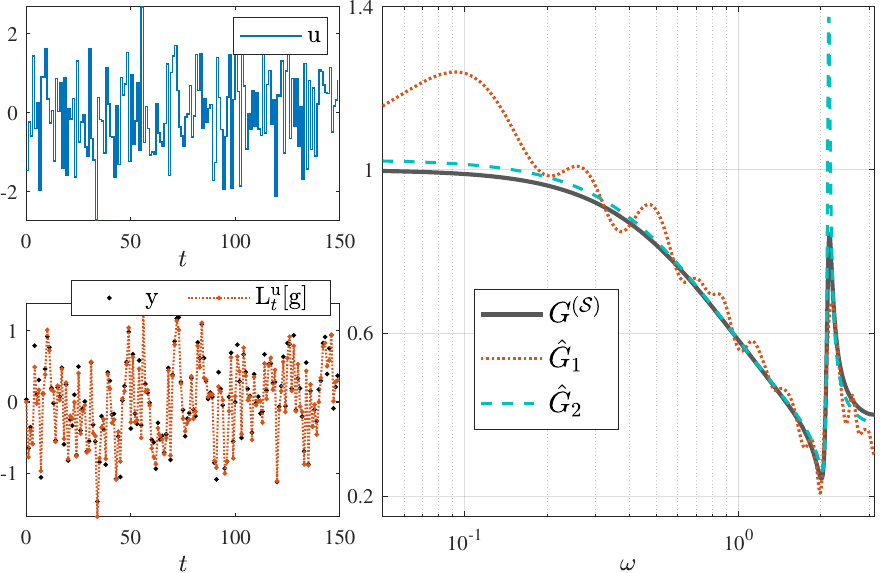}
	\caption{Left: The input-output pairs of data for system \eqref{eqn:sys_example} are shown. Right: The estimated models $\hat{G}_1$ and $\hat{G}_2$ for the system \eqref{eqn:sys_example} are compared with  the true transfer function.}
	\else%-----------
	\includegraphics[width =0.43\textwidth]{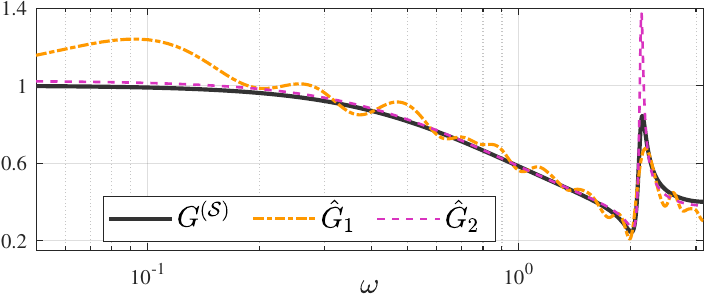}
	\caption{The estimated models $\hat{G}_1$ and $\hat{G}_2$ for the system \eqref{eqn:sys_example} are compared with  the true transfer function.}	
	\fi
	\label{fig:example}
\end{figure}
%----------------------------------------------
Motivated by this example, we introduce the main problem discussed in this paper as the system identification with side-information on the $\Hcal_{\infty}$-norm of the system.
\begin{problem}\label{prob:Hinft_le_1}
	Given the set of input-output data $\Dscr$, estimate the impulse response of system $\Scal$ satisfying the side-information $\|\GS\|_{\Hcal_\infty}\le \rho$, where $\rho$ is a given non-negative real scalar. 
\end{problem}

This problem can be extended to the identification problem with more general sense of dissipativity side-information. 
More precisely, let $\mx{Q}\in\Sbb^2$ be an indefinite matrix and assume that in addition to the given set of data $\Dscr$, we know that the system $\Scal$ is 
{\em dissipative} \cite{antoulas2005approximation} with respect to quadratic supply rate function $s_{\mx{Q}}(u,y)$ defined as 
\begin{equation}\label{eqn:supplyrate}
	\begin{split}
		s_{\mx{Q}}(u,y)&:= 
		\begin{bmatrix}u&y\end{bmatrix} 
		\mx{Q}\begin{bmatrix}u\\y\end{bmatrix}
		= 
		\begin{bmatrix}u&y\end{bmatrix} 
		\begin{bmatrix}q_{u}& q_{uy}\\q_{uy}&q_{y}\end{bmatrix} 
		\begin{bmatrix}u\\y\end{bmatrix}
		\\&= 
		q_{u} u^2 + 2 q_{uy} uy + q_{y} y^2, 
	\end{split}   
\end{equation}
where $q_y<0$ \footnote{In other words, given that the system is initially at rest,
for any input-output pairs $\big(u_t,y_y\big)_{t\in\Tbb}$ and any $\tau\in\Tbb$, we have 
$\int_{0}^{\tau}s_{\mx{Q}}(u_t,y_t)\drm t \ge 0$, if $\Tbb=\Rbb_+$, 
and $\sum_{t=0}^{\tau}s_{\mx{Q}}(u_t,y_t) \ge 0$, if $\Tbb=\Zbb_+$.
}\cite{antoulas2005approximation,haddad2011nonlinear}. 
For the case of $s_{\mx{Q}}(u,y)=  \rho^2 u^2 - y^2$ where $\rho\in\Rbb_+$, this dissipativity prior knowledge is equivalent to being $\Lscr_2$-gain or $\Hcal_\infty$-norm of the system not larger than $\rho$. 
Given the above side-information on the dissipativity of system $\Scal$, a desired identification procedure for estimating the impulse response of the system should suitably utilize this information and also guarantee that the identified model satisfies the given feature.
More precisely, one has to address the following problem:
\begin{problem}\label{prob:dissip}
	Given the set of data $\Dscr$ and considering the side-information on the dissipativity of system with respect to the supply rate $s_{\mx{Q}}(u,y)$, estimate the impulse response of system $\Scal$ satisfying the provided side-information. 
\end{problem}
Note that we have 
\begin{equation}
	\mx{Q} = \begin{bmatrix}q_{u}& q_{uy}\\q_{uy}&q_{y}
	\end{bmatrix} =
	\begin{bmatrix}l_1& l_2\\0&l_3
	\end{bmatrix} 
	\begin{bmatrix}1& 0\\0&-1
	\end{bmatrix} 
	\begin{bmatrix}l_1& 0\\l_2&l_3
	\end{bmatrix} 
	%=\begin{bmatrix}l_1& l_2\\0&l_3\end{bmatrix}
	%\begin{bmatrix}l_1& 0\\-l_2&-l_3\end{bmatrix} %=\begin{bmatrix}l_1^2-l_2^2&-l_2l_3\\-l_2l_3&-l_3^2\end{bmatrix}.
\end{equation}
where 
%$l_1=(q_{uy}^2-q_uq_y)/(-q_y)^{\frac{1}{2}}$,
$l_1=(\det\mx{Q}/q_y)^{\frac{1}{2}}$,
$l_2=-q_{uy}/(-q_y)^{\frac{1}{2}}$,
and 
$l_3 = (-q_y)^{\frac{1}{2}}$, 
%\begin{equation}
%l_1 = \frac{q_{uy}^2-q_uq_y}{(-q_y)^{\frac{1}{2}}}\qquad l_2 = -\frac{q_{uy}}{(-q_y)^{\frac{1}{2}}},
%\qquad l_3 = (-q_y)^{\frac{1}{2}}.
%\end{equation}
%
and subsequently, one can define system $\tilde{\Scal}$ with input $v = l_1 u$ and output $z = l_2u+l_3y$ as shown in Figure \ref{fig:Hsys}. Then, $\tilde{\Scal}$ is a dissipative system with respect to supply rate function $s(v,z) = u^2-y^2$, or equivalently, we have $\|\GStilde\|_{\Hcal_\infty}\le 1$ where $\GStilde$ is the transfer function of system $\tilde{\Scal}$.
Meanwhile, using the introduced change of variables and based on $\Dscr$, we can define set of data $\Dscr_{\tilde{\Scal}}:=\{(v_t,z_t)\ |\ t\in\Tscr\},$ where $v_t := l_1 u_t$  and $z_t := l_2u_t+l_3y_t$, for $t\in\Tscr$. 
Accordingly, the problem is equivalent to identifying system $\tilde{\Scal}$ given the set of data $\Dscr_{\tilde{\Scal}}$ and the side-information $\|\GStilde\|_{\Hcal_\infty}\le 1$.
%, which is a reformulation of Problem \ref{prob:Hinft_le_1}.
Therefore, in order to address Problem \ref{prob:dissip} it is enough to find a solution approach for Problem \ref{prob:Hinft_le_1}. 
%for the case of $\rho=1$.
\begin{figure}[h]
	\centering
	\includegraphics[width =0.35\textwidth]{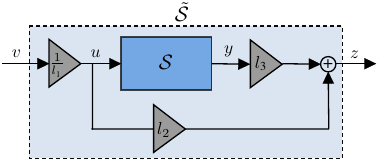}
	\caption{The equivalent system $\tilde{\Scal}$ with $\Hcal_\infty$-norm less than or equal to $1$.}
	\label{fig:Hsys}
\end{figure}
%\begin{remark}
%	If  $\mx{Q}\succeq \zeromx_2$, then we have $s_{\mx{Q}}(u,y)\ge$ independent of the underlying system. Therefore, being dissipative with respect to this supply rate function does not add any information.
%	Also, if $\zeromx_2\succeq \mx{Q}$, then being dissipative with respect to $s_{\mx{Q}}(u,y)$ is either infeasible or leads to the trivial system $G(z)=a/b $ where 
%	$\begin{bmatrix}a& b\end{bmatrix}^\tr$ 
%	is the eigenvector of $\mx{Q}$ corresponding to eigenvalue $0$ (if any).
%	Accordingly, the only reasonable scenario is that $\mx{Q}$ being indefinite, i.e., $\det{Q}<0$.
%\end{remark}

The Problem \ref{prob:Hinft_le_1} can be further extended to the case where system $\Scal$ is approximately known.
More precisely, let $\overline{\Scal}$ be a known system with transfer function $\GSbar$ and we are given that $\|\GS-\GSbar\|_{\Hcal_\infty}\le \rho$, for a known $\rho\in\Rbb_+$. 
This can be the case in various situations, e.g., the system might be identified previously as $\overline{\Scal}$, while it has been changed later due to aging or possible modifications of the plant; or towards a specific application, the model of the system has been estimated as $\overline{\Scal}$ using a particular identification approach, and additionally, we would like to obtain an improved version of the model for some other purpose like prediction.
Based on this discussion, one can propose the following identification problem:
\begin{problem}\label{prob:Hinft_err_eps}
	Given the set of input-output data $\Dscr$, estimate the impulse response of system $\Scal$ satisfying the side-information $\|\GS-\GSbar\|_{\Hcal_\infty}\le \rho$, where $\overline{\Scal}$ is a known system with transfer function $\GSbar$ and $\rho\in\Rbb_+$ is a given scalar. 
\end{problem}
Let $\Delta\Scal$ be the system with the transfer function $\GSDelta:=\GS-\GSbar$ and the impulse response $\gSDelta:=\gS-\gSbar$, where $\gSbar$ is the impulse response of system $\bar{\Scal}$. 
Since $\gSbar$ is known, one can obtain $\Lu{t}(\gSbar)$, and subsequently, define  $d_t $ as $d_t:= y_t-\Lu{t}(\gSbar)$, for any $t\in\Tscr$. 
Due to \eqref{eqn:output_sys_S} and the linearity of $\Lu{t}$, we have
\begin{equation}
	d_t = \Lu{t}(\gS)+w_t-\Lu{t}(\gSbar)  = \Lu{t}(\gSDelta) +w_t, \quad t\in\Tscr.
\end{equation}
Accordingly, we can define $\Dscr_{\Delta\Scal}:=\{(u_t,d_t)\ |\ t\in\Tscr\}$ as the input-output pairs of data for the system $\Delta\Scal$.
Therefore, the Problem \ref{prob:Hinft_err_eps} is equivalent to identifying system $\Delta\Scal$ given the set of data $\Dscr_{\Delta\Scal}$ and the side-information $\|\GSDelta\|_{\Hcal_\infty}\le \rho$, which is in form of Problem \ref{prob:Hinft_le_1}.

Let $\Tbb=\Zbb_+$, $\Tscr:=\{i=0,\ldots,\nD\!-\!1\}$, and consider the following identification problem:
\begin{problem}\label{prob:Hinft_weight_eps}
	Given the set of input-output data $\Dscr$, estimate the impulse response of system $\Scal$ satisfying the frequency domain side-information $\|W\GS\|_{\Hcal_\infty}\le \rho$, where $\rho\in\Rbb_+$ is a given scalar, and weight $W$ is a known stable transfer function with stable casual inverse. 
\end{problem}
Let the output signal $\vc{y}=(y_t)_{t=0}^{\nD-1}$ be filtered by system $W$ and $\vc{p}:=(p_t)_{t=0}^{\nD-1}$ be the resulting filtered signal. 
One can see that $\Dscr_H:=\{(u_t,p_t)|t\in\Tscr\}$ is a set of input-output data for the system with transfer function $H:=W\GS$, where we know that $\|H\|_{\Hcal_\infty}\le \rho$. Let $\hat{H}$ be the solution of Problem~\ref{prob:Hinft_le_1} for this setting. Then, $\hat{G}:=W^{-1}\hat{H}$ is a solution to Problem~\ref{prob:Hinft_weight_eps}.

Considering the above discussion, a solution approach for Problem \ref{prob:Hinft_le_1} leads to addressing Problem~\ref{prob:dissip}, Problem~\ref{prob:Hinft_err_eps}, and Problem~\ref{prob:Hinft_weight_eps}. 
In the remainder of the paper, we discuss solving Problem~\ref{prob:Hinft_le_1}. 
\section{The Estimation Problem: Existence and Uniqueness of the Solution}
\label{sec:MainOpt}
In this section, an optimization problem is introduced to address the estimation Problem \ref{prob:Hinft_le_1}. Additionally, we investigate the existence and uniqueness properties for the solution of this problem.
 
Let $\Fscr\subseteq\Rbb^\Tbb$ be a suitable functional space of stable impulse responses taken as the hypothesis set for the estimation Problem \ref{prob:Hinft_le_1}.
Given bounded signal $\vc{u}\in\Rbb^\Tbb$, with respect to each $t\in\Tbb$, we have the linear map $\Lu{t}:\Fscr\to \Rbb$ as defined in \eqref{eqn:Lut_Z} and \eqref{eqn:Lut_R}.
Based on this definition and the set of data $\Dscr$, one can define the {\em empirical loss} function or the {fitting error} function, $\Lcal_{\Dscr}:\Fscr\to\Rbb$, as the sum of squared error. In other words, for a given candidate impulse response $\vc{g}\in\Fscr$, we have that
\begin{equation}\label{eqn:LD}
\Lcal_{\Dscr}(\vc{g}) := \sumOp_{t\in\Tscr}(\Lu{t}(\vc{g})-y_t)^2.
\end{equation} 
Define $\Gscr\subseteq\Lscr_1$ as the set of impulse responses corresponding to the systems with $\Hcal_\infty$-norm less than or equal to $\rho$. 
More precisely, 
we have
\begin{equation}\label{eqn:set_G_Z}
\!\!\Gscr\!:=\!\bigg\{\vcg\!=\!(g_t)_{t\in\Zbb_+}
\!\in\!\Lscr_1
\bigg| \supOp_{\omega\in[0,\pi]}
\Big|
\sum_{t\in\Zbb_+}g_t\expe^{-\Jimage\omega t}
\Big| \!\le\! \rho
\bigg\},
\end{equation}
and
\begin{equation}\label{eqn:set_G_R}
\Gscr\!:=\!\bigg\{\vcg\!=\!(g_t)_{t\in\Rbb_+}
\!\in\!\Lscr_1
\bigg| \supOp_{\omega\in\Rbb_+}
\Big|
\int_{\Rbb_+}g_t\expe^{-\Jimage\omega t}
\drm t\Big| \!\le\! \rho
\bigg\},
\end{equation}
respectively for $\Tbb=\Zbb_+$ and $\Tbb=\Zbb_+$.
Note that since each element of $\Gscr$ belongs to $\Lscr_1$ the summation in \eqref{eqn:set_G_Z} and the integration in \eqref{eqn:set_G_R} are well-defined.
Accordingly, in order to address Problem \ref{prob:Hinft_le_1}, it is enough to solve the following optimization problem
\begin{equation}\label{eqn:opt_F_0}
\begin{array}{cl}
\minOp_{\vc{g}\in\Fscr}
&\Lcal_{\Dscr}(\vc{g}) + \lambda\Rcal(\vc{g})
\\
\mathrm{s.t.}
&
\vc{g}\in\Gscr, 
\end{array}
\end{equation}
where $\Rcal:\Fscr\to\Rbb_+$ is a suitable regularization function and
$\lambda>0$ is the regularization weight.
Due to the definition of set $\Gscr$ in \eqref{eqn:set_G_Z} and \eqref{eqn:set_G_R}, the constraint in \eqref{eqn:opt_F_0} is essentially equivalent to uncountably infinite number of inequalities. Consequently, the optimization problem \eqref{eqn:opt_F_0} is an infinite-dimensional program with uncountably infinite number of constraints which is not tractable at the current form. 
Therefore, we should address the following questions:
\begin{enumerate}
\item {\it What is a suitable candidate for hypothesis set $\Fscr$?}
\item {\it Does optimization problem \eqref{eqn:opt_F_0} admit an optimal solution? Is this solution unique?}
\item {\it How can we obtain the solution of \eqref{eqn:opt_F_0} or a tight approximation for it?}
\end{enumerate}
In the remainder of this section, we address these questions except the last one which is postponed to Section \ref{sec:Tractable}.
\subsection{Stable Reproducing Kernel Hilbert Spaces}
%-----------------------------------------------
The hypothesis space taken for estimating the unknown impulse response is a type of Hilbert spaces known as  {\em reproducing kernel Hilbert spaces} (RKHS) which
are introduced  briefly below (see \cite{aronszajn1950theory,berlinet2011reproducing} for more details).
The structure of RKHS provides a suitable framework for investigating the problem and obtaining a tractable scheme for solving \eqref{eqn:opt_F_0}. 

\begin{definition}[\cite{berlinet2011reproducing}]
\label{def:RKHS}
Let $\Hscr\subseteq \Rbb^{\Tbb}$ be a Hilbert space endowed with inner product $\inner{\cdot}{\cdot}_{\Hscr}$ and induced norm $\norm{\cdot}_{\Hscr}$. Then, $\Hscr$ is a {\em reproducing kernel Hilbert space} (RKHS) if for any $t\in \Tbb$, 
we have
$\sup\big\{|g_t|\ \big|\ \vcg:=(g_t)_{t\in\Tbb}\in\Hscr,\|\vc{g}\|_{\Hscr}\le 1\big\}<\infty$. 
%\begin{equation}
%\sup\bigg\{|g_t|\ \Big|\ \vcg:=(g_t)_{t\in\Tbb}\in\Hscr,\|\vc{g}\|_{\Hscr}\le 1\bigg\}<\infty.
%\end{equation}
\end{definition}
Along with the RKHS, the notion of Mercer kernel is introduced which is provided in the next definition. 
%-----------------------------------------------
\begin{definition}[\cite{berlinet2011reproducing}]
\label{def:kernel_and_section}
The continuous function $\kernel:\Tbb\times\Tbb\to \Rbb$ is a {\em Mercer kernel} or {\em positive-definite kernel} when for any $m\in\Nbb$, $t,s,t_1,\ldots,t_m\in\Tbb$ and $a_1,\ldots,a_m\in\Rbb$, we have $\kernel(t,s)=\kernel(s,t)$ and $\sum_{1\le i,j \le m} a_ia_j\kernel(t_i,t_j)\ge 0$.
For each $t\in\Tbb$, 
the {\em section} of kernel $\kernel$ at $t$ %, also known as the {\em feature map}, 
is defined as the function $\kernel(t,\cdot):\Tbb\to\Rbb$ and denoted by $\kernel_{t}$.
\end{definition}
The next theorem shows the connection between RKHS introduced in Definition \ref{def:RKHS} and the notion of Mercer kernel defined in Definition \ref{def:kernel_and_section}.
%-----------------------------------------------
\begin{theorem}[\cite{berlinet2011reproducing}]\label{thm:RKHS_mercer}
Given Mercer kernel $\kernel:\Tbb\times\Tbb\to \Rbb$, there exists a RKHS $\Hk\subseteq\vc{g}\in\Rbb^{\Tbb}$, endowed with inner product $\inner{\cdot}{\cdot}_{\Hk}$ and norm $\|\cdot\|_{\Hk}$, such that for any $t\in\Tbb$, we have
$\kernel_k\in\Hk$,  and $\inner{\vc{g}}{ \kernel_{t}}_{\Hk}=g_t$, for all $\vc{g}=(g_t)_{t\in\Tbb}\in\Hk$.
%\begin{itemize}
%	\item[i)] $ \kernel_k\in\Hk$, and
%	\item[ii)] $\inner{\vc{g}}{ \kernel_{k}}_{\Hk}=g_k$, for all $\vc{g}=(g_k)_{k\in\Tbb}\in\Hk$.
%\end{itemize} 
The second feature is called {\em reproducing property}.
\end{theorem}
%-----------------------------------------------
Based on Theorem \ref{thm:RKHS_mercer}, we know that a RKHS is uniquely characterized with a Mercer kernel. 
Therefore, in the current context, the kernel $\kernel$ is supposed to be chosen suitably such that each impulse response $\vc{g}\in\Hk$ represents a stable system in the bounded-input-bounded-output (BIBO) sense. More precisely, we should have $\Hk\subseteq\Lscr^1$. When kernel $\kernel$ satisfies this feature, it is called a {\em stable kernel} \cite{ chen2018stability}. The following theorem provides a necessary and sufficient condition for the stability of a given kernel.
%-----------------------------------------------
\begin{theorem}[\cite{chen2018stability,carmeli2006vector}]
Let $\kernel:\Tbb\times\Tbb\to \Rbb$ be a Mercer kernel. Then, $\kernel$ is {\em stable}  if and only if,  
when $\Tbb=\Zbb_+$, we have
$\sum_{t\in\Zbb_+}\big|\sum_{s\in\Zbb_+}u_s\kernel(t,s)\big|<\infty$, and, when $\Tbb=\Rbb_+$, we have 
$\int_{\Rbb_+}\big|\int_{\Rbb_+}u_s\kernel(t,s)\drm s\big|\drm t<\infty$,
for any $\vc{u}=(u_t)_{t\in\Tbb}\in\Lscr^{\infty}$.

%\begin{itemize}
%\item when $\Tbb=\Zbb_+$, we have
%\begin{equation}
%\sum_{t\in\Zbb_+}\bigg|\sum_{s\in\Zbb_+}u_s\kernel(t,s)\bigg|<\infty,
%\end{equation} 
%\item and, when $\Tbb=\Rbb_+$, we have
%\begin{equation}
%\int_{\Rbb_+}\bigg|\int_{\Rbb_+}u_s\kernel(t,s)\drm s\bigg|\drm t<\infty,
%\end{equation} 
%\end{itemize}
%for any $\vc{u}=(u_t)_{t\in\Tbb}\in\Lscr^{\infty}$.
\end{theorem}
The most common stable kernel in the literature \cite{pillonetto2014kernel} is the \emph{tuned/correlated} (TC) kernel defined as follows
\begin{equation}\label{eqn:TC_kernel}
\kernel(s,t) = 
\begin{cases}
\alpha^{\max(s,t)} & \text{ if $\Tbb=\Zbb_+$,}\\
\expe^{-\beta\max(s,t)} & \text{ if $\Tbb=\Rbb_+$,}
\end{cases}	
\end{equation}
where $\alpha\in[0,1)$ and $\beta>0$. By setting $\alpha$ as $\expe^{-\beta}$, one can obtain same definition for both cases in \eqref{eqn:TC_kernel}.  
This is the default kernel employed in \texttt{impulseest} function of 
\textsc{Matlab}'s \textsc{System Identification Toolbox}. 
One can easily see that if for kernel $\bbh$ is \emph{dominated} by TC kernel $\kernel$, i.e., there exists $\gamma\in\Rbb_+$ such that for any $s,t\in\Tbb$ we have $|\bbh(s,t)|\le \gamma \kernel(s,t)$, then $\bbh$ is a stable kernel. 
This highlights the importance of TC kernels.
%-----------------------------------------------
\subsection{The Estimation Problem in Stable Reproducing Kernel Hilbert Spaces}
\label{ssec:general_opt}
Let $\kernel$ be a stable kernel and $\Hk$ be the corresponding RKHS.
Motivated by \cite{pillonetto2010new,pillonetto2014kernel,ljung2020shift} and the above discussion, we set $\Hk$ as the hypothesis space for the estimation problem, i.e., $\Fscr=\Hk$.
Subsequently, we can introduce a suitable kernel-based regularization. More precisely, let the regularization function $\Rcal:\Hk\to\Rbb_+$ be defined as $\Rcal(\vc{g}) := \|\vc{g}\|_{\Hk}^2$.
Therefore, in correspondence with the estimation Problem  \ref{prob:Hinft_le_1}, we have the following optimization problem:  
\begin{equation}\label{eqn:opt_1}
\begin{array}{cl}
\minOp_{\vc{g}\in\Hk}
&
\sumOp_{t\in\Tscr}(\Lu{t}(\vc{g})-y_t)^2 + \lambda\ \! \|\vc{g}\|_{\Hk}^2,
\\
\mathrm{s.t.}
& 
|G(\Jimage \omega)|\le \rho, \quad  \forall \omega\in\Omega_{\Tbb},
\end{array}
\end{equation}
where $G$ denotes the transfer function corresponding to the impulse response $\vc{g}$, $\Omega_{\Tbb}:=[0,\pi]$ when $\Tbb=\Zbb_+$, and $\Omega_{\Tbb}:=\Rbb_+$ when $\Tbb=\Rbb_+$. 
Note that by abuse of notation, for both cases of discrete-time and continuous-time, we employ same expression $G(\Jimage\omega)$ for the transfer function of the system.  

The constraints introduced in \eqref{eqn:opt_1} are the Fourier transform of the impulse response $\vcg$ at different frequencies. 
Since $\vcg$ belongs to RKHS $\Hk$, one should study the corresponding notion and its properties in the domain of RKHS $\Hk$.
%\begin{definition}\label{def:Fw}
%With respect each $\omega\in \Omega_{\Tbb}$, we define maps $\Fwr,\Fwi:\Hk\to\Rbb$ such that for any $\vc{g}=(g_t)_{t\in\Tbb} \in\Hk$, we have
%\begin{equation}
%\!\!	
%\Fwr(\vc{g})\!:=\!  \sum_{t\in\Zbb_+}\! g_t\cos(\omega t), 
%\ \  
%\Fwi(\vc{g})\!:=\! -\!\!\sum_{t\in\Zbb_+}\! g_t\sin(\omega t),
%\end{equation}
%when $\Tbb=\Zbb_+$, and 
%\begin{equation}
%\!\!
%\Fwr(\vc{g})\!:=\! \int_{\Rbb_+}\!\!\! g_t\cos(\omega t)\drm t, 
%\ \     
%\Fwi(\vc{g})\!:=\! -\!\!\int_{\Rbb_+}\!\!\! g_t\sin(\omega t)\drm t,
%\end{equation}
%when $\Tbb=\Rbb_+$.
%Also, map $\Fw:\Hk\to\Cbb$ is defined as
%$\Fw(\vc{g}) :=  \sum_{t\in\Zbb_+}g_t\ \! \expe^{-\Jimage \omega t}$, when $\Tbb=\Zbb_+$,
%and 
%$\Fw(\vc{g}) :=  \int_{\Rbb_+}g_t\ \!\expe^{-\Jimage \omega t}\drm t$, when $\Tbb=\Rbb_+$,
%for any $\vc{g}=(g_t)_{t\in\Tbb} \in\Hk$.
%\end{definition}
\iftrue
\begin{definition}\label{def:Fw}
With respect each $\omega\in \Omega_{\Tbb}$, we define maps $\Fwr:\Hk\to\Rbb$ and  $\Fwi:\Hk\to\Rbb$ such that for any $\vc{g}=(g_t)_{t\in\Tbb} \in\Hk$, we have
\begin{equation}\!\!\!\!\!\!\!\!
\Fwr(\vc{g}) \!:=\!\sum_{t\in\Zbb_+}
\! g_t \cos(\omega t),  \quad
\Fwi(\vc{g}) \!:=\! -\!\! \sum_{t\in\Zbb_+} 
\!g_t \sin(\omega t),\!\!\!\!
\end{equation}
when $\Tbb=\Zbb_+$, and 
\begin{equation}\label{eqn:def_CT_Fw}\!\!\!\!\!\!
\Fwr(\vc{g})\!:= \! \! \int_{\Rbb_+}\!\!\! g_t\cos(\omega t)\drm t, 
\quad
\Fwi(\vc{g})\!:=\! -\!\!\int_{\Rbb_+}\!\!\! g_t\sin(\omega t)\drm t,\!\!
\end{equation}
when $\Tbb=\Rbb_+$.
Also $\Fw:\Hk\to\Cbb$ is defined as $\Fw(\vc{g})=\Fwr(\vc{g})+\Jimage \Fwi(\vc{g})$,  for any $\vc{g}\in\Hk$.
\end{definition}
\else %-------------------------
\begin{definition}\label{def:Fw_2}
With respect each $\omega\in \Omega_{\Tbb}$, we define maps $\Fwr:\Hk\to\Rbb$, $\Fwi:\Hk\to\Rbb$ and $\Fw:\Hk\to\Cbb$ such that for any $\vc{g}=(g_t)_{t\in\Tbb} \in\Hk$, we have
\begin{equation}
\begin{split}
\Fwr(\vc{g}) &:= \     \sum_{t\in\Zbb_+}
\ \! g_t\ \! \cos(\omega t), \\  
\Fwi(\vc{g}) &:= -\!\! \sum_{t\in\Zbb_+} 
\ \!g_t\ \! \sin(\omega t),\\
\Fw(\vc{g})  &:= \     \sum_{t\in\Zbb_+} 
\ \!g_t\ \! \expe^{-\Jimage \omega t},
\end{split}	
\end{equation}
when $\Tbb=\Zbb_+$, and 
\begin{equation}\label{eqn:def_CT_Fw_2}
\begin{split}
\Fwr(\vc{g}) &:= \ \ \int_{\Rbb_+} g_t\ \!\cos(\omega t)\ \!\drm t, \\
\Fwi(\vc{g}) &:= -\!\!\int_{\Rbb_+} g_t\ \!\sin(\omega t)\ \!\drm t,\\
\Fw(\vc{g})  &:= \ \ \int_{\Rbb_+}g_t\ \!\expe^{-\Jimage \omega t}\ \!\drm t,
\end{split}
\end{equation}
when $\Tbb=\Rbb_+$.
\end{definition}
\fi 
%-----------------------------------------------
Based on the definition of $\Fw$,  for any $\vc{g}\in\Hk$ and $\omega\in\Omega_{\Tbb}$, we have  $G(\Jimage \omega) = \Fw(\vc{g})$, where $G$ is the transfer function corresponding to the impulse response $\vc{g}$.
Accordingly, the problem \eqref{eqn:opt_1} can be re-written in the following form
\begin{equation}\label{eqn:opt_2}
\begin{array}{cl}
\minOp_{\vc{g}\in\Hk}
&
\sumOp_{t\in\Tscr}(\Lu{t}(\vc{g})-y_t)^2 + \lambda\ \! \|\vc{g}\|_{\Hk}^2,
\\
\mathrm{s.t.}
& 
|\Fw(\vc{g})|\le \rho, \quad  \forall \omega\in\Omega_{\Tbb}.
\end{array}
\end{equation}
%------------------------------------------------
\begin{remark}
One can see that $\Fw$ is the Fourier transform restricted on the RKHS $\Hk$ evaluated for the frequency $\omega\in\Omega_{\Tbb}$.
However, we should note that it does not necessarily inherit the same properties of standard Fourier transform. More precisely, the structure of $\Hk$ plays a key role. %Accordingly, we set Assumption \ref{ass:kernel_sqrt_summable} given below. 
\end{remark}
%------------------------------------------------
Before proceeding further, we need to introduce a notion based on the kernel $\kernel$. For any $n\in\Zbb_+$, define $\mu_n\in[0,\infty]$ as
\begin{equation}\label{eqn:mu_n}
\mu_n:=
\begin{cases}
\sumOp_{t\in\Zbb_+}t^n
\ \!\kernel(t,t)^{\frac{1}{2}}, & \text{ if } \Tbb=\Zbb_+,\\
\int_{\Rbb_+}t^n
\ \!\kernel(t,t)^{\frac{1}{2}}\ \!\drm t, & \text{ if } \Tbb=\Rbb_+.
\end{cases}
\end{equation}
The next lemma introduces properties of $\Fwr{}$, $\Fwi{}$ and $\Fw{}$.
%------------------------------------------------
\begin{lemma}\label{lem:Fw}
%Let Assumption \ref{ass:kernel_sqrt_summable} hold. 
Let $\mu_0<\infty$. 
Then, the followings hold:
\\1) The maps $\Fwr,\Fwi:\Hk\to \Rbb$ and $\Fw:\Hk\to \Cbb$ are linear continuous with
$\|\Fwr\|_{\Lcal(\Hk,\Rbb)},\|\Fwi\|_{\Lcal(\Hk,\Rbb)}\ \le\ \mu_0.$
%\begin{equation}\label{eqn:Fwr_Fwi_bounded_mu0}
%\|\Fwr\|_{\Lcal(\Hk,\Rbb)},\ \|\Fwi\|_{\Lcal(\Hk,\Rbb)}\ \le\ \mu_0.
%\end{equation} 
%2) For any $\vc{g}\in\Hk$, we have $\Fw(\vc{g})=\Fwr(\vc{g})+\Jimage \Fwi(\vc{g})$.
\\2) There exist unique elements $\phiwr=(\phiwrk{t})_{t\in\Tbb}$ and $\phiwi=(\phiwik{t})_{t\in\Tbb}$ in $\Hk$ such that, for any $\vc{g}\in\Hk$, we have
\begin{equation}\label{eqn:Fwr_Fwi_inner_product}
\Fwr(\vc{g})=\inner{\phiwr}{\vc{g}}_{\Hk},
\qquad
\Fwi(\vc{g})=\inner{\phiwi}{\vc{g}}_{\Hk},	
\end{equation}
3) For any $t\in\Tbb$, when $\Tbb=\Zbb_+$, we have
%\begin{equation}\label{eqn:phiwr_phiwi_t_TZ}
%\begin{split}
%\phiwrk{t} &= \     \sum_{s\in\Zbb_+} \kernel(t,s)\ \! \cos(\omega s), 
%\\  
%\phiwik{t} &= -\!\! \sum_{s\in\Zbb_+} 
%\kernel(t,s)\ \! \sin(\omega s),\\
%\end{split}	
%\end{equation}
\begin{equation}\label{eqn:phiwr_phiwi_t_TZ}\!\!\!\!\!\!
\phiwrk{t} \!=\!\!\! \sum_{s\in\Zbb_+} \kernel(t,s)\cos(\omega s), 
\ 
\phiwik{t} \!=\! -\!\! \sum_{s\in\Zbb_+} 
\kernel(t,s)\sin(\omega s),\!\!
\end{equation}
and,  when $\Tbb=\Rbb_+$, we have
%\begin{equation}\label{eqn:phiwr_phiwi_t_TR}
%\begin{split}
%\phiwrk{t} &=\ \int_{\Rbb_+}\!
%\kernel(t,s)\ \! \cos(\omega s)\ \!\drm s,
%\\  
%\phiwik{t} &=-\!\!\int_{\Rbb_+}\! 
%\kernel(t,s)\ \!\sin(\omega s)\ \!\drm s,\\
%\end{split}	
%\end{equation}
\begin{equation}\label{eqn:phiwr_phiwi_t_TR}
\!\!\!\!\!\!
\phiwrk{t}\! =\!\!\!\int_{\Rbb_+}\!\!\!\!\!
\kernel(t,s)\cos(\omega s)\drm s,
\  
\phiwik{t}\! =\!-\!\!\int_{\Rbb_+}\!\!\!\!\!
\kernel(t,s)\sin(\omega s)\drm s,\!\!	
\end{equation}
\end{lemma}
\begin{proof}
See Appendix \ref{sec:appendix_proof_lem_Fw}.
\end{proof}
%-----------------------------------------------
Let $\Omega\subseteq\Omega_{\Tbb}$ and $\eta\in\Rbb_+$. 
Define set $\Gscr_{\kernel}(\rho,\Omega)$ as 
\begin{equation}\label{eqn:G_rho_Omega}
\Gscr_{\kernel}(\eta,\Omega) 
:=
\bigg\{\vc{g}\in\Hk\ \bigg|\ |\Fw(\vc{g})|\le \eta, \forall \omega \in \Omega\bigg\}. 
\end{equation}
One can see that the feasible set in \eqref{eqn:opt_2} is an especial case of this set. 
Based on Lemma \ref{lem:Fw}, we study the main properties of $\Gscr_{\kernel}(\eta,\Omega)$ in the next theorem. 
%------------------------------------------------
\begin{theorem}\label{thm:G_convex}
Let $\mu_0<\infty$. Then, $\Gscr_{\kernel}(\eta,\Omega)$ is non-empty, closed and convex.
\end{theorem}
\begin{proof}
From \eqref{eqn:G_rho_Omega}, one can see that
\begin{equation*}
\begin{split}
&\Gscr_{\kernel}(\eta,\Omega) 
=
\{\vc{g}\in\Hk\ |\ |\Fw(\vc{g})|\le \eta, \forall \omega \in \Omega\} 
\\&= \bigcap\limits_{\omega\in\Omega} \  
\{\vc{g}\in\Hk\ |\ |\Fw(\vc{g})|\le \eta\}
%\\&= 
=\bigcap\limits_{\omega\in\Omega} \ 
\Gscr_{\kernel}(\eta,\{\omega\}).
\end{split}
\end{equation*}
Therefore, it is enough to show that, for any $\omega\in\Omega_{\Tbb}$, $\Gscr_{\kernel}(\eta,\{\omega\})$ is a closed and convex subset of $\Hk$. 
Since $|\Fw(\vc{g})|^2 = |\Fwr(\vc{g})|^2 + |\Fwi(\vc{g})|^2$, we have that
\begin{equation}
\Gcal_{\kernel}(\eta,\{\omega\})
=
\bigg\{\vc{g}\in\Hk\ \bigg|\ |\Fwr(\vc{g})|^2 + |\Fwi(\vc{g})|\le \eta^2\bigg\}.
\end{equation}
Due to Lemma \ref{lem:Fw}, we know that  $\Fwr(\vc{g})=\inner{\phiwr}{\vc{g}}_{\Hk}$ and $\Fwi(\vc{g})=\inner{\phiwi}{\vc{g}}_{\Hk}$, for any $\vc{g}\in\Hk$. 
With respect to each $\theta\in[0,\frac{\pi}{2}]$, define set $\Gcal_{\kernel}(\eta,\omega,\theta)$ as
%$\Gcal_{\kernel}(\eta,\omega,\theta)
%\!:=\!
%\big\{\vc{g}\!\in\!\Hk
%\big| 
%|\inner{\phiwr}{\vc{g}}|\! \le \! \eta\cos\theta,
%|\inner{\phiwi}{\vc{g}}|\! \le \! \eta\sin\theta
%\big\}$.
\begin{equation*}
		\Gcal_{\kernel}(\eta,\omega,\theta)
		\!:=\!
		\big\{\!\vc{g}\!\in\!\Hk
		\big| 
		|\inner{\phiwr}{\vc{g}}|\! \le \! \eta\cos\theta,
		|\inner{\phiwi}{\vc{g}}|\! \le \! \eta\sin\theta
		\big\}.
\end{equation*}
Note that $\Gscr_{\kernel}(\eta,\omega,\theta)$
is the intersection of \vspace{-2mm}
%$\{\vc{g}\!\in\!\Hk|\inner{\phiwr}{\vc{g}}_{\Hk}\! \le\! \eta\cos\theta\}$,
%$\{\vc{g}\!\in\!\Hk|\inner{\phiwr}{\vc{g}}_{\Hk}\! \ge\! -\eta\cos\theta\}$,
%$\{\vc{g}\!\in\!\Hk|\inner{\phiwr}{\vc{g}}_{\Hk}\! \le\! \eta\sin\theta\}$, and, 
%$\{\vc{g}\!\in\!\Hk|\inner{\phiwr}{\vc{g}}_{\Hk}\! \ge\! -\eta\sin\theta\}$,
\begin{equation}\vspace{-2mm}
\begin{split}
&\{\vc{g}\!\in\!\Hk|\inner{\phiwr}{\vc{g}}_{\Hk}\! \le\! \eta\cos\theta\},
\\&\{\vc{g}\!\in\!\Hk|\inner{\phiwr}{\vc{g}}_{\Hk}\! \ge\! -\eta\cos\theta\},
\\&\{\vc{g}\!\in\!\Hk|\inner{\phiwi}{\vc{g}}_{\Hk}\! \le\! \eta\sin\theta\},
\\&\{\vc{g}\!\in\!\Hk|\inner{\phiwi}{\vc{g}}_{\Hk}\! \ge\! -\eta\sin\theta\},
\end{split}
\end{equation}
which are closed half-spaces in $\Hk$. 
Therefore, $\Gcal_{\kernel}(\eta,\omega,\theta)$ is a closed and convex set. 
Since we have
$ %\begin{equation}
	\Gcal_{\kernel}(\eta,\{\omega\})
	=
%	\bigcap\limits_{\theta\in[0,\frac{\pi}{2}]}
	\bigcap_{\theta\in[0,\frac{\pi}{2}]}
	\Gcal_{\kernel}(\eta,\omega,\theta),
$ %\end{equation}
the set $\Gcal_{\kernel}(\eta,\{\omega\})$ is closed and convex as well. 
Note that for $\zero\in\Hk$ and any $\omega\in\Omega$, we have $\Fw(\zero)=0$. 
Therefore, $|\Fw(\zero)|=0\le \eta$ and subsequently, $\zero\in\Gscr(\eta,\Omega)$.
This shows that $\Gscr(\eta,\Omega)$ is a non-empty set and concludes the proof.
\end{proof}
%------------------------------------------------
Before proceeding to the main theorem of this section, we need to present an auxiliary lemma.
%------------------------------------------------
\begin{lemma}\label{lem:Lui_bounded}
Let $\mu_0<\infty$ and $\vc{u}\in\Lscr_{\infty}$. 
Then, for any $t\in\Tbb$, the map $\Lu{t}:\Hk\to\Rbb$, defined in \eqref{eqn:Lut_Z} and  \eqref{eqn:Lut_R}, is linear and continuous with 
$\|\Lu{t}\|_{\Lcal(\Hk,\Rbb)}\le\mu_0 \|\vc{u}\|_{\infty}$.
%\begin{equation}
%\|\Lu{t}\|_{\Lcal(\Hk,\Rbb)}\ \le\ \mu_0\ \! \|\vc{u}\|_{\infty}.     
%\end{equation}
Also, there exists unique $\phiu{t}:=(\phiu{t,s})_{s\in\Tbb}\in \Hk$ such that $\Lu{t}(\vc{g}) =\inner{\phiu{t}}{\vc{g}}_{\Hk}$, for any $\vc{g}\in\Hk$.
Moreover, for any $s\in\Tbb$, we have
\begin{equation}\label{eqn:phi_u_ts}
\phiu{t,s} = 
%\Lu{i}(\Kernel_k) = 
\begin{cases}
\sum_{\tau\in\Zbb_+}\kernel(s,\tau) u_{t-\tau},
& \text{ if } \Tbb=\Zbb_+,\\
\int_{\Rbb_+}\kernel(s,\tau) u_{t-\tau}\drm \tau,
& \text{ if } \Tbb=\Rbb_+.\\
\end{cases}
\end{equation}
\end{lemma}
\begin{proof}
	See Appendix \ref{sec:appendix_proof_lem_Lui_bounded}.
\end{proof}
%------------------------------------------------
\begin{theorem}\label{thm:opt_unique_convex}
Let $\mu_0>0$ and consider the following program 
\begin{equation}\label{eqn:opt_3}
\minOp_{\vc{g}\in\Gscr_{\kernel}(\eta,\Omega)}\ 
\sumOp_{t\in\Tscr}(\Lu{t}(\vc{g})-y_t)^2 + \lambda\ \! \|\vc{g}\|_{\Hk}^2.
\end{equation}
Then, \eqref{eqn:opt_3} is a convex optimization problem with a \emph{unique} solution $\gstar_{\Omega}$.
Moreover, we have 
\begin{equation}\label{eqn:gstar_Omega_bound}
\|\gstar_{\Omega}\|_{\Hk}\ \! \le\ \! \big(\frac{1}{\lambda}\sum_{t\in\Tscr}y_t^2\big)^{\frac12}.	
\end{equation}
\end{theorem}
%------------------------------------------------
\begin{proof} 
Define $\Jcal:\Hk\to\Rbb\cup\{+\infty\}$ such that for any $\vc{g}\in\Hk$ we have
\begin{equation}\label{eqn:J}
\Jcal(\vc{g}) 
= 
\sum_{t\in\Tscr}(\Lu{t}(\vc{g})-y_t)^2 + \lambda \|g\|^2_{\Hk} + \delta_{\Gscr_{\kernel}(\eta,\Omega)}(\vc{g}).
\end{equation}
Since $\zero \in \Gscr_{\kernel}(\eta,\Omega)$,  we have $\delta_{\Gscr_{\kernel}(\eta,\Omega)}(\zero)=0$.
Due to the definition of $\Lu{t}$ in \eqref{eqn:Lut_Z} and \eqref{eqn:Lut_R}, one has that $\Lu{t}(\zero)=0$, for each $t\in\Tscr$, and subsequently, we have
$ %\begin{equation}
\Jcal(\zero)=\sum_{t\in\Tscr}y_t^2<\infty.
$ %\end{equation}
From Theorem \ref{thm:G_convex}, we know that $\Gscr_{\kernel}(\eta,\Omega)$ is a convex and closed set, and consequently, $\delta_{\Gscr_{\kernel}(\eta,\Omega)}$ is a proper lower semi-continuous convex function \cite{peypouquet2015convex}. 
Due to Lemma \ref{lem:Lui_bounded}, we know that $\Lu{t}:\Hk\to\Rbb$ is a continuous linear map, for each $t\in\Tscr$. Therefore, function  $\Lcal_{\Dscr}:\Hk\to \Rbb$, defined in \eqref{eqn:LD}, is a convex and continuous function.
Since $\lambda>0$, we know that $\Jcal$ is a proper and lower semi-continuous strongly convex function.
Therefore, $\min_{g\in\Hk}\Jcal(g)$ has a unique (finite) solution \cite{peypouquet2015convex}, and subsequently, \eqref{eqn:opt_3} is a convex program with a unique solution $\gstar_{\Omega}$ with finite cost.
Since $\zero \in \Gscr_{\kernel}(\eta,\Omega)$ and due to optimality of $\gstar_{\Omega}$, we have $\Jcal(\gstar_{\Omega})\le \Jcal(\zero)$. 
Subsequently, one can see
\begin{equation*}
\lambda\|\gstar_{\Omega}\|^2 
\le
\sumOp_{t\in\Tscr}(\Lu{t}(\gstar_{\Omega})-y_t)^2 + \lambda\ \! \|\gstar_{\Omega}\|^2 
\le 
\Jcal(\zero) = 
\sum_{t\in\Tscr}y_t^2,	
\end{equation*}
which induces \eqref{eqn:gstar_Omega_bound}.
This concludes the proof.
\end{proof}
%------------------------------------------------
\begin{corollary}\label{cor:opt_2_unique_convex}
Let assume $\mu_0<\infty$. Then, \eqref{eqn:opt_2} is a convex optimizations with a unique solution denoted by $\gstar$. Moreover, $\gstar$ satisfies inequality \eqref{eqn:gstar_Omega_bound}. Similar property holds for \eqref{eqn:opt_1}.
\end{corollary}
\begin{proof}
In Theorem \ref{thm:opt_unique_convex}, set $\Omega$ and $\eta$  respectively to $\Omega_{\Tbb}$ and $\rho$. 
Then, the convexity of \eqref{eqn:opt_2} as well as the existence and uniqueness of its solution is directly concluded.
Since \eqref{eqn:opt_1} and \eqref{eqn:opt_2} are equivalent,
the same claim holds for \eqref{eqn:opt_1}.
\end{proof}
%------------------------------------------------
In the  above discussion, we have assumed that $\mu_0<\infty$. The next theorem shows that the boundedness of $\mu_0$ is a valid assumption for most of the stable kernels introduced in the literature, especially it holds for TC kernel. 
%------------------------------------------------
\begin{theorem}\label{lem:bound_mu_n}
	Consider kernel $\kernel$ and assume there exist $\beta,\gamma> 0$ such that $|\kernel(s,t)|\le \gamma \expe^{-\beta\max(s,t)}$, for any $s,t\in\Tbb$. Then, for any $n\in\Zbb_+$, we have $\mu_n\le	\gamma^{\frac12}(\frac{2}{\beta})^{n+1}n!$, 
%	\begin{equation}
%		\mu_n\le
%%		\gamma^{\frac12}\frac{2^{n+1} }{\beta^{n+1}}n!,
%		\gamma^{\frac12}(\frac{2}{\beta})^{n+1}n!,
%	\end{equation} 
	where $n!:=\prod_{k=1}^nk$, when $n\ge 1$, and, $0!:=1$.
\end{theorem}
\begin{proof}
	See Appendix \ref{sec:appendix_proof_bound_mu_n}.
\end{proof}
%------------------------------------------------

Based on Corollary \ref{cor:opt_2_unique_convex}, we know that the optimization problem \eqref{eqn:opt_1} has a unique solution addressing Problem \ref{prob:Hinft_le_1}. 
Since this optimization problem is defined over an infinite-dimensional space with uncountably infinite number of constraints, obtaining this solution does not seem to be tractable at the current form. 
This is discussed in the next section.

\section{Towards a Tractable Scheme}\label{sec:Tractable}
In this section, we present a tractable approach to derive the solution of the nonparametric estimation problem introduced in Section \ref{sec:MainOpt}.

Consider optimization problem \eqref{eqn:opt_2}. For the ease of discussion and without loss of generality, we assume $\rho=1$.
More precisely, one can use a change of variable and replace $y_t$ with %$\frac{1}{\rho}y_t$
$\rho^{-1}y_t$ in $\Dscr$ and then, set $\rho=1$ in Problem \ref{prob:Hinft_le_1} as well as in \eqref{eqn:opt_2}, or equivalently in \eqref{eqn:opt_1}.  
With respect to each $\omega\in\Omega_{\Tbb}$, there exists a constraint in \eqref{eqn:opt_2}.
Thus, we have uncountably infinite number of constraints which makes the problem intractable. 
One possible approach to resolve this issue is approximating the problem by considering a only a suitable finite subset of $\Omega_{\Tbb}$. 
In order to investigate this possibility, we need the notion of {\em partition} introduced in the next definition.
\begin{definition}
We say $\Pscr$ is a {\em partition} of interval $[a,b]$ if $\Pscr$ is a finite subset of $[a,b]$ as $\Pscr=\{\omega_i\ |\ i=0,\ldots,\nP\}$ where 
\begin{equation}
    a = \omega_0<\omega_1<\ldots<\omega_{\nP}=b.
\end{equation}
With respect to partition $\Pscr$, the {\em mesh} of $\Pscr$, denoted by $\mesh(\Pscr)$, is defined as \begin{equation}
    \mesh(\Pscr):=\max\{|\omega_{i}-\omega_{i-1}|\ | \ i=1,2,\ldots,\nP\}.
\end{equation}
\end{definition}
Now, let $\Pscr=\{\omega_i\ |\ i=0,\ldots,\nP\}$ be a given partition. 
One can see that satisfying the constraints $|\Fw(\vc{g})|\le 1$, for all $w\in \Pscr$, does not necessarily imply that the desired feature $\sup_{\omega\in \Omega_{\Tbb}}|\Fw(\vc{g})|\le 1$.
More precisely, these constraints do not guarantee that for all $\omega\in \Omega_{\Tbb}\backslash\Pscr$ we have $|\Fw(\vc{g})|\le 1$ as well.
Accordingly, in order to approximate problem \eqref{eqn:opt_2},
we take $\epsilon> 0$ and consider the following problem as our approximation \begin{equation}\label{eqn:opt_4}
\begin{array}{cl}
    \minOp_{\vc{g}\in\Hk}
    &
    \sumOp_{t\in\Tscr}(\Lu{t}(\vc{g})-y_t)^2 + \lambda\ \! \|\vc{g}\|_{\Hk}^2,
    \\
    \mathrm{s.t.}
    & 
    |\Fw(\vc{g})|^2\le 1-\epsilon, \quad  \forall \omega\in\Pscr.
\end{array}
\end{equation}
In the following, we provide appropriate conditions on partition $\Pscr$ and $\epsilon>0$ 
for ensuring that the solution of \eqref{eqn:opt_4} satisfies $\sup_{\omega\in \Omega_{\Tbb}}|\Fw(\vc{g})|\le 1$. 
These conditions depend on the bandwidth of system $\vcg$ and the rate of changes for $|\Fw(\vcg)|$ with respect to $\omega$.
Accordingly, first we need to introduce necessary definitions and preliminaries.

%------------------------------------------------
\begin{definition}
Let $\vc{g}\in\Hk$ be a given impulse response. With respect to $\vc{g}$, we define {\em magnitude function} $\mg:\Omega_{\Tbb}\to\Rbb$ as $\mg(\omega)=|\Fw(\vc{g})|^2$. 
\end{definition}
%------------------------------------------------
Given impulse response $\vcg$, the function $\mg$ shows the squared magnitude of transfer function corresponding to $\vcg$ at different frequencies. The following lemma introduces a bound for the rate of changes of $\mg$ in terms of its Lipschitz constant. This will be used later in the analysis of problem \eqref{eqn:opt_4}. 
%------------------------------------------------
\begin{lemma}\label{lem:gm_Lipschitz}
Let $\mu_1,\mu_2<\infty$. Then, for any $\vc{g}\in\Hk$ and $\omega_1,\omega_2\in\Omega_{\Tbb}$, we have
\begin{equation}
    |\mg(\omega_2)-\mg(\omega_1)|\le L_{\vc{g}}|\omega_2-\omega_1|,
\end{equation}
where $L_{\vc{g}} := 4\mu_1\mu_2\|\vc{g}\|^2$.
\end{lemma}
\begin{proof}
See Appendix \ref{sec:appendix_proof_gm_Lipschitz}.
\end{proof}
%------------------------------------------------
For the case of discrete-time systems, i.e., $\Tbb=\Zbb_+$, the frequency range $\Omega_{\Tbb}$ is the bounded interval  $[0,\wmax]$, where $\wmax=\pi$.
In order to introduce analogous of $\wmax$ for the continuous-time systems, we need the next assumption. 

%------------------------------------------------
\begin{assumption}
\label{ass:lim_int_kst_cos_wst}
For the case $\Tbb=\Rbb_+$, we have
\begin{equation}\label{eqn:ass_lim_int_kst_cos_wst}
\lim_{\omega\to\infty}
\int_{\Rbb_+}\int_{\Rbb_+}
\kernel(s,t)\expe^{-\Jimage\omega(s-t)}\drm s\drm t = 0.
\end{equation}
\end{assumption}
%------------------------------------------------
This assumption says that the bandwidth of kernel $\kernel$ is bounded. Moreover, from \eqref{eqn:ass_lim_int_kst_cos_wst}, we know that there exists $\wmax\in\Rbb_+$ such that
\begin{equation}
\label{eqn:int_kst_cos_wst_le_1}
\bigg|\int_{\Rbb_+}\int_{\Rbb_+}
\kernel(s,t)\expe^{-\Jimage\omega(s-t)}\drm s\drm t\bigg| \le 
\frac{\lambda}{\sum_{t\in\Tscr}y_t^2},
\end{equation}
for all $\omega>\wmax$.
In the followings, it is shown that the frequency value $\wmax$ plays the role of an upper bound for the bandwidth of the system of interest.
The next theorem shows that for TC kernels Assumption \ref{ass:lim_int_kst_cos_wst} holds. One may show similar result for other stable kernels in the literature.
%------------------------------------------------
\begin{theorem}\label{thm:TC_assumption}
Let $\beta>0$ and $\kernel:\Rbb_+\times\Rbb_+\to\Rbb_+$ be the TC kernel $\kernel(s,t)=\expe^{-\beta\max(s,t)}$. Then, we have
\begin{equation}
\int_{\Rbb_+}\int_{\Rbb_+}
\kernel(s,t)\expe^{-\Jimage\omega(s-t)}\drm s\drm t = \frac{2}{\omega^2+\beta^2}.
\end{equation} 	
\end{theorem}
\begin{proof}
See Appendix \ref{sec:appendix_proof_TC_assumption}.
\end{proof}
%------------------------------------------------
Based on the introduced notions and facts, we can present the main theorem of this section.

%------------------------------------------------
\begin{theorem}\label{thm:opt_unique_convex_eps}
Let $\mu_0,\mu_1<\infty$, $\Pscr$ be a given partition of $[0,\wmax]$ and $\epsilon>0$. Then, optimization problem \eqref{eqn:opt_4} is a convex program with a unique solution  $\gstareps$. For $\gstareps$, we have 
\begin{equation}\label{eqn:gepsstar_Omega_bound}
\|\gstareps\|_{\Hk}\ \! \le\ \! \Big(\frac{1}{\lambda}\sum_{t\in\Tscr}y_t^2\Big)^{\frac12}.	
\end{equation}
Moreover, if $\mesh(\Pscr)\le
\frac{2\epsilon}{L}$ where 
$L:=\frac{1}{\lambda} 4\mu_0\mu_1\sum_{t\in\Tscr}y_i^2$, 
then we have
\begin{equation}\label{eqn:Gepsstar_le_1}
\|\Gstareps\|_{\Hcal_\infty} = \sup_{\omega\in\Omega_{\Tbb}}|\Fw(\gstareps)|\le 1.	
\end{equation}
where $\Gstareps$ is the transfer function corresponding to $\gstareps$. 
\end{theorem}
%------------------------------------------------
\begin{proof}
Due to the definition of set $\Gscr_{\kernel}$ in \eqref{eqn:G_rho_Omega}, we know that the feasible set of optimization problem \eqref{eqn:opt_4} is $\Gscr_{\kernel}((1-\epsilon)^{\frac{1}{2}},\Pscr)$. 
Subsequently, problem \eqref{eqn:opt_4} can be written as
\begin{equation}\label{eqn:opt_5}
\minOp_{\vc{g}\in\Gcal_{\kernel}((1-\epsilon)^{\frac{1}{2}},\Pscr)}\ \
\sumOp_{t\in\Tscr}(\Lu{i}(\vc{g})-y_i)^2 + \lambda \|\vc{g}\|_{\Hk}^2.
\end{equation}
Then, according to Theorem \ref{thm:opt_unique_convex}, the optimization problem \eqref{eqn:opt_5} as well as \eqref{eqn:opt_4}, is a convex program with unique solution, denoted by $\gstareps$, which  satisfies \eqref{eqn:gepsstar_Omega_bound}. 
%Moreover, we know that 
%$\vc{g}= \zero\in\Gcal_{\kernel}((1-\epsilon)^{\frac{1}{2}},\Pscr)$. Accordingly, by comparing the cost of \eqref{eqn:opt_5}, at $\vc{g}=\zero$ and $\vc{g}=\geps$, we have that
%\begin{equation}
%\sumOp_{i=0}^{\nD-1}(\Lu{i}(\geps)-y_i)^2 + \lambda \|\geps\|_{\Hk}^2
%\le\sumOp_{i=0}^{\nD-1}y_i^2,
%\end{equation}
%and subsequently, it follows that $ \|\geps\|_{\Hk}^2 \le \frac{1}{\lambda}\sum_{i=0}^{\nD-1}y_i^2$.

Let $\omega\in[0,\wmax]$. 
If $\omega\in\Pscr$, then $|\Fw(\gstareps)|\le(1-\epsilon)^{\frac{1}{2}}\le 1$. 
If $\omega\notin\Pscr$, then there exists $i\in\{1,\ldots,n\}$, such that $\omega\in(\omega_{i-1},\omega_i)$. 
If $|\Fw(\gstareps)|>1$, then, due to Lemma \ref{lem:gm_Lipschitz} and $|\Fcal_{\omega_{i-1}}(\gstareps)|^2,|\Fcal_{\omega_i}(\gstareps)|^2\le 1-\epsilon $, we know that
\begin{equation}
\begin{split}
\epsilon&<
\Big|\ |\Fcal_{\omega_i}(\gstareps)|^2-|\Fcal_{\omega}(\gstareps)|^2\Big| 
\\&\ \  = 
|\mgstareps(\omega_i)-\mgstareps(\omega)|
\!\le\! 
L_{\gstareps}|\omega_i-\omega|,\\
%\end{split}
%\end{equation}
%and
%\begin{equation}
%\begin{split}
\epsilon&<
\Big|\ |\Fcal_{\omega}(\gstareps)|^2-|\Fcal_{\omega_{i-1}}(\gstareps)|^2\Big| 
\\&\ \ =
|\mgstareps(\omega)-\mgstareps(\omega_{i-1})|
\!\le\! L_{\gstareps}|\omega-\omega_{i-1}|,
\end{split}
\end{equation}
where $L_{\gstareps} = 4\mu_0\mu_1\|\gstareps\|^2$.
Therefore, since $|\omega_i-\omega_{i-1}|\le \mesh(\Pscr)$, we have that $2\epsilon< L_{\gstareps}\mesh(\Pscr)$. 
Subsequently, it follows that
\begin{equation*}
\!	
\mesh(\Pscr)\!>\!
\frac{2\epsilon}{L_{\geps}} 
\!= \!
\frac{2\epsilon}{4\mu_0\mu_1\|\geps\|^2}
%\\&
\!\ge\! 
\frac{2\epsilon\lambda}{4\mu_0\mu_1\!\sum_{t\in\Tscr}\!y_t^2}
\!=\!
\frac{2\epsilon}{L},
\!
\end{equation*}
which  contradicts with $\mesh(\Pscr)\le
\frac{2\epsilon}{L}$. 
Therefore, we have $|\Fw(\geps)|\le 1$. This shows that $\|\Gstareps\|_{\Hcal_\infty} = \sup_{\omega\in[0,\wmax]}|\Fw(\gstareps)|\le 1$, which concludes the proof for the case $\Tbb=\Zbb_+$ where we have $\wmax=\pi$.
Now, we consider the case $\Tbb=\Rbb_+$ and let $\omega>\wmax$. 
Due to Lemma \ref{lem:Fw} and the Cauchy-Schwartz  inequality, we know that
\begin{equation}\label{eqn:Fw_le_g_phiwr_phiwi}
\begin{split}
|\Fw(\gstareps)|^2 
&= 
%|\inner{\gstareps}{\phiwr}_{\Hk}|^2
%+
%|\inner{\gstareps}{\phiwi}_{\Hk}|^2
|\inner{\gstareps}{\phiwr}|^2
+
|\inner{\gstareps}{\phiwi}|^2
\\&\le 
%\|\gstareps\|_{\Hk}^2\ 
%\Big[\|\phiwr\|_{\Hk}^2 + \|\phiwi\|_{\Hk}^2\Big].
\|\gstareps\|^2\ 
\Big[\|\phiwr\|^2 + \|\phiwi\|^2\Big].
\end{split}	
\end{equation}
On other hand, from \eqref{eqn:Fwr_Fwi_inner_product}, \eqref{eqn:def_CT_Fw}, %
     %\eqref{eqn:phiwr_phiwi_t_TR}
and
\eqref{eqn:int_kst_cos_wst_le_1}, we have 
\begin{equation}
\begin{split}
%\|\phiwr&\|_{\Hk}^2 + \|\phiwi\|_{\Hk}^2
%=
%\inner{\phiwr}{\phiwr}_{\Hk}^2
%+ 
%\inner{\phiwi}{\phiwi}_{\Hk}^2
%\\&=
\|\phiwr&\|^2 + \|\phiwi\|^2
=
\inner{\phiwr}{\phiwr}^2
+ 
\inner{\phiwi}{\phiwi}^2
\\&=
\int_{\Rbb_+}\!\!\!\phiwrk{t}\cos(\omega t)
\drm t +
\int_{\Rbb_+}\!\!\!\phiwik{t}\sin(\omega t)
\drm t
\\&=
\int_{\Rbb_+}\!
\int_{\Rbb_+}\!
\kernel(t,s)\cos(\omega s)\cos(\omega t)
\drm s\drm t 
\\&\qquad +
\int_{\Rbb_+}\!
\int_{\Rbb_+}\!
\kernel(t,s)\sin(\omega s)\sin(\omega t)
\drm s\drm t
\\&=
\int_{\Rbb_+}\!
\int_{\Rbb_+}\!
\kernel(t,s)\expe^{-\Jimage\omega (s-t)}
\drm s\drm t 
\le \frac{\lambda}{\sum_{t\in\Tscr}y_t^2},
\end{split}
\end{equation}
where the last equality is due to the fact that  $\int_{\Rbb_+}\!
\int_{\Rbb_+}\!
\kernel(t,s)\sin(\omega (s-t))
\drm s\drm t=0$ which is a result of $\kernel(t,s)=\kernel(s,t)$ and $\sin(\omega (s-t))= -\sin(\omega (t-s))$. 
Accordingly, due to 
\eqref{eqn:Fw_le_g_phiwr_phiwi} and \eqref{eqn:gepsstar_Omega_bound}, we have  
$|\Fw(\gstareps)|^2\le 1$. This shows that 
$\|\Gstareps\|_{\Hcal_\infty} = \sup_{\omega\in\Omega_{\Tbb}}|\Fw(\gstareps)|\le 1$ and concludes the proof.
\end{proof}
%------------------------------------------------
Due to Theorem \ref{thm:opt_unique_convex_eps}, we know that  \eqref{eqn:opt_4} admits a unique solution. 
While this optimization problem is defined over the infinite-dimensional Hilbert space $\Hk$, we can find its unique  solution $\gstareps$ by solving an equivalent convex finite-dimensional program.
This feature which makes \eqref{eqn:opt_4} a tractable problem, is due to the structure of the RKHS $\Hk$ and the {\em Representer Theorem}, 
%where a version of this theorem is 
provided below.
\begin{theorem}[Representer Theorem, \cite{scholkopf2001generalized, dinuzzo2012representer}]\label{thm:rep_thm}
Let $e:\Rbb^m\to \Rbb\cup\{+\infty\}$ and $r:\Rbb_+\to \Rbb$ be functions such that $r$ is an increasing function. Also, let $\Hcal$ be a Hilbert space with inner product $\inner{\cdot}{\cdot}_{\Hcal}$. Consider the optimization problem 
\begin{equation}
\label{eqn:opt_representer_thm}
\min_{\vc{w}\in\Hcal} \    
e(\inner{\vc{w}_1}{\vc{w}}_{\Hcal}, \ldots,
\inner{\vc{w}_m}{\vc{w}}_{\Hcal})+ r(\|\vc{w}\|_{\Hcal}),
\end{equation}
where  $\vc{w}_1,\ldots,\vc{w}_m\in\Hcal$ are given vectors. Then, if \eqref{eqn:opt_representer_thm} admits a solution, it has also a solution in $\Wscr:=\linspan\{\vc{w}_i\}_{i=1}^m$.
\end{theorem}

In order to present the tractable finite-dimensional optimization problem equivalent to \eqref{eqn:opt_4}, additional definitions are required. Define  $m :=\nD+2\nP+2$ and the index sets
$\Iu := \{0,1,\ldots,\nD-1\}$,
$\IP := \{0,1,\ldots,\nP\}$, 
%$\Ir := \{\nD+2k\ |\  k\in \IP\}$,
%$\Ii := \{\nD+2k+1\ |\  k\in \IP\}$,
% and $\Iall := \Iu\cup\Ir\cup\Ii$.
and $\Iall := \{0,1,\ldots,m-1\}$.
Let 
%$\varphi_0,\ldots,\varphi_{m-1}$ 
$\{\varphi_i\}_{i=0}^{m-1}$ 
be vectors defined as
\begin{equation}
\begin{cases}
\varphi_i:=\phiu{t_i}, 
& \text{ for }
i\in\Iu,\\
\varphi_{\nD+2j}:=\phir{\omega_j}, 
& \text{ for }
j\in\IP,\\
\varphi_{\nD+2j+1}:=\phii{\omega_j}, 
& \text{ for }
j\in\IP.
\end{cases}
\end{equation}
Let $\Phi\in\Rbb^{m\times m}$ be a symmetric matrix such that its entry at 
the $i^{\text{\tiny{th}}}$ row and 
the $j^{\text{\tiny{th}}}$ column is $\inner{\varphi_{i-1}}{\varphi_{j-1}}_{\Hk}$, for $i,j=1,\ldots,m$.
One can see that $\Phi$ is the Gram matrix of vectors $\varphi_0,\ldots,\varphi_{m-1}$.
\iftrue%-------------------------
Moreover, for $i\in\Iu$, we define vector $\vca_i\in\Rbb^m$ as
the $(i+1)^{\text{\tiny{th}}}$ column of $\Phi$.
Similarly, vectors $\vc{b}_j\in\Rbb^m$ and $\vc{c}_j\in\Rbb^m$ are  defined respectively as the 
$(\nD+2j+1)^{\text{\tiny{th}}}$ and the 
$(\nD+2j+2)^{\text{\tiny{th}}}$ column of $\Phi$, for $j\in\IP$.
\else %-------------------------
Moreover, for $i\in\Iu$, we define vector $\vca_i$ as
\begin{equation}\label{eqn:a_i}
\vc{a}_i:=
%\begin{bmatrix}
%\inner{\varphi_{i}}{\varphi_{0}}_{\Hk}&
%\ldots&
%\inner{\varphi_{i}}{\varphi_{m-1}}_{\Hk}
%\end{bmatrix}^\tr,
\Big[\inner{\varphi_{k}}{\varphi_{i}}_{\Hk}\Big]_{k=0}^{m-1}\in\Rbb^m,
\end{equation}
which is the $(i+1)^{\text{\tiny{th}}}$ column of $\Phi$.
Similarly, vectors $\vc{b}_j\in\Rbb^m$ and $\vc{c}_j\in\Rbb^m$ are  defined respectively as the 
$(\nD+2j+1)^{\text{\tiny{th}}}$ and the 
$(\nD+2j+2)^{\text{\tiny{th}}}$ column of $\Phi$, for $j\in\IP$,
i.e., 
we have
\begin{equation}\label{eqn:b_j}
\vcb_j :=
\Big[\inner{\varphi_{i}}{\varphi_{\nD+2j}}
\Big]_{i=0}^{m-1}\in\Rbb^m,
\end{equation} 
and 
\begin{equation}\label{eqn:c_j}
\vcc_j :=
\Big[
\inner{\varphi_{i}}{\varphi_{\nD+2j+1}}
\Big]_{i=0}^{m-1}\in\Rbb^m.
\end{equation}
\fi %-------------------------
In the followings, without loss of generality, we assume that $\varphi_0,\ldots,\varphi_{m-1}$ are linearly independent. Indeed, if for some $i\in\Iall$, the vector  $\varphi_i$ belongs to $\linspan\{\varphi_{j}|j\in\Iall
\backslash\{i\}\}$, it does not have any additional information and one can replace it with a linear combination of   $\{\varphi_{j}|j\in\Iall
\backslash\{i\}\}$. 

Based on the introduced notations, we can present the theorem on the equivalent tractable finite-dimensional program. 
%-------------------------------
\begin{theorem}\label{thm:opt_unique_convex_eps_finite}
The unique solution of optimization problem \eqref{eqn:opt_4}, $\gstareps$, is in the following linear form
\begin{equation}
\label{eqn:gstareps_lin_xphi}
\gstareps=\sum_{i=0}^{m-1}x_i\varphi_i,	
\end{equation}
where $\vc{x}:=
\begin{bmatrix}x_0&x_1&\ldots&x_{m-1}\end{bmatrix}^\tr\in\Rbb^m$ is the unique solution of the following convex program
\begin{equation}\label{eqn:opt_6_finite}
\begin{array}{cl}
\minOp_{\vc{x}\in\Rbb^m}
&
\sum_{i\in\Iu}(\vc{a}_i^\tr\vc{x}-y_i)^2 + \lambda \vc{x}^\tr\Phi\vc{x}
\\
\mathrm{s.t.}
& 
(\vc{b}_j^\tr\vc{x})^2
+
(\vc{c}_j^\tr\vc{x})^2
\le 1-\epsilon, \quad  \forall j\in\IP.
\end{array}
\end{equation}
\end{theorem}
%-------------------------------
\begin{proof}
Define $\Ecal:\Hk\to \Rbb\cup\{+\infty\}$ as
\begin{equation}\label{eqn:E}
\Ecal(\vc{g}) := \sum_{t\in\Tscr}(\Lu{t}(\vcg)-y_t)^2 
+ 
\sum_{\omega\in\Pscr}\delta_{\{|\Fw|^2\le 1-\epsilon\}}(\vcg).
\end{equation}
Since $\Pscr=\{\omega_0,\ldots,\omega_{\nP}\}$ is a finite set, the summation in \eqref{eqn:E} is well-defined.
We know that optimization problem \eqref{eqn:opt_4}
is equivalent to $\min_{\vc{g}\in\Hk} \Ecal(\vc{g})+\lambda \Rcal(\vc{g})$,
where $\Rcal:\Hk\to\Rbb$ is defined as $\Rcal(\vc{g})=\|\vc{g}\|^2$ (see Section \ref{ssec:general_opt}). 
Let function 
$e:\Rbb^m\to \Rbb\cup\{+\infty\}$ be defined such that for any $\vc{z}=(z_i)_{i=0}^{m-1}\in\Rbb^m$ we have
\begin{equation}\label{eqn:e}
e(\vc{z}) := \sum_{i\in\Iu}(z_i-y_{t_i})^2 + \sum_{j\in\IP}\delta_{\Ascr_j}(\vc{z}),
\end{equation}
where, for each $j\in\IP$, $\Ascr_j\subseteq \Rbb^m$ is the following set 
\begin{equation}
\Ascr_j:=
\bigg\{
%\vc{z}=
(z_i)_{i=0}^{m-1}\in\Rbb^n
\bigg|
z_{\nD+2j}^2+z_{\nD+2j+1}^2\le 1-\epsilon
\bigg\},	
\end{equation}
Also, let $r:\Rbb_+\to \Rbb$ be a function defined as $r(z)=\lambda z^2$, which is an increasing function.
For $i\in\Iu$ and
$j\in\IP$, we know that $\Lu{t_i}(\vc{g}) = \inner{\varphi_i}{\vc{g}}$,  $\Fwrj{j}(\vc{g})=\inner{\varphi_{\nD+2j}}{\vc{g}}$ 
and
$\Fwrj{j}(\vc{g})=\inner{\varphi_{\nD+2j+1}}{\vc{g}}$.
Accordingly, we have 
\begin{equation*}
\Ecal(\vc{g}) + \lambda \Rcal(\vc{g}) =
e(\inner{\varphi_0}{\vc{g}},\ldots,\inner{\varphi_{m-1}}{\vc{g}}) + 
r(\|\vc{g}\|). 
\end{equation*}
From Theorem \ref{thm:opt_unique_convex_eps}, it follows that 
$\min_{\vc{g}\in\Hk} \Ecal(\vc{g})+\lambda \Rcal(\vc{g})$ has a unique solution denoted by $\gstareps$.
Therefore, due to Theorem \ref{thm:rep_thm}, it has a solution which belongs to $\Wscr 
= \ \linspan\{\varphi_i\}_{i=0}^{m-1}$.
%\begin{equation*}
%\begin{split}
%\Wscr 
%:=& \ \linspan\{\varphi_i\}_{i=0}^{m-1} = \bigg\{\sum_{i=0}^{m-1}x_i\varphi_i
%\bigg|
%x_i\in\Rbb,  0\le i\le m-1\bigg\}.	
%\end{split}	
%\end{equation*}
From the uniqueness of the solution of \eqref{eqn:opt_4}, it follows that $\gstareps$ belongs to $\Wscr$.
Therefore, in order to find $\gstareps$,  we need to obtain the corresponding coefficients in linear representation \eqref{eqn:gstareps_lin_xphi}.
Accordingly, we replace $\vcg$ in $\min_{\vc{g}\in\Hk} \Ecal(\vc{g})+\lambda \Rcal(\vc{g})$, or equivalently in \eqref{eqn:opt_4}, by $\vc{g}=\sum_{i=0}^{m-1}x_i\varphi_i$, and solve the problem for $\vcx:=[x_i]_{i=0}^{m-1}\in\Rbb^m$.
% in $\Rcal(\vc{g})$ and $\Ecal(\vc{g})$.
Due to the linearity of inner product and the definition of matrix $\Phi$, 
%for $\Rcal(\vc{g})$ 
we have 
\begin{equation}
\begin{split}
\Rcal(\vc{g}) =
%\|\vc{g}\|^2 &=
\inner{\vc{g}}{\vc{g}} 
%\\&
= 
%\langle\sum_{i=0}^{m-1}x_i\varphi_i,\sum_{j=0}^{m-1}x_j\varphi_i\rangle 
%\\&=
\sum_{i=0}^{m-1}\sum_{j=0}^{m-1}x_ix_j
\inner{\varphi_i}{\varphi_j}=
\vc{x}^\tr\Phi\vc{x}.
\end{split}
\end{equation}
Also, for each $i\in\Iu$, due to the definition of $\vca_i$, one has
\begin{equation}
\Lu{t_i}(\vc{g}) 
\!=\! \langle{\varphi_{i}},{\sum_{k=0}^{m-1}\!x_k\varphi_k}\rangle 
%\\&
\!=\!
\sum_{k=0}^{m-1}\!
x_k\inner{\varphi_i}{\varphi_k} 
\!=\!
\vc{a}_i^\tr \vc{x}.
\end{equation}
Similarly,  from the definition of vectors $\vc{b}_j$ and $\vc{c}_j$, one has
$ \Fwrj{j}(\vc{g})=\vc{b}_j^\tr \vc{x}$ and $\Fwij{j}(\vc{g}) =\vc{c}_j^\tr \vc{x}$,
\iffalse
\begin{equation}
\begin{split}
\Fwrj{j}(\vc{g}) 
&=
\langle{\varphi_{\nD+2j}},
{\sum_{k=0}^{m-1}x_k\varphi_k}\rangle 
=\vc{b}_j^\tr \vc{x},
\\
\Fwij{j}(\vc{g}) &=
\langle\varphi_{\nD+2j+1},
\sum_{k=0}^{m-1}x_k\varphi_k\rangle 
=\vc{c}_j^\tr \vc{x},
\end{split}
\end{equation}
\fi
for each $j\in\IP$.
Therefore, from \eqref{eqn:E}, it follows that
\begin{equation*}
\Ecal(\vc{g}) = 
\Ecal(\sum_{k=0}^{m-1}x_k\varphi_k)
= \sum_{i=0}^{\nD-1}(\vc{a}_i^\tr\vc{x}-y_i)^2
+ 
\sum_{j=0}^{\nP}\delta_{\Bscr_j}(\vc{x}),
\end{equation*}
where, for each $j\in\IP$, $\Bscr_j\subseteq \Rbb^m$ is the set defined as 
\begin{equation*}\!\!
\Bscr_j:=
\bigg\{\vc{x}
\!=\!
(x_i)_{i=0}^{m-1}\in\Rbb^m
\ \!\bigg|\ \!
(\vc{b}_j^\tr\vc{x})^2+(\vc{c}_j^\tr\vc{x})^2\le 1-\epsilon
\bigg\}.	
\end{equation*}
Accordingly, solving
$\min_{\vc{g}\in\Hk} \Ecal(\vc{g})+\lambda \Rcal(\vc{g})$ 
reduces to
\begin{equation}
\min_{\vc{x}\in\Rbb^m} \sum_{i=0}^{\nD-1}(\vc{a}_i^\tr\vc{x}-y_i)^2
+ 
\sum_{j=0}^{\nP}\delta_{\Bscr_j}(\vc{x}) 
+
\lambda 
\vc{x}^\tr\Phi\vc{x},
\end{equation}
which is equivalent to \eqref{eqn:opt_6_finite}.
%\begin{equation}\label{eqn:opt_6_finite_pf}
%\begin{array}{cl}
%\minOp_{\vc{x}\in\Rbb^m}
%&
%\sumOp_{i=0}^{\nD-1}(\vc{a}_i^\tr\vc{x}-y_i)^2 + \lambda \vc{x}^\tr\Phi\vc{x}
%\\
%\mathrm{s.t.}
%& 
%(\vc{b}_j^\tr\vc{x})^2
%+
%(\vc{c}_j^\tr\vc{x})^2
%\le 1-\epsilon, \quad  \forall j=0,\ldots,\nP.
%\end{array}
%\end{equation}
Since $\vc{x}^\tr\Phi\vc{x}=\|\sum_{i=0}^{m-1}x_i\varphi_i\|^2$ and $\{\varphi_i\}_{i=0}^{m-1}$ are linearly independent, $\Phi$ is a positive definite matrix. Therefore, the cost function in 
\eqref{eqn:opt_6_finite}
%\eqref{eqn:opt_6_finite_pf} 
is strongly convex.
Also, for each $j=0,\ldots, \nP$, we have
\begin{equation}
(\vc{b}_j^\tr\vc{x})^2+(\vc{c}_j^\tr\vc{x})^2
=
\vc{x}^\tr(\vc{b}_j\vc{b}_j^\tr+\vc{c}_j\vc{c}_j^\tr)\vc{x},
\end{equation}
and $\vc{b}_j\vc{b}_j^\tr+\vc{c}_j\vc{c}_j^\tr$ is a positive semi-definite matrix. 
Consequently, the feasible set in 
\eqref{eqn:opt_6_finite}
%\eqref{eqn:opt_6_finite_pf} 
is a convex and closed set. 
Moreover, since $\epsilon<1$, we know that $\vc{g} = \zero$ is a feasible point of
\eqref{eqn:opt_6_finite}.
%\eqref{eqn:opt_6_finite_pf}.
Therefore, the optimization problem 
\eqref{eqn:opt_6_finite}
%\eqref{eqn:opt_6_finite_pf} 
is a convex program with a unique solution.
This concludes the proof.
\end{proof}
%-------------------------------
Before proceeding further, we present a corollary which is useful in the implementation of the proposed approach.
Let $\Pscrbar:=\{\omegabar_i\}_{i=0}^{\nPbar}\subseteq\Pscr$, $\IPbar:=\{0,1,\ldots,\nPbar\}$ and $\mbar:=\nD+2\nPbar+2$.
Accordingly, similar to $\Phi$ and $\{\varphi_i\}_{i=0}^{m-1}$, we define $\Phibar$ and $\{\phibar_i\}_{i=0}^{\mbar-1}$.
Also, for $i\in\Iu$ and $j\in\IPbar$, let $\vcabar_i$, $\vcbbar_j$ and  $\vccbar_j$ be defined similar to 
%\eqref{eqn:a_i}, \eqref{eqn:b_j} and \eqref{eqn:c_j}
$\vca_i$, $\vcb_j$ and  $\vcc_j$, respectively.

%-------------------------------
\begin{corollary}\label{cor:Pbar}
Consider the following optimization problem
\begin{equation}\label{eqn:opt_4_bar}
\begin{array}{cl}
\minOp_{\vc{g}\in\Hk}
&
\sumOp_{t\in\Tscr}(\Lu{t}(\vc{g})-y_t)^2 + \lambda\ \! \|\vc{g}\|_{\Hk}^2,
\\
\mathrm{s.t.}
& 
|\Fw(\vc{g})|^2\le 1-\epsilon, \quad  \forall \omega\in\Pscrbar.
\end{array}
\end{equation}
Then, \eqref{eqn:opt_4_bar} has a unique solution denoted by $\gstarepsbar$.
This solution satisfies \eqref{eqn:gepsstar_Omega_bound}, and admits
a parametric form as $\gstarepsbar=\sum_{i=0}^{\mbar-1}x_i\phibar_i$
where $\vcx=[x_i]_{i=0}^{\mbar-1}$ is the solution of the following convex program
\begin{equation}\label{eqn:opt_6_finite_bar}
\begin{array}{cl}
\minOp_{\vc{x}\in\Rbb^{\mbar}}
&
\sum_{i\in\Iu}(\vcabar_i^\tr\vc{x}-y_i)^2 + \lambda \vc{x}^\tr\Phibar\vc{x}
\\
\mathrm{s.t.}
& 
(\vcbbar_j^\tr\vc{x})^2
+
(\vccbar_j^\tr\vc{x})^2
\le 1-\epsilon, \quad  \forall j\in\IPbar.
\end{array}
\end{equation}
Moreover, if $|\Fw(\gstarepsbar)|^2\le 1-\epsilon$, for all $\omega\in\Pscr$, then 
$\gstarepsbar$ coincides with $\gstareps$. 
\end{corollary}
\begin{proof}
Using similar lines of arguments for the proof of Theorem \ref{thm:opt_unique_convex_eps}, one can show the existence and uniqueness of the solution of \eqref{eqn:opt_4_bar}. 
The parametric form of $\gstarepsbar$ and the fact that $\vcx$ is the solution of \eqref{eqn:opt_6_finite_bar} can be concluded based on a proof similar to the proof of Theorem \ref{thm:opt_unique_convex_eps_finite}.

Since, for each $\omega\in\Pscr$, one has $|\Fw(\gstarepsbar)|^2\le 1-\epsilon$, we know that $\gstarepsbar$ is feasible for optimization \eqref{eqn:opt_4}. Therefore, being $\gstareps$ the optimal solution of \eqref{eqn:opt_4}, it follows that  
$ %\begin{equation*}
		\Lcal_{\Dscr}(\gstareps) + \lambda\Rcal(\gstareps)
		\le 
		\Lcal_{\Dscr}(\gstarepsbar) + \lambda\Rcal(\gstarepsbar),
$ %\end{equation*}
where $\Lcal_{\Dscr}$ is defined in \eqref{eqn:LD}. Also, we know that the feasible set in \eqref{eqn:opt_4} is a subset of feasible set of \eqref{eqn:opt_4_bar}. Accordingly, since $\gstarepsbar$ is the optimal solution for \eqref{eqn:opt_4_bar}, we have
$ %\begin{equation*}
		\Lcal_{\Dscr}(\gstarepsbar) + \lambda\Rcal(\gstarepsbar)
		\le 
		\Lcal_{\Dscr}(\gstareps) + \lambda\Rcal(\gstareps).
$ %\end{equation*}	 	
Therefore, $\gstarepsbar$ is an optimizer of \eqref{eqn:opt_4}. Consequently, from the uniqueness of the solution of \eqref{eqn:opt_4}, we have  $\gstarepsbar=\gstareps$, and concludes the proof.
\end{proof} 
%-------------------------------
Theorem \ref{thm:opt_unique_convex_eps_finite} introduces finite-dimensional convex program  \eqref{eqn:opt_6_finite} as a tractable problem equivalent to optimization \eqref{eqn:opt_4}.
Accordingly, in order to solve \eqref{eqn:opt_4} it is enough to solve \eqref{eqn:opt_6_finite} which is more discussed in the next section. 
Nevertheless, first we should verify that solving optimization problem \eqref{eqn:opt_4}, for small enough $\epsilon>0$, provides a close approximation to the solution of the main optimization problem \eqref{eqn:opt_2}.
This is addressed by the next theorem.

\begin{theorem}\label{thm:tightness}
Let 
$\Lu{}:\Hk\to\Rbb^{\nD}$ be a linear operator defined as
$\Lu{}(\vc{g}):=[\Lu{t}(\vc{g})]_{t\in\Tscr}$, for any $\vc{g}\in\Hk$.
Then, we have
\begin{equation}
\|\gstareps-\gstar\|_{\Hk}^2
\le 
4\epsilon(\|\Lu{}\|^2+1)^{\frac{1}{2}}
\Big(\sum_{t\in\Tscr}y_t^2\Big).
\end{equation}
Moreover, we have $\limOp_{\epsilon \to \zero}\gstareps = \gstar$ in $\Hk$.
Furthermore, if $\sup_{t\in\Tbb}
\kernel(t,t)^{\frac12}\!<\!\infty$, one has $\gstarepst{t}\tolimOp^{\epsilon\to 0}\gstart{t}$, uniformly in $\Tbb$. % $t\in\Tbb$.
\end{theorem}
\begin{proof}
	See Appendix \ref{sec:appendix_proof_tightness}.
\end{proof}

\begin{remark}\normalfont
The property $\sup_{t\in\Tbb}
\kernel(t,t)^{\frac12}<\infty$ is satisfied 
by the TC kernel and other common kernels in the literature, such as the
diagonally/correlated (DC) kernel and the stable spline (SS) kernel \cite{pillonetto2014kernel}.
\end{remark} 

\begin{remark}\normalfont
In the case of incorrect side-information, we have $\gS\notin\Gscr_{\kernel}(1,\Omega_{\Tbb})$. 
Let $\gSproj\in\Hk$ be the projection of $\gS$ on $\Gscr_{\kernel}(1,\Omega_{\Tbb})$, i.e.,
$\gSproj$ is defined as 
\begin{equation}
\gSproj := \argmin_{\vcg\in\Gscr_{\kernel}(1,\Omega_{\Tbb})} \|\vcg-\gS\|_{\Hk},    
\end{equation}
which exists uniquely (due to Theorem~\ref{thm:G_convex} and \cite{peypouquet2015convex}). 
Due to the definition of $\gSproj$ and since $\gstar,\gstareps\in\Gscr_{\kernel}(1,\Omega_{\Tbb})$ (see Theorem~\ref{thm:opt_unique_convex_eps}), we have
\begin{equation}
0<\|\gS-\gSproj\|_{\Hk}\le\|\gS-\gstar\|_{\Hk}, 
\end{equation}
and
\begin{equation}
0<\|\gS-\gSproj\|_{\Hk}\le\|\gS-\gstareps\|_{\Hk}. 
\end{equation}
In other words, we have a systematic bias in the estimated impulse response $\gstar$, and also, in its approximation $\gstareps$, for all $\epsilon>0$. 
Accordingly, in this situation, not including the side-information may result in a more accurate identified model.
\end{remark}

\section{Optimization Algorithm}
%In this section, we discuss further details of solving optimization problem \eqref{eqn:opt_4}  introduced in Section \ref{sec:Tractable} to address the estimation Problem \ref{prob:Hinft_le_1}.

Due to Theorem \ref{thm:opt_unique_convex_eps_finite}, the problem to be solved is
\begin{equation}\label{eqn:opt_7_finite}
\begin{array}{cl}
\minOp_{\vc{x}\in\Rbb^m}
&
\|\mxA\vcx-\vcy\|^2 + \lambda \vc{x}^\tr\Phi\vc{x}
\\
\mathrm{s.t.}
& 
\vcx^\tr
(\vcb_j\vcb_j^\tr+\vcc_j\vcc_j^\tr)\vcx
\le 1-\epsilon, \quad  \forall j\in\IP,
\end{array}
\end{equation}
where $\mxA:=[\vca_0,\ldots,\vca_{\nD-1}]^\tr %\in\Rbb^{\nD \times m}
$ and
$\vcy:=[y_t]_{t\in\Tscr}%\in\Rbb^{\nD}
$.
From the definition of matrix $\mxA$, we know that $\mxA$ is the first $\nD$ rows of $\Phi$. Also, $\vcb_j$ and $\vcc_j$ are respectively 
the $(\nD+2j+1)^{\text{\tiny{th}}}$ and 
the $(\nD+2j+2)^{\text{\tiny{th}}}$ column of $\Phi$, for $j\in\IP$.
Therefore, being matrix $\Phi$ given is sufficient for setting up the optimization problem \eqref{eqn:opt_7_finite}. To this end, for each $i,j\in\Iall$, we need to obtain the value of $\inner{\varphi_{i}}{\varphi_j}_{\Hk}$,  which demands calculating an infinite double summation, when $\Tbb=\Zbb_+$, or an improper double integral,  when $\Tbb=\Rbb_+$.
In general, to obtain these values one should employ numerical methods which are essentially inexact and also computationally demanding.
%======================================================
%\iffalse 
\iftrue
Meanwhile, one can derive these values analytically in specific but rather general situations, e.g., when the TC kernel is employed (see Appendix \ref{sec:appendix_setting_up} for more details).
Note that \eqref{eqn:opt_7_finite} is a convex {\em quadratically constrained quadratic program} program, for which there exist various efficient methods \cite{boyd2004convex}. 
For example, we can utilize methods which are based on {\em log-barrier functions} \cite{boyd2004convex} and adapt them suitably to the current settings (see Appendix \ref{sec:appendix_barrier_QCQP} for more details).
\else
Meanwhile, one can derive these values analytically in specific but rather general situations, e.g., when the TC kernel is employed\footnote{ see Appendix G of \cmr{arXiv2104.xxxxx} for more details.}.
Note that \eqref{eqn:opt_7_finite} is a convex {\em quadratically constrained quadratic program} program, for which there exist various efficient methods \cite{boyd2004convex}. 
For example, we can utilize methods which are based on {\em log-barrier functions} \cite{boyd2004convex} and adapt them suitably to the current settings\footnote{ see Appendix H of \cmr{arXiv2104.xxxxx} for more details.}.
\fi 
In many situations, such as in the example given in Section \ref{sec:problem_statement}, the constraint  $|G(\Jimage\omega)|=|\Fw(\vc{g})|^2\le 1-\epsilon$ is not binding on the whole frequency range $\Omega_{\Tbb}$. Accordingly,  it is not required to impose this constraint for each $\omega\in\Pscr$. Motivated by this fact and Corollary \ref{cor:Pbar}, we can introduce an iterative scheme. % in which a subset of $\Pscr$ is chosen and the corresponding constraints are considered in \eqref{eqn:opt_4}. Once the equivalent finite-dimensional optimization is solved, then constraint $|\Fw(\vc{g})|^2\le 1-\epsilon$ is checked on the remaining points in $\Pscr$ and subsequently, we add to the considered set the frequencies at which the constraint is violated.
More precisely, let $\Pscrbar_0:=\emptyset$, and at iteration $k$, let $\Pscrbar_k$ be a given subset of $\Pscr$.
Consider the optimization problem \eqref{eqn:opt_4} where only the constraints corresponding to the frequencies in $\Pscrbar_k$ are imposed, i.e., we have the following program
\begin{equation}\label{eqn:opt_4_bar_k}
	\begin{array}{cl}
		\minOp_{\vc{g}\in\Hk}
		&
		\sumOp_{t\in\Tscr}(\Lu{t}(\vc{g})-y_t)^2 + \lambda\ \! \|\vc{g}\|_{\Hk}^2,
		\\
		\mathrm{s.t.}
		& 
		|\Fw(\vc{g})|^2\le 1-\epsilon, \quad  \forall \omega\in\Pscrbar_k.
	\end{array}
\end{equation}
Due to Corollary \ref{cor:Pbar}, we know that \eqref{eqn:opt_4_bar_k} has a unique solution, denoted by $\vcg_k$, which can be obtained by solving an equivalent  finite-dimensional convex program as in \eqref{eqn:opt_6_finite_bar}.
Given $\vcg_k$, one can check whether the constraint $|\Fw(\vcg_k)|^2\le 1-\epsilon$ is violated on the remaining frequencies in $\Pscr$. Accordingly, one can obtain the following set
\begin{equation}\label{eqn:DeltaPscr}
\Delta_k\Pscr := 
\Big\{\omega\in\Pscr\backslash\Pscrbar_k
\ \Big|\ 
 |\Fw(\vcg_k)|^2> 1-\epsilon 
\Big\},
\end{equation}
Subsequently, we update the frequency set as $\Pscrbar_{k+1}:=\Pscrbar_k \cup \Delta_k\Pscr$, and proceed to  the iteration $k+1$.
The iterative scheme stops when $\Delta_k\Pscr = \emptyset$. 
This happens either when $\Pscrbar_k=\Pscr$ or the solution $\vcg_k$ satisfies the constraint $|\Fw(\vcg_k)|^2\le 1-\epsilon$ for all $\omega\in\Pscr$.
Consequently, due to Corollary \ref{cor:Pbar}, we have $\vcg_k=\gstareps$ when the stopping condition is met.
It is noteworthy that given matrix $\Phi$,  one can extract $\Phibar_k$ as a sub-matrix of  $\Phi$, and subsequently,  the corresponding vectors $\vcabar_i$, $\vcbbar_j$ and $\vccbar_j$ in \eqref{eqn:opt_6_finite_bar} are obtained as columns of $\Phibar_k$.
This fact improves the computational tractability of the proposed approach.  
The Algorithm \ref{alg:ID_Dissip} summarizes the introduced iterative scheme.

%--------------------------------------
\begin{algorithm}[t]
\caption{Kernel-Based Identification with Frequency Domain Side Information 
	$\|G^{(\Scal)}\|_{\Hcal_\infty}\le 1$
}\label{alg:ID_Dissip}
\begin{algorithmic}[1]
\State \textbf{input:} data set $\Dscr$, partition set $\Pscr$, kernel $\kernel$,   $\epsilon\in (0,1)$, matrix $\Phi$, unconstrained estimation $\vcg_0$ 
\State $k\leftarrow 0$ and $\vcg_k\leftarrow \vcg_0$.  
\State $\Pscrbar_k\leftarrow \emptyset $ and get $\Delta_k\Pscr$ due to \eqref{eqn:DeltaPscr}.
\While{ $\Delta_k\Pscr \ne \emptyset$}
\State $\Pscrbar\leftarrow \Pscrbar_k \cup  \Delta\Pscr_k$.
\State update $\Phibar$, $\vcabar_i$ for $i\in\Iu$, $\vcbbar_j$ and $\vccbar_j$, for $j\in\IPbar$.
\State solve optimization problem \eqref{eqn:opt_6_finite_bar} and $\vcg_k\leftarrow \sum_i x_i\phibar_i$. 
\State $\Pscr_k\leftarrow \Pscrbar$ and get $\Delta_k\Pscr$ due to \eqref{eqn:DeltaPscr}.
\EndWhile
\State \textbf{end} 
\State \textbf{Output:} $\gstareps$
\end{algorithmic}
\end{algorithm}
%--------------------------------------
This procedure generates  
a strictly increasing sequence of sets $\emptyset=\Pscrbar_0\subset\Pscrbar_1\subset\Pscrbar_2\subset\ldots\subset\Pscr$. Since $\Pscr$ is a finite set, this sequence needs to be finite and therefore, Algorithm \ref{alg:ID_Dissip} stops after finite number of iterations.

\section{Numerical Examples}\label{sec:numerical}
In this section, we provide numerical examples demonstrating the performance of the proposed scheme in Algorithm \ref{alg:ID_Dissip}.

%--------------------------------------
\begin{example}\label{exm:pf_ext}\normalfont
We consider the settings of the example given in Section \ref{sec:problem_statement}. 	
We set $\epsilon=10^{-5}$ and take partition set $\Pscr=\{\omega_i|i=0,\ldots,\nP\}$ for interval $[0,\pi]$ such that $\omega_i=\frac{\pi}{\nP}i$, for $i=0,\ldots,\nP$, with $\nP=3141$. 
We employ TC kernel and apply the Algorithm \ref{alg:ID_Dissip}.
In order to tune the hyperparameters, we utilize Bayesian optimization with lower-confidence-bound acquisition function \cite{srinivas2012information} to find the hyperparameters minimizing an objective function defined based on a cross-validation procedure. 
More precisely, for a choice of hyperparameters, 
the first $100$ points of $\Dscr$ are used for training the model and following this, 
the cost function in be optimized by Bayesian optimization is defined as the validation error calculated using the remaining points of $\Dscr$. 
Starting from $\Pscrbar_0=\emptyset$, we estimate $\vcg_0$ as the solution of the unconstrained problem. This solution violates the constraints on $\Pscrbar_1:=\{\omega_i| i=36,\ldots,140\}$.
Proceeding from this partition set, we obtain $\vcg_1$ which satisfies the constraints for all of the frequencies in $\Pscr$. Therefore, $\Delta_1\Pscr=\emptyset$ and the algorithm terminates with $\gstareps = \vcg_1$.

Figure \ref{fig:example1} shows the transfer function of $\Gstareps$ along with the
estimated models $\hat{G}_1$ and $\hat{G}_2$, obtained in Section \ref{sec:problem_statement}, and also the
estimated model $\hat{G}_3$ resulted from the method in \cite{abe2016subspace} which is briefly reviewed later in this section.
One can see that the result of proposed scheme satisfies the desired feature, and also, fits better to the true transfer function $\GS$.
For quantitative evaluation and comparison of the estimated impulse response, we use \emph{coefficient of determination}, also known as \emph{R-squared}, which is defined as following 
\begin{equation}\label{eqn:R2}
\mathrm{fit}(\vcg) = 100 \times (1-\frac{\|\vcg-\gS\|_2}{\|\gS\|_2})
\end{equation}
where $\vcg$ is the estimated impulse response. 
Here, we have $\mathrm{fit}(\hat{\vcg}_1)=87.13$\%, $\mathrm{fit}(\hat{\vcg}_2)=87.15$\%, and $\mathrm{fit}(\hat{\vcg}_3)=79.08$\%, 
where  $\hat{\vcg}_i$ is the impulse responses  corresponding to $\hat{G}_i$, for $i=1,2,3$.
On the other hand, we have $\mathrm{fit}(\gstareps)=90.84$\% which shows an improvement in the estimation as well as satisfying the given side information. 
Note that while $\hat{G}_3$ also satisfies this feature, the proposed approach outperforms significantly in terms of fitting performance. %This issue is further discussed in the next example.
\xqed{$\triangle$}
%\textit{Discussion}: The incorporation of the side information leads to ruling out spurious and undesired model candidates, and subsequently, we obtain a better estimation.
%--------------------------------------
\begin{figure}[t]
	\centering
	\includegraphics[width =0.45\textwidth]{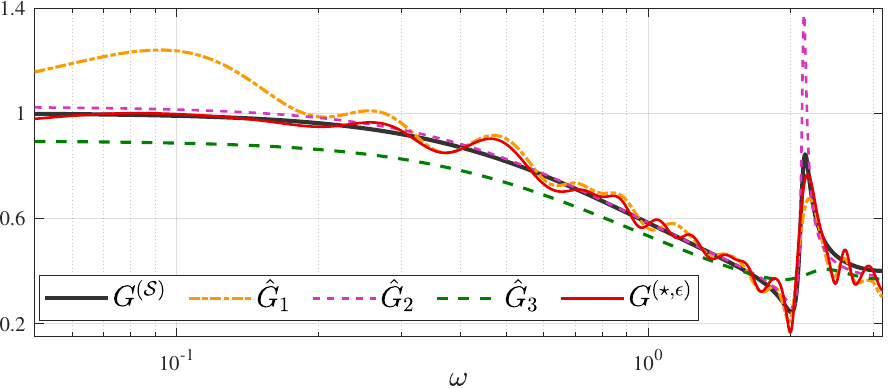}
	\caption{The transfer function of system \eqref{eqn:sys_example}, the model estimated using the proposed approach, $\Gstareps$,  and the estimated models $\hat{G}_1$, $\hat{G}_2$ and $\hat{G}_3$.}
	\label{fig:example1}
\end{figure}
\end{example}
%--------------------------------------

The next example  compares the proposed approach with the method in \cite{abe2016subspace} which is a variant of subspace identification scheme incorporating frequency domain side information, and accordingly, it is denoted by $\FDIsub$ in the followings. 
In this approach, initially a sequence of state variables is estimated, and then, a formulation of subspace identification method is presented where matrix inequalities, coming from the Kalman-Yakubovich-Popov (KYP) lemma, are imposed to estimation problem for the incorporation of the side information. 
Following this, the resulting nonlinear program is reduced to a convex one using appropriate transformations.

\begin{example}\label{exm:MC_compare}\normalfont
In this example, we perform numerical experiments to compare the proposed method with \cite{abe2016subspace}. We generate randomly four sets of $150$ stable systems using \texttt{drss} \textsc{Matlab}'s function with orders in the range of $\{10,\ldots,30\}$ and poles not larger than $0.98$. Following this, the systems are normalized with their $\Hcal_{\infty}$-norm, and then, actuated with a realization of standard random white Gaussian signal $\vc{u} = (u_t)_{t=0}^{\nD-1}$, i.e., $u_0,\ldots,u_{\nD-1}$ are i.i.d. samples of $\Ncal(0,1)$.  
The output of system is corrupted with additive Gaussian noise. The variance of output noise is chosen such that we have $10$dB, $20$dB, $30$dB and $40$dB signal-to-noise ratio (SNR) in the respective sets of the systems.
Similar to the previous example, we tune the hyperparameters of both the proposed approach and the $\FDIsub$ %one presented in \cite{abe2016subspace} 
using Bayesian optimization \cite{srinivas2012information} and cross-validation.
The settings for the proposed approach are similar to ones employed in Example \ref{exm:pf_ext}. The box plots of results are shown in Figure \ref{fig:example2}.
While, we can see that the performance of the both of these approaches improves when SNR increases, the kernel-based approach significantly outperforms the subspace-based method $\FDIsub$. 
%--------------------------------------
\begin{figure}[htb]
\begin{center}
	\includegraphics[width=0.45\textwidth]{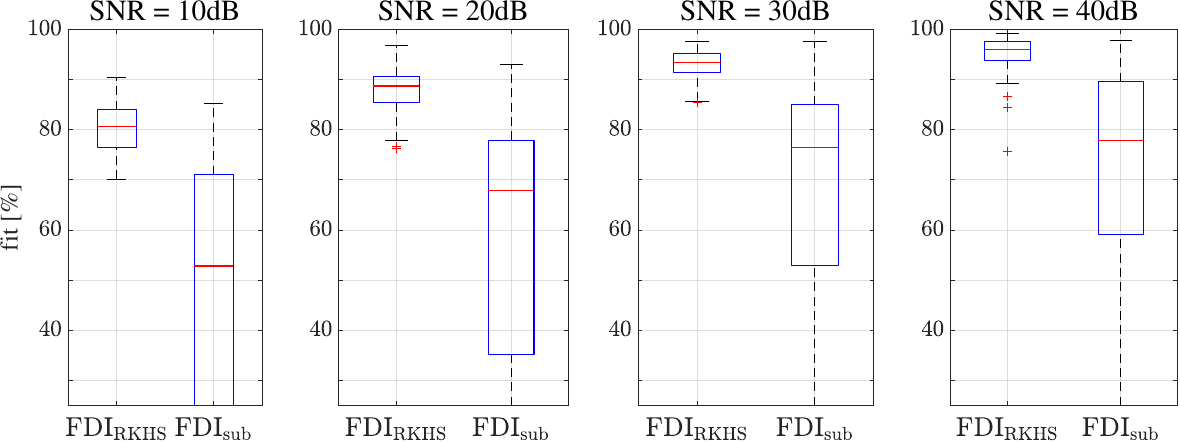}
\end{center}	
	\caption{The performance of $\FDIsub$ \cite{abe2016subspace} is compared with the proposed method, $\FDIRKHS$.}
	\label{fig:example2}
\end{figure}
%--------------------------------------
\\\noindent\underline{\textit{Discussion}}: The $\FDIsub$ approach \cite{abe2016subspace} suffers from the well-known model order selection issue \cite{ljung2020shift}.
On the other hand, the proposed scheme is a derivation of kernel-based regularization methods, and therefore, tuning the complexity of model is performed by the powerful concept of estimating continuous regularization hyperparameters rather than picking an integer order based on a selection rule \cite{pillonetto2014kernel,ljung2020shift}.
Moreover, while the proposed approach works directly with input-output data, $\FDIsub$ employs an estimation of the state trajectory which leads to model estimation prone to high variance and noisy results. \xqed{$\triangle$}
\end{example}

The next example is adapted from \cite{abe2016subspace} and	demonstrate employing the proposed method for solving a joint case of Problem \ref {prob:Hinft_err_eps} and Problem \ref{prob:Hinft_weight_eps} introduced in Section \ref{sec:problem_statement}.
\begin{example}\label{exm:MC_DT_dc_gain}\normalfont
Consider \emph{unknown} system $G(z)$ defined as
\begin{equation}
G(z) = \frac{0.1204\ \! z + 0.1184}{z^2 - 1.7114\ \! z + 0.9502},
\end{equation} 
which is actuated with  $\vc{u} := (u_t)_{t=0}^{\nD-1}$, a standard random Gaussian signal of length $\nD=150$, and the output of system, $\vc{y} := (y_t)_{t=0}^{\nD-1}$, corrupted with additive Gaussian noise, is measured with SNR = $20$dB.

Let assume  that we are given information on the {dc-gain} of system, i.e., we know that $\alphadc:=G(\expe^{\Jimage 0}) =1$.
To incorporate this information in the estimation of the impulse response of system, ${\vcg}$, one can propose two approaches, 
\emph{filter-based} (FB) and \emph{directly constrained} (DC), which are discussed in the following.
In the \emph{filter-based} approach, we identify system $G_W :=W (G -\alphadc)$ with constraint 
$\|G_W\|_{\Hcal_{\infty}}\le \gamma$, where $W $ is a suitably designed stable and inverse-stable transfer function, and $\gamma$ is a scalar chosen appropriately. Following this, $G$ is estimated as $W^{-1} \hat{G}_W +\alphadc$. 
To impose the constraint only on the low
frequencies, $W$ should be a low-pass filter with cut-off frequency close to $\omega=0$ and with comparatively large gain in low frequencies.
Here, we employ  $W(z)=a(\frac{z-1+b}{z-1+a})^2$, where $a=10^{-4}$ and $b=10^{-3}$.
Following this, the outputs of system is modified to $\vc{y}-\alphadc\vc{u}$, and then, we apply filter $W$ on them to obtain data points $\vc{p}:=(p_t)_{t=0}^{\nD-1}$.
The input-output pairs $\{(u_t,p_t)|t=0,\ldots,\nD-1\}$ correspond to the system $F$ to be estimated with constraint $\|F \|_{\Hcal_{\infty}}\le \gamma=10^{-3}$.
%--------------------------------------
\begin{figure}[tb]
	\begin{center}
		\includegraphics[width=0.45\textwidth]{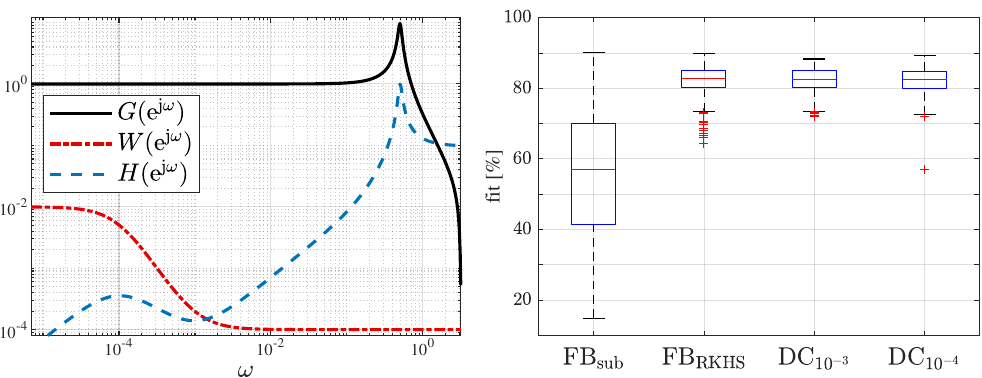}
	\end{center}	
	\caption{Left: transfer function of $G$, $W$ and $H$. Right: The performance of different approaches is compared for the estimation of impulse response of  $G$.}
	\label{fig:example3_bode_box}
\end{figure}
%--------------------------------------
Figure \ref{fig:example3_bode_box} shows bode plots of $G$, $W$ and system $H$ defined as $H:=\gamma^{-1} G_W$ for which one can see that $\|H\|_{\Hcal_{\infty}}\le 1$.
For the estimation of $G_W$, we can either employ the subspace method developed in \cite{abe2016subspace},  or the proposed identification scheme.
We denote these methods respectively by $\FBsub$ and $\FBRKHS$.
In the \emph{directly constrained} approach, we estimate system $K:=G-\alphadc$ from the input-output data $\big((u_t,y_t-\alphadc u_t)\big)_{t=0}^{\nD-1}$ subject to constraint $|K(\expe^{\Jimage 0})|\le \gamma$, which is a special case of problem \eqref{eqn:opt_1} where $\Omega_{\Tbb}$ is the singleton set $\{0\}$. 
Here, we consider two cases of $\gamma = 10^{-3}$ and $\gamma = 10^{-4}$.
We generate $150$ realizations of noise and apply the introduced approaches to estimate $\vc{g}$ with given side information. Figure \ref{fig:example3_bode_box} shows the boxplot for the estimation performance of the estimated impulse responses, $\hat{\vcg}$.
The histogram of the resulting {dc-gain}, $\hat{\alpha}_{\text{dc}}$, is shown in Figure \ref{fig:example3_hist}. 
%--------------------------------------
\begin{figure}[t]
	\begin{center}
		\includegraphics[width=0.4\textwidth]{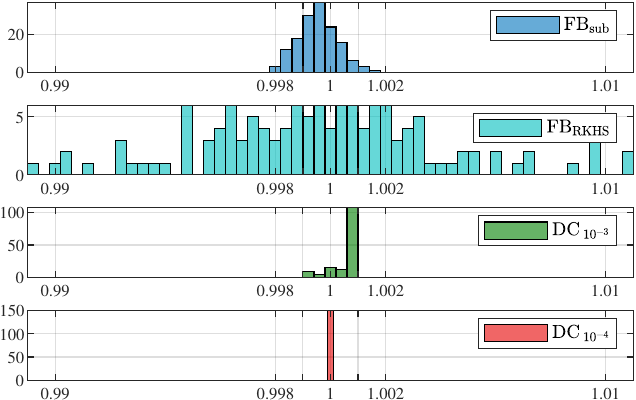}
	\end{center}	
	\caption{The histogram of the estimated values of the dc-gain.}
	\label{fig:example3_hist}
\end{figure}
%--------------------------------------
For further quantitative evaluation and comparison of the estimated values, the bias, variance and mean squared error (MSE) of the estimations are presented in Table \ref{tbl:mc}.
%--------------------------------------
\begin{table}[h]
\renewcommand{\arraystretch}{1.2} 
\centerline{
\begin{tabular}{lcccc}
\toprule\hline
 & \textbf{$\FBsub$} & \textbf{$\FBRKHS$} & \textbf{$\text{DC}_{10^{-3}}$} & \textbf{$\text{DC}_{10^{-4}}$} \\ 
\hline 
%-------------------
\textbf{$\text{Bias}^2$}$(\hat{\vcg})$
& 0.338 & 0.027 & 0.024 & 0.025 \\
\textbf{$\text{Var}$}$(\hat{\vcg})$
& 0.227 & 0.059 & 0.057 & 0.060 \\ 
\textbf{$\text{MSE}$}$(\hat{\vcg})$
& 0.564 & 0.086 & 0.081 & 0.085 \\ 
\hline
%-------------------
\textbf{$\text{Bias}^2$}$(\hat{\alpha}_{\text{dc}})\ [\times 10^{-6}]$ 
& 0.214 & 0.151 & 0.418 & 0.003 \\
\textbf{$\text{Var}$}$(\hat{\alpha}_{\text{dc}})\ [\times 10^{-6}]$    
& 0.478 & 81.691 & 0.306 & 0.003 \\
\textbf{$\text{MSE}$}$(\hat{\alpha}_{\text{dc}})\ [\times 10^{-6}]$    
& 0.692 & 81.842 & 0.724 & 0.006\\ 
\hline
\bottomrule
\end{tabular}
}%\label{tbl:mc}
\caption{ \label{tbl:mc} 
	Bias, variance and MSE for the estimations}
\end{table}
%--------------------------------------
\\\noindent\underline{\textit{Discussion}}: 
The numerical experiment shows that the $\FBRKHS$ and directly constrained methods %with $\gamma=10^{-3},10^{-4}$ 
have similar performance in the estimation of impulse response. Moreover, they provide considerably more accurate estimation of $\vcg$ comparing to the $\FBsub$ method, which is expected due to Example \ref{exm:MC_compare}.
Additionally, comparing the {dc-gain} calculated from the results, one can see that the values obtained from the directly constrained method with $\gamma=10^{-4}$ are significantly closer to the given {dc-gain} of system.
The directly constrained method with $\gamma=10^{-3}$ and $\FBsub$ method behave similarly accurate regarding the calculated {dc-gain} in the sense of MSE, which is potentially  due to the same choice for $\gamma$. On the other hand, the former method, which is based on the proposed approach, estimates the impulse response $\vcg$ significantly more accurate. 
Due to the current choice of $W$, employing smaller values of $\gamma$ is not feasible in the filter-based approach as one can see from the transfer function $H$ shown in Figure \ref{fig:example3_bode_box}, where we have $\|H\|_{\Hcal_{\infty}}=0.9984$.
Meanwhile, the employed $W$ still does not impose the desired constraint on the low  frequencies strong enough, which results in a considerable degree of freedom for the {dc-gain}.
Consequently, while the $\FBRKHS$ method estimates the impulse response more precisely than  $\FBsub$, the corresponding calculated {dc-gain} values are not as accurate as the ones provided by $\FBsub$ method.
The better performance of $\FBsub$ method regarding the calculated {dc-gain} is potentially due to being the original system low order.
To improve the $\FBRKHS$  from this aspect, we need to design low-pass filter $W$ with narrower band and higher gain, which results in large order filters with  poles and zeros close to $z=1$. 
This can lead to numerical instability and necessity of larger set of data. 
On the other hand, the directly constrained methods do not depend on the choice of $W$ and perform well, especially for small enough values of $\gamma$ like  $\gamma=10^{-4}$.
\xqed{$\triangle$}
\end{example}

The next example concerns Problem \ref{prob:Hinft_le_1} for the case of continuous-time and demonstrates the role of proposed approach for the continuous-time systems. 

\begin{example}\label{exm:CT}\normalfont
Consider the \emph{unknown} continuous-time system $\Scal$ with transfer function $\GS(s)$ defined \cite{scandella2020kernel} as
\begin{equation}\label{exm:sys_CT}\!\!\!\!
\GS(s)\!=\!- 
\frac{2s^3 \!+\! 3.6s^2 \!+\! 2.095s \!+\! 0.396s} {0.461s^4 \!+\! 2.628s^3 \!+\! 4.389s^2 \!+\! 2.662s \!+\! 0.519},\!\!
\end{equation}
with side information $\|\GS\|_{\Hcal_{\infty}}\le 1$. 
Let the system be initially at rest and actuated with a random switching pulse signal as shown in Figure \ref{fig:example4}. 
Then, the output of system is measure at $\nD=250$ time instants 
$t_0,\ldots,t_{\nD-1}\in[0,10]$, where
$t_k = kT_{\mathrm{s}}+\delta_k$, for $k=0,\ldots,\nD-1$, with $T_{\mathrm{s}}=0.04$ and $\delta_k\sim \mathrm{Uniform}([0,T_{\mathrm{s}}])$.
The output measurements, $\{y_{t_i}\}_{i=0}^{\nD-1}$, are subject to additive white Gaussian noise such that SNR is  $20$dB. 
We interpolate the output values at time instants $\bar{t}_k = (k\! +\! 1)T_{\mathrm{s}}$, for $k=0,\ldots,\nD-1$, using shape-preserving piecewise cubic interpolation available in \textsc{Matlab}'s function \texttt{interp1} and \texttt{pchip} option.
The SNR in the interpolated outputs, $\{\bar{y}_{t_i}\}_{i=0}^{\nD-1}$, is $20.68$dB. This slight improvement in SNR is potentially due to being the initial point almost regularly distributed in the sampling time interval.

We employ \textsc{CONTSID Toolbox} \cite{garnier2018contsid} and obtain transfer function estimations $\hat{G}_1$ and $\hat{G}_2$, respectively using 
\texttt{tfsrivc} and \texttt{rivc} functions with known orders of system.
Furthermore, we identify the system through \emph{indirect} approaches, i.e., first a discrete-time impulse response is estimated using the interpolated data, and then, the continuous-time version is derived by shape-preserving piecewise cubic interpolation method as discussed above.
To this end, we use the subspace method \cite{abe2016subspace} explained in Example \ref{exm:MC_compare} and also, the discrete-time version of the proposed method. Let the corresponding transfer functions be denoted respectively  by $\hat{G}_{\mathrm{sub}}$ and $\Gstareps_{\mathrm{ind}}$. 
In obtaining $\Gstareps_{\mathrm{ind}}$, $\epsilon$ is set to $10^{-3}$ and the rest of settings are taken similar to Example \ref{exm:pf_ext}. 
Starting from $\Pscrbar_0=\emptyset$, the algorithm terminates in the second iteration with $\Pscrbar_1:=\{\omega_i| i=50,\ldots,76\}$.

In addition to the above methods, we identify the system in a \emph{direct} approach based on the proposed algorithm for the case of continuous-time, and using original measurement data. 
Accordingly, we apply Algorithm \ref{alg:ID_Dissip} where $\epsilon$ is chosen as  $10^{-3}$, the partition set is taken as $\Pscr=\{\omega_i=10^{-2}i|i=0,\ldots,\nP=10^4\}$, TC kernel is employed, and hyperparameters are tuned similar to Example \ref{exm:pf_ext}.
Initially, we have  $\Pscrbar_0=\emptyset$ where the resulting solution, $\vcg_0$, violates the constraints on $\Delta\Pscr_0:=\{\omega_i| i=159,\ldots,205\}$.
In the next iteration, $\Pscrbar_1$ is $\Delta\Pscr_0$ for which we obtain  solution $\vcg_1$. The result  satisfies the constraints for the frequencies in $\Pscr$, and,  hence, $\Delta_1\Pscr=\emptyset$, i.e., the algorithm terminates in the second iteration. Denote the estimated impulse response and the corresponding transfer function respectively by $\gstareps_{\mathrm{dir}}\!$ and $\Gstareps_{\mathrm{dir}}\!$. 

%--------------------------------------
\begin{figure}[t]
	\centering
	\includegraphics[width =0.4\textwidth]{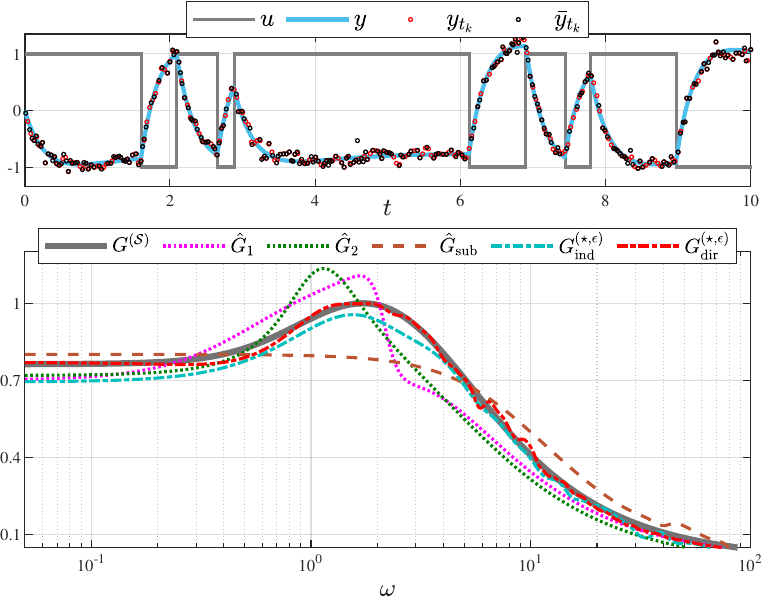}
	\caption{Top: the input and output of system, output measurements and interpolations. Bottom:  the transfer function of system \eqref{exm:sys_CT}, the model estimated using the proposed approach, $\Gstareps_{\mathrm{dir}}$,  and the estimated models $\hat{G}_1$, $\hat{G}_2$, $\hat{G}_{\mathrm{sub}}$ and $\Gstareps_{\mathrm{ind}}$.
	}
	\label{fig:example4}
\end{figure}
%--------------------------------------
In Figure \ref{fig:example4}, the estimated transfer functions 
$\hat{G}_1$, $\hat{G}_2$, $\hat{G}_{\mathrm{sub}}$, $\Gstareps_{\mathrm{ind}}$ and $\Gstareps_{\mathrm{dir}}$ are shown and graphically compared with $\GS$.
From the figure, one can see that the side information is preserved only for the estimated transfer functions $\hat{G}_{\mathrm{sub}}$, $\Gstareps_{\mathrm{ind}}$ and $\Gstareps_{\mathrm{ind}}$, and it is violated by $\hat{G}_1$ and $\hat{G}_2$. Indeed, we have $\|\hat{G}_1\|_{\Hcal_{\infty}}=1.11$ and $\|\hat{G}_2\|_{\Hcal_{\infty}}=1.17$.
Furthermore, we can see that the proposed method significantly outperforms other schemes.
To evaluate quantitatively the estimation results, we employ \emph{R-squared} metric defined in \eqref{eqn:R2}.
Accordingly, the fitting results are 
$\mathrm{fit}(\hat{\vcg}_1) = 78.93$\%, 
$\mathrm{fit}(\hat{\vcg}_2) = 63.37$\%, 
$\mathrm{fit}(\hat{\vcg}_{\mathrm{sub}}) = 73.76$\%, 
$\mathrm{fit}(\gstareps_{\mathrm{ind}})  = 87.15$\%, and
$\mathrm{fit}(\gstareps_{\mathrm{dir}})  = 91.86$\%, 
where  
$\hat{\vcg}_1$, $\hat{\vcg}_2$, $\hat{\vcg}_{\mathrm{sub}}$ and $\gstareps_{\mathrm{ind}}$  
are the impulse responses corresponding 
$\hat{G}_1$, $\hat{G}_2$, $\hat{G}_{\mathrm{sub}}$ and $\Gstareps_{\mathrm{ind}}$.  

For further comparison, we perform a Monte Carlo experiment with two sets of $100$ runs, where the settings are similar to the numerical experiment above with the only difference that regular sampling  is employed in the first set.
In the case of regular sampling, the original measurement samples coincide with the interpolated ones. 
On the other hand, for the irregular case, the interpolation of output measurements differs from the original ones, however the resulting SNR is in the interval $19.8$dB to  $21.6$dB, with average SNR equal to $20.8$dB.
This is due to  being  the time instants of output measurements close to regular and the fact that the response of system is fairly smooth.
Using the previously mentioned direct and indirect approaches, we estimate the impulse response of system in each run, i.e., we obtain impulse response estimations $\hat{\vcg}_{\mathrm{sub}}$, $\gstareps_{\mathrm{ind}}$ and  $\gstareps_{\mathrm{dir}}$, which are satisfying the given side information.
Figure \ref{fig:example4_boxplot} shows the boxplot comparing the estimation performance of the results. Also, Table \ref{tbl:mc_exm4} provides the bias, variance and mean squared error (MSE) of the estimations.
%--------------------------------------
\begin{figure}[t]
	\begin{center}
		\includegraphics[width=0.4\textwidth]{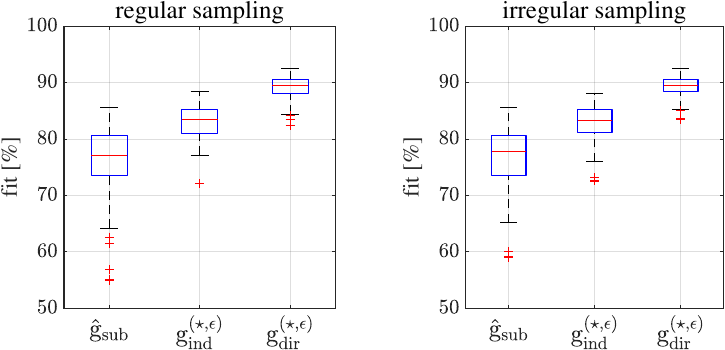}
	\end{center}	
	\caption{The performance of different estimation for the cases with regular sampling (left) and irregular sampling (right).}
	\label{fig:example4_boxplot}
\end{figure}
%--------------------------------------
%--------------------------------------
\begin{table}[b]
\renewcommand{\arraystretch}{1.05} 
\centerline{
\begin{tabular}{c|lccc}
\toprule 
\hline
\multirow{2}{*}{samples}
& \multirow{2}{*}{method}
& \textbf{$\text{Bias}^2$}$(\hat{\vcg})$%{$[\times 10^{-6}]$}
& \textbf{$\text{Var}$}$(\hat{\vcg})$   %{$[\times 10^{-6}]$}  
& \textbf{$\text{MSE}$}$(\hat{\vcg})$   %{$[\times 10^{-6}]$}  
\\ 
&
&{$[\times 10^{-6}]$}
&{$[\times 10^{-6}]$}  
&{$[\times 10^{-6}]$}  
\\ \hline %-------------------
\multirow{3}{*}{\rotatebox[origin=c]{90}{regular}}
&
$\hat{G}_{\mathrm{sub}}$
& 6.39 & 7.97 & 14.36 \\
&
$\Gstareps_{\mathrm{ind}}$
& 5.33 & 1.73 & 7.07 \\ 
&
$\Gstareps_{\mathrm{dir}}$
& 2.25 & 0.63 & 2.88 \\ 
\hline
%-------------------
\multirow{3}{*}{\rotatebox[origin=c]{90}{irregular}}
&
$\hat{G}_{\mathrm{sub}}$
& 7.40 & 6.36 & 13.76 \\
&
$\Gstareps_{\mathrm{ind}}$
& 5.48 & 1.79 & 7.27 \\ 
&
$\Gstareps_{\mathrm{dir}}$
& 2.16 & 0.67 & 2.83 \\ 
\hline
\bottomrule
\end{tabular}
}
\caption{\label{tbl:mc_exm4} Bias, variance and MSE for the estimations}
\end{table}
%--------------------------------------
%--------------------------------------
\\\noindent\underline{\textit{Discussion}}:
The first and most important observation is that direct method significantly outperforms the indirect approaches.
This observation is in confirmation with the literature on continuous-time system identification \cite{garnier2008direct} and highlights the importance of developing the proposed scheme for the case of continuous-time impulse responses.
Furthermore, one can see  that the proposed methods shows better estimation performance comparing to the subspace method, which is expected due to Example \ref{exm:MC_compare}.
Comparing the estimation results for the regularly and irregularly sampling cases, the minor improvement of the subspace method \cite{abe2016subspace} for the case of irregular measurement samples is potentially due to slightly higher SNR in the interpolation data.
Meanwhile, for other two approaches we observe almost similar behavior, and the current numerical experiments does not reveal further aspects.
Finally, we should note that when the sampling is highly irregular or the output of system is wiggles significantly, the interpolation either provides erroneous results or it is infeasible.
In this situation, one can only employ the direct approach, i.e., the proposed method for the case of continuous-time. 
This issue again highlights the importance of the developed scheme for the estimation of continuous-time impulse responses.
\xqed{$\triangle$}	
\end{example}

\section{Conclusion}\label{sec:conclusion}
The problem identification with side information on the dissipativity of the system has been studied in this paper.
We have employed a RKHS frame work allowing considering both of the discrete-time and continuous-time dynamics.
The problem is initially formulated as an infinite-dimensional optimization problem  estimating a stable impulse response fitting to the data and satisfying an $\Hcal_\infty$-norm constraint.
We have shown that the problem is well-defined and convex with a unique solution.
To obtain the solution, we have proposed a heuristic tightly approximating the solution. It is shown that the proposed approach is equivalent to solve a finite-dimensional convex quadratically constrained quadratic programming. 
The efficiency of the discussed method is verify by several numerical examples.
%================================
\appendix
\section{Appendix} 
%-----------------------------------------------
\label{sec:AppendixProofs}
%-----------------------------------------------
%\begin{lemma}\label{lem:gk_le_normg_Kkk}
%For any $\vc{g}\in\Hk$ and $k\in\Zbb_+$, we have $|g_k|\le \|\vc{g}\|_{\Hk} \kernel(k,k)^{\frac{1}{2}}$.
%\end{lemma}
%\begin{proof}
%From reproducing property of the kernel, we have that $g_k = \inner{\vc{g}}{\kernel_{k}}_{\Hk}$, for any $k\in\Zbb_+$.
%Therefore, from Cauchy-Schwartz inequality, one has
%\begin{equation}
%    |g_k|= |\inner{\vc{g}}{\kernel_{k}}_{\Hk}|\le \|\vc{g}\|_{\Hk} \|\kernel_k\|_{\Hk} 
%\end{equation}
%Also, we have $\|\kernel_k\|_{\Hk}^2 =  \inner{\kernel_{k}}{\kernel_{k}}_{\Hk}$. Subsequently, from the reproducing property, it follows that $\|\kernel_k\|_{\Hk} = \inner{\kernel_{k}}{\kernel_{k}}_{\Hk}^{\frac{1}{2}}$. Accordingly, we have $|g_k|\le \|\vc{g}\|_{\Hk} \kernel(k,k)^{\frac{1}{2}}$ and the proof concludes.
%\end{proof}
%================================
\subsection{Proof of Lemma~\ref{lem:Fw}} \label{sec:appendix_proof_lem_Fw}
%\begin{proof}
Consider the case $\Tbb=\Rbb_+$ and let 
$\vcgone:=(\gone_t)_{t\in\Rbb_+}, \vcgtwo:=(\gtwo_t)_{t\in\Rbb_+}\in\Hk$ and 
$\alpha_1,\alpha_2\in \Cbb$. 
Since $\Hk\subseteq\Lscr_1$, we know that $\|\vcgone\|_{1}$ and $\|\vcgtwo\|_{1}$ are finite.
Also, for any  $t\in\Rbb_+$ and any $\omega\in\Rbb_+$, from triangle inequality, we have
$|(\alpha_1\gone_t+\alpha_2\gtwo_t)\ \expe^{-\Jimage\omega t}| 
\le  |\alpha_1||\gone_t|+|\alpha_2||\gtwo_t|$. 
%\begin{equation*}
%|(\alpha_1\gone_t+\alpha_2\gtwo_t)\ \expe^{-\Jimage\omega t}| 
%\le  
%|\alpha_1||\gone_t|+|\alpha_2||\gtwo_t|.
%\end{equation*}
Hence, 
$\big{(}(\alpha_1\gone_t+\alpha_2\gtwo_t)\ \expe^{-\Jimage\omega t}\big{)}_{t\in\Rbb_+}$ is absolutely integrable. Therefore, the improper integral
$\Fw(\alpha_1\vcgone+\alpha_2\vcgtwo):=\int_{\Rbb_+}(\alpha_1\gone_t+\alpha_2\gtwo_t)\ \expe^{-\Jimage\omega t}\drm t$  
%\begin{equation*}
%\Fw(\alpha_1\vcgone+\alpha_2\vcgtwo):=\int_{\Rbb_+}(\alpha_1\gone_t+\alpha_2\gtwo_t)\ \expe^{-\Jimage\omega t}\drm t
%\end{equation*}
converges and we have
%$\Fw(\alpha_1\vcgone+\alpha_2\vcgtwo) := \alpha_1\Fw(\vcgone)+\alpha_2\Fw(\vcgtwo)$.    
\begin{equation}\label{eqn:proof_lem_2}
\Fw(\alpha_1\vcgone+\alpha_2\vcgtwo) := \alpha_1\Fw(\vcgone)+\alpha_2\Fw(\vcgtwo).    
\end{equation}
Therefore,  $\Fw(\alpha_1\vcgone+\alpha_2\vcgtwo)$ is a well-defined $\Cbb$-valued linear map.
Moreover, from $\expe^{-\Jimage\omega t} = \cos(\omega t)-\Jimage\sin(\omega t)$, one can easily see that 
$\Fwr(\vc{g}) = \real{\Fw(\vc{g})}$,
$\Fwi(\vc{g}) = \imag{\Fw(\vc{g})}$, and
$\Fw(\vc{g}) = \Fwr(\vc{g}) + \Jimage \Fwi(\vc{g})$, for any $\vc{g}\in\Hk$.
Let $\alpha_1,\alpha_2\in\Rbb$ and  compare the real and imaginary parts of the left-hand side and right-hand side of \eqref{eqn:proof_lem_2}. Hence, we have 
\begin{equation*}
\begin{array}{ccl}
\Fwr(\alpha_1\vcgone+\alpha_2\vcgtwo) &=& \alpha_1\Fwr(\vcgone)+\alpha_2\Fwr(\vcgtwo),    
     \\
\Fwi(\alpha_1\vcgone+\alpha_2\vcgtwo) &=& \alpha_1\Fwi(\vcgone)+\alpha_2\Fwi(\vcgtwo),    
\end{array}    
\end{equation*}
and also $|\Fwr(\alpha_1\vcgone+\alpha_2\vcgtwo)|,|\Fwi(\alpha_1\vcgone+\alpha_2\vcgtwo)| <\infty$.
This shows that $\Fwr$ and $\Fwi$ are well-defined linear maps.

Since $\Fwr(\vc{g}) = \real{\Fw(\vc{g})}$, from triangle inequality and the definition of $\Fw$, one can see
% \begin{equation*}
% |\Fwr(\vc{g})|%\!=\!|\real{\Fw(\vc{g})}| 
% \!\le\!
% |\Fw(\vc{g})| \!=\! \bigg|\!\int_{\Rbb_+}\!\!\!\!\! g_t \expe^{-\Jimage\omega t}\drm t\bigg| \!\le \!\int_{\Rbb_+}\!\!\!\!\!|g_t|\drm t,
% \end{equation*}
\begin{equation}
\begin{split}
|\Fwr&(\vc{g})| = |\real{\Fw(\vc{g})}| 
\\&\le
|\Fw(\vc{g})| = \bigg|\int_{\Rbb_+}g_t \expe^{-\Jimage\omega t}\drm t\bigg| \le \int_{\Rbb_+}|g_t|\drm t,
\end{split}	
\end{equation}
for any $\vc{g}\in\Hk$. 
Subsequently, due to the reproducing property and the Cauchy-Schwartz inequality, we have
$|\Fwr(\vc{g})| \le  
\int_{\Rbb_+}|\inner{\vcg}{\kernel_t}|\drm t \le 
\int_{\Rbb_+}\|\vcg\|\|\kernel_t\|\drm t$.
%\begin{equation*}
%|\Fwr(\vc{g})| \le  
%\int_{\Rbb_+}|\inner{\vcg}{\kernel_t} |\drm t \le 
%\int_{\Rbb_+}\|\vcg\| \|\kernel_t\| \drm t.
%\end{equation*}
From $\|\kernel_t\|^2 = \inner{\kernel_t}{\kernel_t} = \kernel(t,t)$ and the definition of $\mu_0$, one can see 
$|\Fwr(\vc{g})| \le \mu_0\|\vcg\| $,
%\begin{equation}
%	|\Fwr(\vc{g})| \le \mu_0\|\vcg\| ,
%\end{equation}
and hence, we have
\begin{equation*}
\|\Fwr\|_{\Lcal(\Hk,\Rbb)}
:= 
\sup_{\substack{\vc{g}\in\Hk\\ \|\vc{g}\|\le 1}} |\Fwr(\vc{g})| 
\le
\sup_{\substack{\vc{g}\in\Hk\\ \|\vc{g}\|\le 1}}  \mu_0\|\vcg\| =\mu_0.
\end{equation*}
Similarly, one can show that
$\|\Fwi\|_{\Lcal(\Hk,\Rbb)}\le \mu_0$.
%\begin{equation}
%\|\Fwi\|_{\Lcal(\Hk,\Rbb)}
%:= \sup_{\vc{g}\in\Hk, \|\vc{g}\| \le 1} |\Fwi(\vc{g})| \le \mu_0.
%\end{equation}
Therefore, $\Fwr$ and $\Fwi$ are linear bounded functionals, i.e., $\Fwr,\Fwi\in \Lcal(\Hk,\Rbb)$.
Since $\Hk$ is a Hilbert space, due to the Riesz's representation theorem, we have that there exist unique  
$\phiwr:=(\phiwrk{t})_{t\in\Rbb_+}\in\Hk$ and 
$\phiwi:=(\phiwik{t})_{t\in\Rbb_+}\in\Hk$ such that, 
for any $\vcg\in\Hk$, one has
$\Fwr(\vcg)=\inner{\phiwr}{\vcg}$
 and 
$\Fwi(\vcg)=\inner{\phiwi}{\vcg}$.
Subsequently, from the reproducing property of kernel, for any $t\in\Rbb_+$, we have that
\begin{equation*}
	\begin{array}{ll}
		&\phiwrk{t}=\inner{\phiwr}{\kernel_t} 
		=
		\Fwr(\kernel_t)
		= \int_{\Rbb_+}\!\!\!\kernel(t,s)\cos(\omega s)\drm s,\\
		&\phiwik{t}=\inner{\phiwi}{\kernel_t} 
		\!=\!
		\Fwr(\kernel_t)
		= -\!\int_{\Rbb_+}\!\!\!\kernel(t,s)\sin(\omega s)\drm s.
	\end{array}	
\end{equation*}
%\begin{equation}
%    \!\!\!
%    \phiwrk{t}=\inner{\phiwr}{\kernel_t} 
%    =
%    \Fwr(\kernel_t)
%    = \int_{\Rbb_+}\!\!\!\kernel(t,s)\cos(\omega s)\drm s,
%\end{equation}
%and
%\begin{equation}
%    \!\!\!\!
%    \phiwik{t}=\inner{\phiwi}{\kernel_t} 
%    \!=\!
%    \Fwr(\kernel_t)
%    = -\!\int_{\Rbb_+}\!\!\!\kernel(t,s)\sin(\omega s)\drm s.
%\end{equation}
This concludes the proof for the case $\Tbb=\Rbb_+$.
Similar lines of arguments hold when $\Tbb=\Zbb_+$ and the proof is concluded.
\qed
%================================
%\subsection{Proof of Lemma~\ref{lem:G_rho_omega}}
%\begin{proof}
%Since $|\Fw(\vc{g})|^2 = |\Fwr(\vc{g})|^2 + |\Fwi(\vc{g})|^2$, we have that
%\begin{equation}
%    \Gcal_{\kernel}(\rho,\{\omega\})
%    =
%    \{\vc{g}\in\Hk\ |\ |\Fwr(\vc{g})|^2 + |\Fwi(\vc{g})|\le \rho^2\}.
%\end{equation}
%Due to Lemma \ref{lem:Fw}, we know that  $\Fwr(\vc{g})=\inner{\phiwr}{\vc{g}} $ and $\Fwi(\vc{g})=\inner{\phiwi}{\vc{g}} $, for any $\vc{g}\in\Hk$. Therefore,  we have
%\begin{equation}\!\!\!\!\!
%\begin{split}
%&\Gcal_{\kernel}(\rho,\{\omega\})
%=
%%\\&
%\capOp_{\theta\in[0,\frac{\pi}{2}]}
%\{\vc{g}\in\Hk\ |\ |\inner{\phiwr}{\vc{g}} |\le \rho\cos\theta,
%\\&\qquad \qquad \qquad \qquad \qquad \qquad 
%|\inner{\phiwi}{\vc{g}} |\le \rho\sin\theta\}.
%\end{split}	\!\!\!\!\!
%\end{equation}
%For any $\theta\in[0,\frac{\pi}{2}]$, one can see that $\{\vc{g}\in\Hk\ |\ |\inner{\phiwr}{\vc{g}} |\le \rho\cos\theta,
%|\inner{\phiwi}{\vc{g}} |\le \rho\sin\theta\}$ is the intersection of the following closed half-spaces:
%\begin{equation}
%\begin{split}
%    \{\vc{g}\in\Hk\ &|\ \inner{\phiwr}{\vc{g}} \le \rho\cos\theta\},\\
%    \{\vc{g}\in\Hk\ &|\ \inner{\phiwr}{\vc{g}} \ge -\rho\cos\theta\},\\
%    \{\vc{g}\in\Hk\ &|\ \inner{\phiwr}{\vc{g}} \le \rho\sin\theta\},\\
%    \{\vc{g}\in\Hk\ &|\ \inner{\phiwr}{\vc{g}} \ge -\rho\sin\theta\}.
%\end{split}    
%\end{equation}
%Since the intersection of closed half-spaces is a closed and convex set, we have that
%$\Gcal_{\kernel}(\rho,\{\omega\})$ is a closed and convex set. 
%This concludes the proof.
%\end{proof}
%================================
\subsection{Proof of Lemma~\ref{lem:Lui_bounded}} \label{sec:appendix_proof_lem_Lui_bounded}
Let assume $\Tbb=\Rbb_+$. For any $\vc{g}:=(g_s)_{s\in\Rbb_+}\in\Hk$, from the triangle inequality
and \emph{reproducing property}, one has
\begin{equation*}
|\Lu{t}(\vcg)| 
\le
\int_{\Rbb_+}\!\!|g_s| |u_{t-s}|\drm s
=
\int_{\Rbb_+}\!\!|\inner{\vcg}{\kernel_s} | |u_{t-s}|\drm s.
\end{equation*}
Then, due to the Cauchy-Schwartz inequality and since $\vc{u}\in\Lscr_\infty$, we have
\begin{equation}\label{eqn:Ltug_le_normg_otherterms}
\begin{split}
|\Lu{t}(\vcg)| 
&\le
\int_{\Rbb_+}\|\vcg\| \kernel(s,s)^{\frac{1}{2}} |u_{t-s}|\drm s
\\&\le \|\vc{g}\| \ 
\|\vc{u}\|_\infty
\int_{\Rbb_+}\kernel(s,s)^{\frac{1}{2}}\drm s, 
\end{split}	
\end{equation}
where we have used the fact that $\|\kernel_s\| =\kernel(s,s)^{\frac12}$, for any $s\in\Rbb_+$.
This shows that $(g_su_{t-s})_{s\in\Rbb_+}$ is absolutely integrable, for any $\vcg=(g_t)_{t\in\Rbb_+}\in\Hk$, and therefore $\Lu{i}$ is a well-defined linear map.
Moreover, from the definition of $\|\Lu{t}\|_{\Lcal(\Hk,\Rbb)}$ and \eqref{eqn:Ltug_le_normg_otherterms}, we have
\begin{equation*}
\|\Lu{t}\|_{\Lcal(\Hk,\Rbb)} 
\le
\|\vc{u}\|_\infty
\int_{\Rbb_+}\!\!\kernel(s,s)^{\frac{1}{2}}\drm s= \mu_0 \|\vc{u}\|_\infty. 
\end{equation*}
Therefore, $\Lu{t}$ is a bounded functional. Subsequently, due to Riesz's representation theorem, there exists a unique element $\phiu{t}=(\phiu{t,s})_{s\in\Rbb_+}\in \Hk$ such that $\Lu{t}(\vcg) =\inner{\phiu{t}}{\vcg}$, for any $\vcg\in\Hk$.
From reproducing property of kernel, we know that
$\phiu{t,s} = \inner{\phiu{t}}{\kernel_s}$, for any $s\in\Rbb_+$. 
Since $\inner{\phiu{t}}{\kernel_s} = \Lu{t}(\kernel_s)$ and due to the definition of $\Lu{t}$, we have
\begin{equation*}
\!\!\!
\phiu{t,s} = 
\Lu{t}(\kernel_s) 
= 
\!\!\int_{\Rbb_+}\!\!\!\kernel_s(\tau) u_{t-\tau}\drm \tau
= 
\!\!\int_{\Rbb_+}\!\!\!\kernel(s,\tau) u_{t-\tau}\drm \tau.
\end{equation*}
Similar arguments hold for $\Tbb\! =\! \Zbb_+$.
\qed
%================================
\subsection{Proof of Theorem~\ref{lem:bound_mu_n}} 
\label{sec:appendix_proof_bound_mu_n}
When $\Tbb=\Zbb_+$, from the upper bound of kernel, we have 
\begin{equation*}
\mu_n \le \sum_{t\in\Zbb_+}t^n\ \!\kernel(t,t)^{\frac{1}{2}} 
\le 
\int_{\Rbb_+}t^n(\gamma\expe^{-\beta t})^{\frac{1}{2}}\drm t,
\end{equation*} 
%\begin{equation}
%	\begin{array}{rcl}
%		\mu_n &\le& \sum_{t\in\Zbb_+}t^n\ \!\kernel(t,t)^{\frac{1}{2}} 
%		\\&
%		\le& 
%		\int_{\Rbb_+}t^n(\gamma\expe^{-\beta t})^{\frac{1}{2}}\drm t,
%	\end{array}
%\end{equation} 
where the second inequality is due to being $f(t):=\expe^{-\frac{1}{2}\beta t}$ a non-increasing function.
This inequality holds for the case of $\Tbb=\Rbb_+$ as well. 
Using change of variable $s=\frac{\beta}{2}t$, we have
\begin{equation*}
	\mu_n \le  
	\gamma^{\frac12}\frac{2^{n+1} }{\beta^{n+1}}
	\int_{\Rbb_+}s^n\expe^{-s}\drm s
	=
	\gamma^{\frac12}\frac{2^{n+1} }{\beta^{n+1}}
	\Gamma(n+1),
\end{equation*}
where the equality is due to the definition of Gamma function. 
The claim follows from $\Gamma(k+1)=k!$, for $k\in\Zbb_+$.
\qed
%================================
\subsection{Proof of Lemma~\ref{lem:gm_Lipschitz}}
\label{sec:appendix_proof_gm_Lipschitz}
Given $\vcg:=(g_t)_{t\in\Tbb}\in\Hk$, since $\Fwr(\vc{g})$ and $\Fwr(\vc{g})$ are functions of $\omega$, we define functions $\mgr,\mgi:\Omega_{\Tbb}\to \Rbb$ as 
\begin{equation*}
\begin{split}\label{eqn:def_mgr_mgi}
\mgr\!(\omega):=\Fwr\!(\vc{g})=\inner{\phiwr}{\vc{g}},
\\
\mgi\!(\omega):=\Fwi\!(\vc{g})=\inner{\phiwi}{\vc{g}},
%\\
\end{split}
\end{equation*}
for any $\omega\in\Omega_{\Tbb}$.
We know that $\ddomega \left(g_t\cos(\omega t)\right) = -t g_t \sin(\omega t)$
and
$\ddomega \left(g_t\sin(\omega t)\right) = t g_t\cos(\omega t)$.
From the reproducing property and the Cauchy-Schwartz inequality, we have
\begin{equation}\label{eqn:ddw_gtcoswt_bound}
\begin{split}
\Big|&\ddomega (g_t\cos(\omega t))\Big| 
= 
|t g_t\sin(\omega t)|
\le  
t|g_t| = t|\inner{\kernel_t}{\vcg}|
\\&\le  
t\|\kernel_t\|\|\vcg\|
=
t\inner{\kernel_t}{\kernel_t}^{\frac{1}{2}}\|\vc{g}\|
=
t \kernel(t,t)^{\frac{1}{2}}\|\vc{g}\|.
\end{split}
\end{equation}
Similarly, one can see
\begin{equation}\label{eqn:ddw_gtsinwt_bound}
\begin{split}		
\Big|\ddomega (g_t\sin(\omega t))\Big| 
\le t\ \kernel(t,t)^{\frac{1}{2}}\|\vc{g}\|.
\end{split}
\end{equation}
Due to
$\mu_1<\infty$ and the Weierstrass M-test theorem, we have the uniform and absolute convergence for the series 
$\sum_{t\in\Zbb_+}\ddomega (g_t\cos(\omega t))$
and 
$\sum_{t\in\Zbb_+}\ddomega (g_t\sin(\omega t))$, when $\Tbb=\Zbb_+$, as well as for the improper integrals
$\int_{\Rbb_+}\ddomega (g_t\cos(\omega t))\drm t$
and 
$\int_{\Rbb_+}\ddomega (g_t\sin(\omega t))\drm t$, when $\Tbb=\Rbb_+$.
Consequently, from the definition of $\mgr,\mgi$, it follows that
\begin{equation*}
\begin{split}
&
\ddomega \mgr(\omega) =  
\ddomega \sum_{t\in\Zbb_+}\!g_t\cos(\omega t) 
=
-\!\!
\sum_{t\in\Zbb_+}\!tg_t\sin(\omega t),\\
&
\ddomega \mgi(\omega) = 
\ddomega \sum_{t\in\Zbb_+}\!g_t\sin(\omega t)
=
\sum_{t\in\Zbb_+}\!t g_t\cos(\omega t),
\end{split}
\end{equation*}
when $\Tbb=\Zbb_+$, and 
\begin{equation*}
\begin{split}
&
\ddomega \mgr(\omega) =  
\ddomega  \int_{\Rbb_+}\!\! g_t\cos(\omega t)\drm t 
=-\!
\int_{\Rbb_+}\!\! tg_t\sin(\omega t)\drm t,\\
&
\ddomega \mgi(\omega) = 
\ddomega \int_{\Rbb_+} \!\! g_t\sin(\omega t)\drm t
=
\int_{\Rbb_+}\!\! t g_t\cos(\omega t)\drm t,
\end{split}
\end{equation*}
when $\Tbb=\Rbb_+$. 
Subsequently, from \eqref{eqn:ddw_gtcoswt_bound}, \eqref{eqn:ddw_gtsinwt_bound}, the triangle inequality and the definition of $\mu_1$ in \eqref{eqn:mu_n}, it follows
%$|\ddomega \mgr(\omega)|\le \mu_1 \|\vc{g}\|,|\ddomega \mgi(\omega)|\le \mu_1 \|\vc{g}\|$. 
\begin{equation*}
|\ddomega \mgr(\omega)|\le \mu_1 \|\vc{g}\|, 
\ 
|\ddomega \mgi(\omega)|\le \mu_1 \|\vc{g}\|.
\end{equation*}
From 
$\mg(\omega) = (\mgr(\omega))^2 + (\mgi(\omega))^2$, it follows that $\mg$ is a differentiable function with derivative given as following
\begin{equation}\label{eqn:ddw_mg_expansion}\!\!\!\!
\ddomega\mg(\omega)
%=\ddomega(\Fwr(\vc{g})^2+\Fwi(\vc{g})^2)
\!=\!
2\Fwr(\vc{g})\ddomega\mgr(\omega)
\!+\!
2\Fwi(\vc{g})\ddomega\mgi(\omega).
\end{equation}
Due to Lemma \ref{lem:Fw}, we know that
% $|\Fwr(\vc{g})|\le \mu_0 \|\vc{g}\|$ 
% and 
% $|\Fwi(\vc{g})|\le \mu_0 \|\vc{g}\|$. 
\begin{equation*}
\begin{split}
|\Fwr(\vc{g})|&
\le     
\|\Fwr\|_{\Lcal(\Hk,\Rbb)}\|\vc{g}\|
\le 
\mu_0 \|\vc{g}\|,  \\
|\Fwi(\vc{g})|
&
\le     
\|\Fwi\|_{\Lcal(\Hk,\Rbb)}\|\vc{g}\|
\le \mu_0
\|\vc{g}\|.  
\end{split}
\end{equation*}
Subsequently, from \eqref{eqn:ddw_mg_expansion}, we have
\begin{equation*}
\begin{split}
|\ddomega&\mg(\omega)|
\!\le\!
2\Big[\!|\Fwr(\vc{g})||\ddomega\mgr\!(\omega)|
\!+\!
|\Fwi(\vc{g})||\ddomega\mgi\!(\omega)|\Big]
\\&
\le 
2\big[
\mu_0\|\vc{g}\|\ \!\mu_1\|\vc{g}\|
+
\mu_0\|\vc{g}\|\ \!\mu_1\|\vc{g}\| 
\big]
= 
4\mu_0\mu_1\|\vc{g}\|^2.
\end{split}
\end{equation*}
Now, let $\omega_1,\omega_2\in [0,\pi]$. 
Without loss of generality, we assume $\omega_2>\omega_1$. 
Consequently, we have
\begin{equation*}
\begin{split}
|\mg(\omega_2)&-\mg(\omega_1)|
 =
|\int_{\omega_1}^{\omega_2}
\ddomega\mg(\omega)  \domega|
\\& \le 
\int_{\omega_1}^{\omega_2}|\ddomega\mg(\omega)|\domega
\le
4\mu_0\mu_1\|\vc{g}\|^2
|\omega_2-\omega_1|.
\end{split}
\end{equation*}
In other words, we have $|\mg(\omega_2)-\mg(\omega_1)|\le L_{\vc{g}}|\omega_2-\omega_1|$
where 
$L_{\vc{g}}=4\mu_0\mu_1\|\vc{g}\|^2$.
\qed
%================================
\subsection{Proof of Theorem~\ref{thm:TC_assumption}}
\label{sec:appendix_proof_TC_assumption}
We know that $\max(s,t)=\frac12\big(s+t+|s-t|\big)$. Accordingly, using change of variable $\tau=s-t$, it follows that
\begin{equation}\label{eqn:proof_intint_k}
\int_{\Rbb_+}\!\int_{\Rbb_+}\!\!
\kernel(s,t)\expe^{-\Jimage\omega(s-t)}\drm s\drm t
=\!
\int_{\Rbb_+}\!\expe^{-\beta t}\!\!
\int_{-t}^{\infty}\!\!
\expe^{-\beta\tau^+-\Jimage\omega\tau}\drm \tau \drm t,
\end{equation}
where $\tau^+:=\frac12\big(\tau+|\tau|\big)$.
Also,  we have
\begin{equation}\label{eqn:proof_inner_int}
\begin{split}
&\!\!\int_{-t}^{\infty}
\expe^{-\beta\tau^+-\Jimage\omega\tau}\drm \tau
=
\int_{-t}^{0}\!
\expe^{-\Jimage\omega\tau}\drm \tau
+
\int_{0}^{\infty}\!
\expe^{-\beta\tau-\Jimage\omega\tau}\drm \tau
\\&=
\Big[-\frac{1}{\Jimage \omega}
+\frac{\expe^{\Jimage \omega t}}{\Jimage \omega}
\Big]
+\frac{1}{\beta +\Jimage \omega}
=
%\frac{(\beta +\Jimage \omega)\expe^{\Jimage \omega t}-\beta}{\Jimage\omega(\beta +\Jimage \omega)}.
\frac{\expe^{\Jimage \omega t}}{\Jimage \omega}
-\frac{\beta}
{\Jimage\omega(\beta +\Jimage \omega)}.
\end{split}		
\end{equation}
Subsequently, by replacing the right-hand side of \eqref{eqn:proof_inner_int} in  \eqref{eqn:proof_intint_k} and simplifying the integrals, we obtain
\begin{equation*}\label{eqn:proof_intint_k_2}
\begin{split}
\int_{\Rbb_+}&\int_{\Rbb_+}
\kernel(s,t)\expe^{-\Jimage\omega(s-t)}\drm s\drm t
\\&=-\frac{1}{\Jimage\omega(-\beta +\Jimage \omega)}
-\frac{1}{\Jimage\omega(\beta +\Jimage \omega)}
=\frac{2}{\beta^2 +\omega^2}.
\end{split}
\end{equation*}
This concludes the proof.
\qed
%================================
\subsection{Proof of Theorem~\ref{thm:tightness}}
\label{sec:appendix_proof_tightness}
Let $\Vk:=\Rbb^{\nD}\times\Hk$ be the Hilbert space endowed with the inner product defined as
\begin{equation}
	%\inner{\vc{v}_1}{\vc{v}_2}_{\Vk}
	\inner{(\vc{x}_1,\vc{g}_1)}{(\vc{x}_2,\vc{g}_2)}_{\Vk}
	:=
	\vc{x}_1^{\tr}
	\vc{x}_2
	+
	\lambda
	\inner{\vc{g}_1}{\vc{g}_2}_{\Hk},
\end{equation}
for all $(\vc{x}_1,\vc{g}_1)$ and $(\vc{x}_2,\vc{g}_2)$  
in $\Vk$.
Given $\Omega\subseteq \Omega_{\Tbb}$ and $\rho\in \Rbb_+$, define set $\Uscr_{\kernel}(\rho,\Omega)\subseteq \Vk $ as 
\begin{equation*}
\Uscr_{\kernel}(\rho,\Omega) :=
\Big\{(\vc{x},\vc{g})\Big| 
\Lu{}(\vc{g})-\vc{x} = \vc{y},
|\Fw(\vc{g})|\le \rho, \forall \omega \in \Omega\Big\}, 
\end{equation*}
where $\vcy:=[y_t]_{t\in\Tscr}\in\Rbb^{\nD}$.
Let 
$\Uscr_{\epsilon}:=\Uscr_{\kernel}(1-\epsilon,[0,\pi])$
and $\Uscr_{\Pscr}:=\Uscr_{\kernel}(1-\epsilon,\Pscr)$.
Accordingly, replacing $\eta$ with $\rho$, \eqref{eqn:opt_3} is equivalent to the following optimization problem
\begin{equation}\label{eqn:opt_general_V}
	\minOp_{(\vc{x},\vc{g})\in\Uscr_{\kernel}(\rho,\Omega)}
	\|(\vc{x},\vc{g})\|_{\Vk}^2,
\end{equation}
and therefore, it is a convex optimization with the same unique solution.
Let $\veps:=(\xeps, \geps)$ be the solution for \eqref{eqn:opt_general_V} for $\Uscr_{\epsilon}$ for $\epsilon\in[0,1)$. When $\epsilon =0$, we simply write 
$\vstar:= (\xstar,\gstar)$.
Also, let $\vstareps:=(\xstareps,\gstareps)$ is the solution of \eqref{eqn:opt_general_V}
for $\Uscr_{\Pscr}$.
These notations is consistent with our previous ones due to the equivalency of \eqref{eqn:opt_3} and  \eqref{eqn:opt_general_V} and the uniqueness of the solution.
One can easily see that $\Uscr_{\epsilon}\subseteq\Uscr_{\Pscr}$.
Also, based on the arguments provided in the proof of Theorem \ref{thm:opt_unique_convex_eps}, we know that $\Uscr_{\Pscr}\subseteq\Uscr_0$.
Accordingly, we have
\begin{equation*}
\begin{split}	
	\sup_{\vc{w}\in\Uscr_\Pscr}\min_{\vc{v}\in\Uscr_{\epsilon}} \|\vc{v}-\vc{w}\|_{\Vk}
	&\le
	\sup_{\vc{w}\in\Uscr_0}\min_{\vc{v}\in\Uscr_{\epsilon}} \|\vc{v}-\vc{w}\|_{\Vk},
\\
	\sup_{\vc{w}\in\Uscr_0}	\min_{\vc{v}\in\Uscr_\Pscr} \|\vc{v}-\vc{w}\|_{\Vk}
	&\le
	\sup_{\vc{w}\in\Uscr_0}\min_{\vc{v}\in\Uscr_{\epsilon}} \|\vc{v}-\vc{w}\|_{\Vk}.
\end{split}
\end{equation*}
%\begin{equation}
%\begin{split}
%\mathrm{d}_{H}(\Uscr_0,\Uscr_{\Pscr})
%:=&
%\sup_{\vc{w}\in\Uscr_0}
%\min_{\vc{v}\in\Uscr_\Pscr} \|\vc{v}-\vc{w}\|_{\Vk}
%\\\le&
%\mathrm{d}_{H}(\Uscr_0,\Uscr_{\epsilon}) :=
%\sup_{\vc{w}\in\Uscr_0}\min_{\vc{v}\in\Uscr_{\epsilon}} \|\vc{v}-\vc{w}\|_{\Vk}
%\end{split}	
%\end{equation}
Let $\vcv_1\in\Vk$ be defined as
\begin{equation}\label{eqn:proj_vstar}
	\vc{v}_1 := \proj_{\Uscr_{\epsilon}}(\vstar) = \argmin_{\vc{v}\in\Uscr_{\epsilon}}
	\|\vc{v}-\vstar\|_{\Vk}^2.
\end{equation}
Since $\vc{v}_1\in \Uscr_{\epsilon}$ and $\Uscr_{\epsilon}\subseteq\Uscr_{0}$, due to the definition of $\veps$, we know that 
$\|\veps\|_{\Vk} \le \|\vc{v}_1\|_{\Vk}$.
Therefore, from triangle inequality and \eqref{eqn:proj_vstar}, we have
\begin{equation*}
	\begin{split}
		\|\veps\|_{\Vk}
		&\le 
		\|\vc{v}_1\|_{\Vk}
		\\&
		\le
		\|\vc{v}_1-\vstar\|_{\Vk} + \|\vstar\|_{\Vk}
		\\&
		=
		\min_{\vc{v}\in\Uscr_{\epsilon}}
		\|\vc{v}-\vstar\|_{\Vk} + \|\vstar\|_{\Vk}
		\\&
		\le 
		\sup_{\vc{w}\in\Uscr_0}\min_{\vc{v}\in\Uscr_{\epsilon}} \|\vc{v}-\vc{w}\|_{\Vk}
		+ 
		\|\vstar\|_{\Vk}.
		%\\&=
		%\mathrm{d}_{H}(\Uscr_0,\Uscr_{\epsilon}) + \|\vstar\|_{\Vk}.
	\end{split}
\end{equation*}
Consequently, it follows that
\begin{equation*}
	0 
	\le \|\veps\|_{\Vk}-\|\vstar\|_{\Vk}
	\le
	\sup_{\vc{w}\in\Uscr_0}\min_{\vc{v}\in\Uscr_{\epsilon}} \|\vc{v}-\vc{w}\|_{\Vk}.
	%\mathrm{d}_{H}(\Uscr_0,\Uscr_{\epsilon}).
\end{equation*}
Due to the convexity of $\Uscr_{0}$ and $\Uscr_{\epsilon}\subseteq\Uscr_{0}$, we know that $\frac{1}{2}(\veps+\vstar)\in\Uscr_{0}$,  and therefore, we have 
$\|\vstar\|_{\Vk} \le\frac{1}{2}\|\veps+\vstar\|_{\Vk}$.
%\begin{equation}
%	\|\vstar\|_{\Vk} \le
%	\|\frac{1}{2}(\veps+\vstar)\|
%	=\frac{1}{2}\|\veps+\vstar\|_{\Vk}.	
%\end{equation}
Subsequently, one can see that
\begin{equation*}
	\begin{split}
		\|\veps&-\vstar\|_{\Vk}^2
		= 2\|\veps\|_{\Vk}^2+ 2\|\vstar\|_{\Vk}^2 -\|\veps+\vstar\|_{\Vk}^2
		\\&\le
		2\|\veps\|_{\Vk}^2+ 2\|\vstar\|_{\Vk}^2 -4\|\vstar\|_{\Vk}^2 
		\\&= 2\|\veps\|_{\Vk}^2 -2\|\vstar\|_{\Vk}^2 
		\\&=2
		(\|\vstareps\|_{\Vk}-\|\vstar\|_{\Vk})
		(\|\vstareps\|_{\Vk}+\|\vstar\|_{\Vk})
		\\&\le 
		2
		%\mathrm{d}_{H}(\Uscr_0,\Uscr_{\epsilon})
		(\sup_{\vc{w}\in\Uscr_0}\min_{\vc{v}\in\Uscr_{\epsilon}} \|\vc{v}-\vc{w}\|_{\Vk})
		(\|\vstareps\|_{\Vk}+\|\vstar\|_{\Vk}).
		%\\&\le4
		%\mathrm{d}_{H}(\Uscr_0,\Uscr_{\epsilon})
		%(\sup_{\vc{w}\in\Uscr_0}\min_{\vc{v}\in\Uscr_{\epsilon}} \|\vc{v}-\vc{w}\|_{\Vk})
		%\|\vc{y}\|.
	\end{split}
\end{equation*}
Since $\zero$ belongs to $\Uscr_{\epsilon}$, $\Uscr_{\Pscr}$ and $\Uscr_{0}$, one has $\|\vcy\|\le\|\geps\|$, 
$\|\vcy\|\le\|\gstareps\|$
and
$\|\vcy\|\le\|\gstar\|$.
Therefore, we have
\begin{equation*}
	\|\veps-\vstar\|_{\Vk}^2
	\le
	4(\sup_{\vc{w}\in\Uscr_0}\min_{\vc{v}\in\Uscr_{\epsilon}} \|\vc{v}-\vc{w}\|_{\Vk})
	\|\vc{y}\|.
\end{equation*}
Similarly, one can show the following inequalities
\begin{align}	
	\|\veps-\vstareps\|_{\Vk}^2
	&\le 
	4
	(\sup_{\vc{w}\in\Uscr_{\Pscr}}\min_{\vc{v}\in\Uscr_{\epsilon}} \|\vc{v}-\vc{w}\|_{\Vk})
	%\mathrm{d}_{H}(\Uscr_\Pscr,\Uscr_{\epsilon})
	\|\vc{y}\|,
	\\
	\label{eqn:vstareps_vstar_dH}
	\|\vstareps-\vstar\|_{\Vk}^2
	&\le 
	4
	%\mathrm{d}_{H}(\Uscr_0,\Uscr_{\Pscr})
	(\sup_{\vc{w}\in\Uscr_0}\min_{\vc{v}\in\Uscr_{\Pscr}} \|\vc{v}-\vc{w}\|_{\Vk})
	\|\vc{y}\|.
\end{align}
Now, let $\vcv = (\vcx,\vcg)$ be an arbitrary element of $\Uscr_0$. Then, we know that $(1-\epsilon)\vcg\in\Gscr(1-\epsilon,\Omega_{\Tbb})$. Also, we have
\begin{equation*}
	\Lu{}((1-\epsilon)\vc{g})-(\vcx-\epsilon\Lu{}(\vc{g}))=
	\Lu{}(\vc{g})-\vcx=\vc{y}.
\end{equation*}
Therefore $\vc{w}:=(\vcx-\epsilon\Lu{}(\vc{g}), (1-\epsilon)\vc{g})$ is an element of $\Uscr_0$. Moreover, we have
\begin{equation*}
		\|\vcv-\vc{w}\|_{\Vk}^2 \!=\! 
		\epsilon(\|\Lu{}(\vc{g})\|^2+\|\vcg\|^2)^{\frac{1}{2}} 
		\le
		\epsilon(\|\Lu{}\|^2+1)^{\frac{1}{2}} \|\vcg\|^2. 
\end{equation*}
Accordingly, one has
\begin{equation}
	\label{eqn:supinfU0Ueps_eps_bound}	
	\begin{split}
		%\mathrm{d}_{H}(\Uscr_0,\Uscr_{\epsilon})
		%&=
		\sup_{\vc{w}\in\Uscr_0}\min_{\vc{v}\in\Uscr_{\epsilon}} \|\vc{v}-\vc{w}\|_{\Vk}
		&\le 
		\sup_{(\vcx,\vcg)\in\Uscr_0} \epsilon(\|\Lu{}\|^2+1)^{\frac{1}{2}} \|\vcg\|^2
		\\&\le 
		\epsilon(\|\Lu{}\|^2+1)^{\frac{1}{2}} \|\vc{y}\|.  
	\end{split}
\end{equation}
Since $\Uscr_{\epsilon} \subseteq \Uscr_{\Pscr} \subseteq \Uscr_{0}$, we know that
\begin{equation*}
	\sup_{\vc{w}\in\Uscr_0}\min_{\vc{v}\in\Uscr_{\Pscr}} \|\vc{v}-\vc{w}\|_{\Vk}
	\le 
	\sup_{\vc{w}\in\Uscr_0}\min_{\vc{v}\in\Uscr_{\epsilon}} \|\vc{v}-\vc{w}\|_{\Vk}.
\end{equation*}
Subsequently, due to \eqref{eqn:vstareps_vstar_dH}
and 
\eqref{eqn:supinfU0Ueps_eps_bound},
we have
\begin{equation}
	\|\vstareps-\vstar\|_{\Vk}^2
	\le 	
	\epsilon(\|\Lu{}\|^2+1)^{\frac{1}{2}} \|\vc{y}\|.  
\end{equation}
Consequently, it follows that
\begin{equation*}
	\begin{split}
		\|\gstareps-\gstar\|^2
		\le 
		\|\vstareps-\vstar\|_{\Vk}^2
		\le 
		4\epsilon(\|\Lu{}\|^2+1)^{\frac{1}{2}}
		\|\vc{y}\|^2.
	\end{split}
\end{equation*}
This also shows that $\lim_{\epsilon \to \zero}\gstareps = \gstar$ in $\Hk$.
Note that, for any $t\in\Tbb$, we have $\gstarepst{t} = \inner{\gstareps}{\kernel_t}_{\Hk}$ 
and
$\gstart{t} = \inner{\gstar}{\kernel_t}_{\Hk}$.
Therefore, from the Cauchy-Schwartz inequality and the reproducing property, one has
\begin{equation*}
		|\gstarepst{t}-\gstart{t}| =
		|\inner{\gstareps-\gstar}{\kernel_t}|
		\le \|\gstareps-\gstar\| \|\kernel_t\|.
\end{equation*} 
Accordingly, from $\|\kernel_t\|=\kernel(t,t)^{\frac12}$, it follows that
\begin{equation*}
	|\gstarepst{t}-\gstart{t}| \le 
	\|\gstareps-\gstar\|  \sup_{t\in\Tbb}
	\kernel(t,t)^{\frac12}.
\end{equation*} 
Since $\sup_{t\in\Tbb}
\kernel(t,t)^{\frac12}\!\!<\!\!\infty$ and 
$\lim_{\epsilon \to \zero}\|\gstareps-\gstar\| =0$,
we have
$\gstarepst{t}\tolimOp^{\epsilon\to 0}\gstart{t}$, uniformly in $\Tbb$, and 
proof concludes.
\qed
%================================
%================================
%================================
%================================
%================================
%================================
%\iffalse
\iftrue
\subsection{Setting Up the Optimization Problem for TC Kernel}\label{sec:appendix_setting_up}
Let $\Tbb=\Rbb_+$ and $\kernel(s,t) := \expe^{-\beta\max(s,t)}$, for $s,t\in\Rbb_+$, where  $\beta$ is a positive real scalar.
Accordingly, we define $\varphi_{\omega,t}$ as  $\varphi_{\omega,t}:=\int_{\Rbb_+}\kernel(t,s)\expe^{-\Jimage \omega s}\drm s$, for $\omega,t\in\Rbb_+$.
\begin{lemma}\label{lem:phi_wt}
	For any  $\omega,t\in\Rbb_+$, we have
	\begin{equation}
		\varphi_{\omega,t} 
		= 
		\begin{cases}
			\expe^{-\beta t}\Big[ \frac{1-\expe^{-\Jimage \omega t}}{\Jimage \omega}
			+	\frac{\expe^{-\Jimage \omega t}}{\beta+ \Jimage \omega}\Big], 
			& \text{ if } \omega\ne 0,\\
			\expe^{-\beta t}(t + \frac{1}{\beta}),& \text{ if } \omega=0.\\	
		\end{cases}
	\end{equation}
\end{lemma}
\begin{proof}
	When $\omega \ne 0$, we have
	\begin{equation*}
		\begin{split}
			&\varphi_{\omega,t}  =
			\int_{0}^{\infty}\!\kernel(t,s)\expe^{-\Jimage \omega s}\drm s 
			\\&\!= 
			%= 
			\int_{0}^{t}\!\expe^{-\beta\max(t,s)}\expe^{-\Jimage \omega s}\drm s 
			+
			\int_{t}^{\infty}\!\expe^{-\beta\max(t,s)}\expe^{-\Jimage \omega s}\drm s 
			\\ &\!=
			\int_{0}^{t}\!\expe^{-\beta t}\expe^{-\Jimage \omega s}\drm s 
			+
			\int_{t}^{\infty}\!\expe^{-\beta s}\expe^{-\Jimage \omega s}\drm s 
			\\ &\!=
			\expe^{-\beta t}\int_{0}^{t}\!\expe^{-\Jimage \omega s}\drm s 
			+
			\int_{t}^{\infty}\!\expe^{-(\beta +\Jimage \omega) s}\drm s 
			\\ &\!=
			\expe^{-\beta t}\frac{1-\expe^{-\Jimage \omega t}}{\Jimage \omega}
			\!+\!
			\frac{\expe^{-(\beta +\Jimage \omega) t}}{\beta+ \Jimage \omega}
			\\&\! 
			=
			\expe^{-\beta t}\bigg[\frac{1-\expe^{-\Jimage \omega t}}{\Jimage \omega}
			\!+\!
			\frac{\expe^{-\Jimage \omega t}}{\beta+ \Jimage \omega}\bigg].
			%\\&\! 
			%=
			%\frac{\expe^{-\beta t}}{\Jimage \omega} 
			%+ 
			%\big{(}\frac{1}{\beta+ \Jimage \omega}-\frac{1}{\Jimage \omega}\big{)}\expe^{-(\beta +\Jimage \omega) t}. 
		\end{split}
	\end{equation*}
	The case $\omega = 0$ follows from similar line of arguments.
\end{proof}
\begin{proposition}\label{prop:phir_wt_phii_wt}
	For any $\omega,t\in\Rbb_+$, we have
	\begin{equation}\label{eqn:phir_wt}
		\begin{split}
			&\phir{\omega,t}= 
			\begin{cases}
				\expe^{-\beta t}\ \real{\frac{1-\expe^{-\Jimage \omega t}}{\Jimage \omega}
					+	\frac{\expe^{-\Jimage \omega t}}{\beta+ \Jimage \omega}}, 
				& \text{ if } \omega\ne 0,\\
				\expe^{-\beta t}(t + \frac{1}{\beta}),	
				& \text{ if } \omega=0,\\	
			\end{cases}
			\\&		
			\phii{\omega,t}
			= 
			\begin{cases}
				\expe^{-\beta t}\ \imag{\frac{1-\expe^{-\Jimage \omega t}}{\Jimage \omega}
					+	\frac{\expe^{-\Jimage \omega t}}{\beta+ \Jimage \omega}}, 
				& \text{ if } \omega\ne 0,\\
				0,& \text{ if } \omega=0.\\	
			\end{cases}
		\end{split}
	\end{equation}	
\end{proposition}
\begin{proof}
	From Lemma \ref{lem:Fw} and definition of $\phir{\omega,t}$, we have
	\begin{equation*}
		\begin{split}
			\phir{\omega,t} 
			%= \inner{\phiwr}{\kernel_t}_{\Hk}
			&= \int_{\Rbb_+}\!\kernel(t,s)\cos(\omega s)\drm s
			\\&
			= \real{\int_{\Rbb_+}\!\kernel(t,s)\expe^{-\Jimage \omega s}\drm s}
			= \real{\varphi_{\omega,t}}.
		\end{split}
	\end{equation*}
	Similarly, one has $\phii{\omega,t} =\imag{\varphi_{\omega,t}}$. Following these, the claim concluded form Lemma \ref{lem:phi_wt}.
\end{proof}
For any $\omega_1,\omega_2\in\Rbb_+$,  define $\zeta_r(\omega_1,\omega_2)$ and $\zeta_i(\omega_1,\omega_2)$  as
% \begin{equation*}\label{eqn:zeta_r_def}
% \zeta_{\text{r}}(\omega_1,\omega_2)\!:=\!\!
% \int_{\Rbb_+}\!\!\!\!\!\phir{\omega_2,t}\expe^{-\Jimage\omega_1 t}\drm t,
% \ 
% \zeta_{\text{i}}(\omega_1,\omega_2)\!:=\!\!
% \int_{\Rbb_+}\!\!\!\!\!\phii{\omega_2,t}\expe^{-\Jimage\omega_1 t}\drm t.
% \end{equation*}
\begin{equation}\label{eqn:zeta_r_def}
	\begin{split}
		\zeta_{\text{r}}(\omega_1,\omega_2)&:=
		\int_{\Rbb_+}\phir{\omega_2,t}\ \!\expe^{-\Jimage\omega_1 t}\drm t,
		%		\\&=
		%		\int_{0}^{\infty}\real{\sum_{s=0}^{\infty}\kernel(t,s)\expe^{-\Jimage \omega_2 s}}\ \!\expe^{-\Jimage\omega_1 t}\drm t,
		%\end{split}    
		%\end{equation}
		%and 
		%\begin{equation}
		%\begin{split}
		\\
		\zeta_{\text{i}}(\omega_1,\omega_2)
		&:=
		\int_{\Rbb_+}\phii{\omega_2,t}\ \!\expe^{-\Jimage\omega_1 t}\drm t.
		%\\&=
		%\int_{0}^{\infty}\imag{\sum_{s=0}^{\infty}\kernel(t,s)\expe^{-\Jimage \omega_2 s}}\ \!\expe^{-\Jimage\omega_1 t}\drm t.
	\end{split}    
\end{equation}
\begin{lemma}\label{lem:zeta_r_zeta_i}
Let $\omega_1,\omega_2\in \Rbb_+$. If $\omega_2\ne 0$, then  we have
\begin{equation}\label{eqn:zeta_r}
\begin{split}
&\zeta_{\text{r}}(\omega_1,\omega_2)=
\frac{\beta}{2}\bigg{[}
\frac{1}{(\omega_2^2-\Jimage \omega_2\beta)(\beta + \Jimage\omega_1 +\Jimage \omega_2)}
\\&\quad  +
\frac{1}{(\omega_2^2+\Jimage \omega_2\beta)(\beta + \Jimage\omega_1 -\Jimage \omega_2)}\bigg{]},\\
\end{split}    
\end{equation}% and 
\begin{equation}\label{eqn:zeta_i}
\begin{split}
&\zeta_{\text{i}}(\omega_1 ,\omega_2)=
\frac{\beta}{2\Jimage}\bigg{[}
\frac{1}{(\omega_2^2-\Jimage \omega_2\beta)(\beta + \Jimage\omega_1 +\Jimage \omega_2)}
\\&\quad  
- 
\frac{1}{(\omega_2^2+\Jimage \omega_2\beta)(\beta + \Jimage\omega_1 -\Jimage \omega_2)}\bigg{]}
%\\&\quad  
-\frac{1}{\omega_2}
\frac{1}{(\beta + \Jimage\omega_1 )}.
\end{split}    
\end{equation}
If $\omega_2= 0$, then $\zeta_{\text{i}}(\omega_1 ,\omega_2)=0$ and
\begin{equation}
\zeta_{\text{r}}(\omega_1 ,\omega_2)=
\frac{2\beta +\Jimage\omega_1}
{\beta(\beta +\Jimage\omega_1)^2}.
\end{equation}
\end{lemma}
\begin{proof}
Note that for any $z\in\Cbb$, we have $\real{z} = \frac12(z+z^*)$, where $z^*$ denotes the complex conjugate of $z$. Accordingly, when $\omega_2\ne 0$, from \eqref{eqn:phir_wt} and the definition of $\zeta_{\text{r}}$, we have
\begin{equation*}
\begin{split}
&\!\!\zeta_{\text{r}}(\omega_1,\omega_2)\!=\!
\frac{1}{2}\!
\int_{\Rbb_+}\!\!\!
\Bigg{[}\!\bigg{[}
\frac{\expe^{-\beta t}}{\Jimage \omega_2} 
\!+\! 
\big{(}\frac{1}{\beta+ \Jimage \omega_2}\!-\!\frac{1}{\Jimage \omega_2}\big{)}\expe^{-(\beta +\Jimage \omega_2) t}
\bigg{]}
\\&\quad  
+
\bigg{[}\!-\!
\frac{\expe^{-\beta t}}{\Jimage \omega_2} 
\!+\! 
\big{(}\frac{1}{\beta- \Jimage \omega_2}\!+\!\frac{1}{\Jimage \omega_2}\big{)}\expe^{-(\beta 	-\Jimage \omega_2) t}
\bigg{]}\!\Bigg{]}
\expe^{-\Jimage\omega_1 t}\drm t
\\&=
\frac{1}{2}
\int_{\Rbb_+}\!
\bigg{[}
\big{(}\frac{1}{\beta+ \Jimage \omega_2}-\frac{1}{\Jimage \omega_2}\big{)}\expe^{-(\beta + \Jimage\omega_1 +\Jimage \omega_2) t}
\\&\quad + 
\big{(}\frac{1}{\beta- \Jimage \omega_2}+\frac{1}{\Jimage \omega_2}\big{)}\expe^{-(\beta+\Jimage\omega_1 	-\Jimage \omega_2) t}\bigg{]}
\drm t.
\end{split}    
\end{equation*}
Accordingly, we have
\begin{equation*}
\begin{split}
\zeta_{\text{r}}(\omega_1,\omega_2)=&
\frac{1}{2}
\big{(}\frac{1}{\beta+ \Jimage \omega_2}-\frac{1}{\Jimage \omega_2}\big{)}
\frac{1}{\beta + \Jimage\omega_1 +\Jimage \omega_2}
\\&\quad  
+
\frac{1}{2}\big{(}\frac{1}{\beta- \Jimage \omega_2}+\frac{1}{\Jimage \omega_2}\big{)}
\frac{1}{\beta + \Jimage\omega_1 -\Jimage \omega_2}.
\end{split}    
\end{equation*}
Subsequently, we have \eqref{eqn:zeta_r}. 
The claim for the other cases can be proved based on similar line of arguments.
\end{proof}
\begin{proposition}\label{prop:inner_products_phi_rr_ri_ir_ii}
For any $\omega_1,\omega_2\in\Rbb_+$, we have
\begin{equation*}
\begin{split}
&\inner{\phiwrone\!}{\!\phiwrtwo} \!=\! \real{\zeta_{\text{r}}(\omega_1\!,\omega_2)},
\ 
\inner{\phiwrone\!}{\!\phiwitwo}  \!=\! \real{\zeta_{\text{i}}(\omega_1\!,\omega_2)},
\\
&\inner{\phiwione\!}{\!\phiwrtwo} \!=\! \imag{\zeta_{\text{r}}(\omega_1\!,\omega_2)},
\
\inner{\phiwione\!}{\!\phiwitwo}  \!=\! \imag{\zeta_{\text{i}}(\omega_1\!,\omega_2)}. 
\end{split}
\end{equation*}
\end{proposition}
\begin{proof}
	This is a straightforward result from \eqref{eqn:Fwr_Fwi_inner_product}, \eqref{eqn:def_CT_Fw}, and the definition of $\zeta_r$ and $\zeta_i$. 
\end{proof}
\begin{proposition}\label{prop:zu}
Define $z_{\vc{u}}(\omega,\tau)$ as
\begin{equation}\label{eqn:zu}
\begin{split}		
\!\!\!\!\!\!
&z_{\vc{u}}(\omega,\tau)
\!:=\! %\\&
\begin{cases}
\int_{\Rbb_+}\!\!
\big{(}\frac{\expe^{-\beta t}}{\Jimage \omega}
+
\frac{\beta\expe^{-(\beta+\Jimage\omega) t}}{\omega^2- \Jimage \omega\beta}
\big{)}
u_{\tau-t}\drm t,		\!\!\!\!\!\!				
& \text{ if } \omega\!\ne\! 0,\\
\int_{\Rbb_+}\!\!\expe^{-\beta t}(t +\frac{1}{\beta})  u_{\tau-t}\drm t,
\!\!\!\!\!\!
& \text{ if } \omega \!=\! 0.\\			
\end{cases}
\end{split}
\end{equation}
for $\omega,\tau\in\Rbb_+$.
Then, we have 
%$\inner{\phiwr}{\phiu{\tau}} =\real{z_{\vc{u}}(\beta,\omega,\tau)}$ and $\inner{\phiwi}{\phiu{\tau}} =\imag{z_{\vc{u}}(\beta,\omega,\tau)}$.
\begin{align}
&\inner{\phiwr}{\phiu{\tau}} =\real{z_{\vc{u}}(\beta,\omega,\tau)}, 
\label{eqn:inner_product_phir_w_phiu_tau}
\\&
\inner{\phiwi}{\phiu{\tau}} =\imag{z_{\vc{u}}(\beta,\omega,\tau)}.    	
\label{eqn:inner_product_phii_w_phiu_tau}
\end{align}
\end{proposition}
\begin{proof}
When $\omega\ne 0$, from \eqref{eqn:phir_wt}, \eqref{eqn:Lut_R} and Lemma \ref{lem:Lui_bounded}, we have
\begin{equation*}
\begin{split}
&\inner{\phiwr}{\phiu{\tau}}
= 
\int_{0}^{\infty}\!\phiwrk{t} u_{\tau-t}\drm t
\\&= 
\int_{0}^{\infty}\! \expe^{-\beta t}\ \real{\frac{1-\expe^{-\Jimage \omega t}}{\Jimage \omega}
+	\frac{\expe^{-\Jimage \omega t}}{\beta+ \Jimage \omega}} u_{\tau-t}\drm t
\\&= 
\real{\int_{0}^{\infty}\! 
\big{(}\frac{1}{\Jimage \omega}\expe^{-\beta t}
+
\frac{\beta}{\omega^2- \Jimage \omega\beta}
\expe^{-(\beta+\Jimage\omega) t}\big{)} u_{\tau-t}\drm t}.
\end{split}
\end{equation*}
Similarly, when $\omega=0$, one has
\begin{equation*}
\inner{\phiwr}{\phiu{\tau}} 
\!=\! 
\int_{\Rbb_+}\!\!\phiwrk{t} u_{\tau-t}\drm t
%\\&
\!=\! 
\int_{\Rbb_+}\!\!\expe^{-\beta t}(t +\frac{1}{\beta})  u_{\tau-t}\drm t. 
\end{equation*}
The other equality can be shown in a similar way.
\end{proof}
\begin{remark}
	When $\Tbb=\Zbb_+$, using similar lines of arguments, one can introduce  results analogous to Lemma \ref{lem:phi_wt}, Proposition \ref{prop:phir_wt_phii_wt}, Lemma \ref{lem:zeta_r_zeta_i}, Proposition \ref{prop:inner_products_phi_rr_ri_ir_ii} and  Proposition \ref{prop:zu}.
\end{remark}
\begin{remark}
	For the case of $\Tbb=\Zbb_+$, for  $z_{\vc{u}}(\omega,\tau)$,
	we have a finite summation in a similar form to \eqref{eqn:zu}, when the system is initially at rest, that is $u_t=0$, for $t<0$.
\end{remark}
The set of piecewise constant functions are dense in $\Lscr_p$, i.e., one can approximate precisely almost any signal of interest using a piecewise constant signal.
Motivated by this fact, let assume the input signal $\vc{u} = (u_t)_{t\in\Rbb_+}$ is given as a piecewise constant function defined as  
\begin{equation}\label{eqn:u_pw}
	u_t=\sum_{i=0}^{\nS-1} \xi_{i+1} \one_{[s_i,s_{i+1})}(t), \quad \forall t\in\Rbb_+,
\end{equation}
where $\nS\in\Nbb$, $\xi_i\in\Rbb$, for $i=0,\ldots,\nS$, %$i\in\{0,\ldots,\nS\}$, 
and $(s_0,s_1,\ldots,s_{\nS})$ is a finite increasing sequence in $\Rbb_+$ such that $s_0=0$.
\begin{proposition}
Let $\tau \in\Rbb_+$. Then, if $\omega\ne 0$, we have
\begin{equation}\label{eqn:zu_w_tau_wnz}
\begin{split}
z_{\vc{u}}&(\omega,\tau) = 
\sum_{i=0}^{\nS-1} \xi_{i+1}\bigg[
\frac{\expe^{-\beta\bar{s}_{i+1}(\tau)}-\expe^{-\beta\bar{s}_{i}(\tau)}}{\Jimage \omega\beta}
\\& \quad +
\frac{\beta\expe^{-(\beta+\Jimage\omega) \bar{s}_{i+1}(\tau)}-\beta\expe^{-(\beta+\Jimage\omega) \bar{s}_{i}(\tau)}}{(\omega^2- \Jimage \omega\beta)(\beta+\Jimage\omega)}
\bigg].
\end{split}
\end{equation}
and, if $\omega = 0$, one has
\begin{equation}\label{eqn:zu_w_tau_wz}
\begin{split}
&z_{\vc{u}}(\omega,\tau) 
= 
\sum_{i=0}^{\nS-1} 
\frac{\xi_{i+1}}{\beta^2}
\bigg[(\beta\bar{s}_{i+1}(\tau)\!+\!2)\expe^{-\beta\bar{s}_{i+1}(\tau)}
\\&\qquad -
(\beta\bar{s}_{i}(\tau)\!+\!2)\expe^{-\beta\bar{s}_{i}(\tau)}\bigg],
\end{split}
\end{equation}
where $\bar{s}_{i}(\tau):=\max(\tau-s_i,0)$, for $i=0,\ldots,\nS$.
\end{proposition}	
\begin{proof}
From \eqref{eqn:u_pw} and \eqref{eqn:zu}, if $\omega\ne 0$, we have
%\begin{equation*}
%\begin{split}
%z_{\vc{u}}(\omega,\tau) 
%\! &=\!\!\!\int_{\Rbb_+}\!\!\!\Big[\frac{\expe^{-\beta t}}{\Jimage \omega}
%\!+\!
%\frac{\beta\expe^{-(\beta+\Jimage\omega) t}}{\omega^2- \Jimage \omega\beta}
%\Big]\!\!
%%\\&\qquad  
%\sum_{i=0}^{\nS-1}\! \xi_{i+1} \one_{[s_i,s_{i+1})}(\tau\!-t)\drm t
%\\& = 
%\sum_{i=0}^{\nS-1} \xi_{i+1}
%\int_{\bar{s}_{i+1}(\tau)}^{\bar{s}_{i}(\tau)}\! \Big[\frac{\expe^{-\beta t}}{\Jimage \omega}
%+
%\frac{\beta \expe^{-(\beta+\Jimage\omega) t}}{\omega^2- \Jimage \omega\beta}
%\Big] \drm t,
%\end{split}
%\end{equation*}
\begin{equation*}
\begin{split}
z_{\vc{u}}(\omega,\tau) 
\! &=\!\!\!\int_{\Rbb_+}\!\!\!\Big[\frac{\expe^{-\beta t}}{\Jimage \omega}
\!+\!
\frac{\beta\expe^{-(\beta+\Jimage\omega) t}}{\omega^2- \Jimage \omega\beta}
\Big]\!\!
\sum_{i=0}^{\nS-1}\! \xi_{i+1} \one_{[s_i,s_{i+1})}(\tau\!-t)\drm t
\\& = 
\sum_{i=0}^{\nS-1} \xi_{i+1}
\int_{\bar{s}_{i+1}(\tau)}^{\bar{s}_{i}(\tau)}\! \frac{\expe^{-\beta t}}{\Jimage \omega} \drm t
+
\int_{\bar{s}_{i+1}(\tau)}^{\bar{s}_{i}(\tau)}\!
\frac{\beta \expe^{-(\beta+\Jimage\omega) t}}{\omega^2- \Jimage \omega\beta} \drm t,
\end{split}
\end{equation*}
which concludes \eqref{eqn:zu_w_tau_wnz} when we replace the integrals with their closed-form values. Similarly, one can show \eqref{eqn:zu_w_tau_wz}. 
\end{proof}
Note that using \eqref{eqn:zu_w_tau_wnz} and \eqref{eqn:zu_w_tau_wz}, the inner products in Proposition~\ref{prop:zu} can be obtained as a finite sum rather than an improper integral.

Define $\psi:\Rbb_+\times\Rbb_+\times\Rbb_+\to\Rbb$ as 
\begin{equation}\label{eqn:psi_def}
	\begin{split}
		&\psi(t,a,b):=
		\frac{1}{\beta}(\expe^{-\beta \max(\min(t,b),a)}-\expe^{-\beta b})
		\\&\qquad\qquad\qquad
		+
		\bigg[\max(\min(t,b),a)-a\bigg]\expe^{-\beta t}. 
	\end{split}
\end{equation} 
for any $a,b,t\in\Rbb_+$. 
\begin{lemma}\label{lem:psi}
	For any $a,b\in\Rbb_+$ such that $a\le b$, we have
	%$\psi(t,a,b) = \int_{a}^{b}\!\expe^{-\beta\max(t,s)}\drm s.$
	\begin{equation}
	\psi(t,a,b) = \int_{a}^{b}\!\expe^{-\beta\max(t,s)}\drm s. 
	\end{equation}	
\end{lemma}
\begin{proof}
	Considering three cases of $t<a$, $t>b$ and $t\in[a,b]$ and then evaluating  the integral $\int_{a}^{b}\!\expe^{-\beta\max(t,s)}\drm s$, one can obtain \eqref{eqn:psi_def} by comparing the resulting values.  
\end{proof}
\begin{proposition}
	For any $t\in\Rbb_+$, we have
	\begin{equation}
		\phiu{\tau,t} =
		\sum_{i=0}^{\nS-1} \xi_{i+1} \psi(t,\bar{s}_{i+1}(\tau),\bar{s}_{i}(\tau)).
	\end{equation}
	%\begin{equation}
	%\begin{split}
	%\phiu{\tau,t} &=
	%\sum_{i=0}^{\nS-1}  \frac{\xi_{i+1}}{\beta}\Big[\expe^{-\beta \max\big{(}\min(t,\bar{s}_{i}(\tau)),\bar{s}_{i+1}(\tau)\big{)}}-\expe^{-\beta\bar{s}_{i}(\tau)}\Big]
	%\\&+
	%\xi_{i+1}\bigg[\max\Big{(}\min(t,\bar{s}_{i}(\tau)),\bar{s}_{i+1}(\tau)\Big{)}-\bar{s}_{i+1}(\tau)\bigg]\expe^{-\beta t}.
	%\end{split}
	%\end{equation}
\end{proposition}
\begin{proof}
	From \eqref{eqn:u_pw} and Lemma \ref{lem:Lui_bounded}, we have
	\begin{equation}
		\begin{split}
			\phiu{\tau,t} 
			&=\inner{\phiu{\tau}}{\kernel_t}_{\Hk} 
			\\&= 
			\int_{\Rbb_+}\kernel(t,s)u(\tau-s)\drm s
			\\
			&= 
			\int_{\Rbb_+}\kernel(t,s)\sum_{i=0}^{\nS-1} \xi_{i+1} \one_{[s_i,s_{i+1})}(\tau-s)\drm s
			\\&= 
			\sum_{i=0}^{\nS-1} \xi_{i+1} \int_{\bar{s}_{i+1}(\tau)}^{\bar{s}_{i}(\tau)}\!\expe^{-\beta\max(t,s)}\drm s.
		\end{split}
	\end{equation}	
	Subsequently, the claim follows from Lemma \ref{lem:psi}.
\end{proof}	
\begin{lemma}\label{lem:nu}
	Define function 
	$\nu:\Rbb_+\times\Rbb_+\to\Rbb$ as 
	\begin{equation}\label{eqn:nu_def}
		\nu(x,y) := 
		\int_{0}^{x}\!
		\int_{0}^{y}\! 
		\kernel(s,t)
		\drm t \ \! \drm s.
	\end{equation}
	Then, we have
	\begin{equation}\label{eqn:nu_closed}
		\nu(x,y)=
		\frac{1}{\beta^2}\Big[2-2\expe^{-\beta \min(x,y)}
		-\beta\min(x,y)(\expe^{-\beta x}+\expe^{-\beta y})\Big].
	\end{equation}
\end{lemma}	
\begin{proof}
	Since $\kernel(s,t)=\kernel(t,s)$, without loss of generality, let assume $x\le y$. 
	Then, we have
	\begin{equation*}
		\begin{split}
			\nu(x,&y) = 
			\int_{0}^{x}\!
			\int_{0}^{y}\! 
			\kernel(s,t)
			\drm t \ \! \drm s
			\\&=
			\int_{0}^{x}\!
			\int_{0}^{y}\! 
			\expe^{-\beta \max(s,t)}
			\drm t \ \! \drm s
			\\&=
			\int_{0}^{x}\!
			\int_{0}^{s}\! 
			\expe^{-\beta \max(s,t)}
			\drm t \ \! \drm s
			+
			\int_{0}^{x}\!
			\int_{s}^{y}\! 
			\expe^{-\beta \max(s,t)}
			\drm t \ \! \drm s
			\\&=
			\int_{0}^{x}\!
			\int_{0}^{s}\! 
			\expe^{-\beta s}
			\drm t \ \! \drm s
			+
			\int_{0}^{x}\!
			\int_{s}^{y}\! 
			\expe^{-\beta t}
			\drm t \ \! \drm s
			\\&=
			\int_{0}^{x}\!
			s\expe^{-\beta s} \drm s
			+
			\int_{0}^{x}\!
			\frac{1}{\beta}(\expe^{-\beta s}-\expe^{-\beta y}) \drm s
			\\&=
			\int_{0}^{x}\!
			s\expe^{-\beta s}\drm s
			+
			\frac{1}{\beta^2}(1-\expe^{-\beta x}) -\frac{x}{\beta}\expe^{-\beta y}
			\int_{0}^{x}\!
			s\expe^{-\beta s}\drm s
			\\&=
			\frac{1}{\beta^2}-\frac{1}{\beta}(x+\frac{1}{\beta})\expe^{-\beta x}
			+
			\frac{1}{\beta^2}(1-\expe^{-\beta x}) -\frac{x}{\beta}\expe^{-\beta y}
			\\&=
			\frac{1}{\beta^2}\bigg{(}2-2\expe^{-\beta x}
			-x\beta(\expe^{-\beta x}+\expe^{-\beta y})\bigg{)}.
		\end{split}	
	\end{equation*}
	Accordingly, since $x=\min(x,y)$ and $y=\max(x,y)$, one has
	\eqref{eqn:nu_closed}.
\end{proof}	
For each $i,j \in \{0,1,\ldots,\nS-1\}$, define function $\kappa_{ij}:\Rbb_+\times\Rbb_+\to\Rbb_+$ such that, for any $\tau_1,\tau_2\in\Rbb_+$, we have  
\begin{equation}
	\begin{split}
		&\kappa_{ij}(\tau_1,\tau_2) = 
		\nu(\bar{s}_{i}(\tau_1),\bar{s}_{j}(\tau_2)) 
		-
		\nu(\bar{s}_{i+1}(\tau_1),\bar{s}_{j}(\tau_2))
		\\&\quad-
		\nu(\bar{s}_{i}(\tau_1),\bar{s}_{j+1}(\tau_2)) 
		+
		\nu(\bar{s}_{i+1}(\tau_1),\bar{s}_{j+1}(\tau_2)).
	\end{split}
\end{equation}
\begin{proposition}
For any $\tau_1,\tau_2\in\Rbb_+$, we have
\begin{equation}\label{eqn:inner_phiu_tau1_phii_tau2}
\begin{split}
&\inner{\phiu{\tau_1}}{\phiu{\tau_2}}
= 
\sum_{i=0}^{\nS-1}\sum_{j=0}^{\nS-1}
%\sum_{i,j\in \Ical_{\text{s}}}
\xi_{i+1}\xi_{j+1} \kappa_{ij}(\tau_1,\tau_2).
\end{split}
\end{equation}
\end{proposition}
\begin{proof}
Due to \eqref{eqn:nu_def}, we have
\begin{equation}\label{eqn:intintk11}
\begin{split}
&\!\!\!\int_{0}^{\infty}\!\!\!\int_{0}^{\infty}\!\!  
\kernel(s,t)
\one_{[s_j,s_{j+1})}(\tau_2-t)
\one_{[s_i,s_{i+1})}(\tau_1-s)\drm t \ \! \drm s 	
\\&=
\int_{\bar{s}_{i+1}(\tau_1)}^{\bar{s}_{i}(\tau_1)}\!
\int_{\bar{s}_{j+1}(\tau_2)}^{\bar{s}_{j}(\tau_2)}\! 
\kernel(s,t)
\drm t \ \! \drm s
\\&=
\nu(\bar{s}_{i}(\tau_1),\bar{s}_{j}(\tau_2)) -
\nu(\bar{s}_{i+1}(\tau_1),\bar{s}_{j}(\tau_2))
\\&\quad
-\nu(\bar{s}_{i}(\tau_1),\bar{s}_{j+1}(\tau_2)) +
\nu(\bar{s}_{i+1}(\tau_1),\bar{s}_{j+1}(\tau_2)).
\end{split}
\end{equation}
From other hand, due to Lemma \ref{lem:Lui_bounded} and \eqref{eqn:u_pw}, we have
\begin{equation*}
\begin{split}
&\inner{\phiu{\tau_1}}{\phiu{\tau_2}}
=
\int_{0}^{\infty}\phiu{\tau_2,s}
\sum_{i=0}^{\nS-1} \xi_{i+1} \one_{[s_i,s_{i+1})}(\tau_1-s)\drm s 
\\&=
\int_{0}^{\infty}\int_{0}^{\infty}
\kernel(s,t)
\sum_{j=0}^{\nS-1} \xi_{j+1} \one_{[s_j,s_{j+1})}(\tau_2-t)
\\&\qquad\qquad
\sum_{i=0}^{\nS-1} \xi_{i+1} \one_{[s_i,s_{i+1})}(\tau_1-s)\drm t \ \! \drm s
\\&=
\sum_{i=0}^{\nS-1}\sum_{j=0}^{\nS-1} 
\xi_{i+1}\xi_{j+1}\bigg[
\int_{0}^{\infty}\int_{0}^{\infty}
\kernel(s,t)
\\&\qquad\qquad
\one_{[s_j,s_{j+1})}(\tau_2-t)
\one_{[s_i,s_{i+1})}(\tau_1-s)\drm t \ \! \drm s\bigg].
\end{split}	
\end{equation*}
Accordingly, one can conclude \eqref{eqn:inner_phiu_tau1_phii_tau2} due to \eqref{eqn:intintk11}.
\end{proof}	
\begin{proposition}
	Let $\Tbb=\Zbb_+$ and the system be initially at rest, i.e., we have $u_t=0$, for $t<0$.
	Also, let assume that $t_i=i$, for $i=0,\ldots,\nD-1$. 
	Define matrix $\mx{K}\in\Rbb^{\nD\times\nD}$  as 
	$\mx{K}:=\begin{bmatrix}
		\kernel(t-1,s-1)	
	\end{bmatrix}_{t=1,s=1}^{\nD,\nD}$ and  Toeplitz matrix $\mx{T}_{\vc{u}}\in\Rbb^{\nD\times\nD}$  as 
	$\mx{T}_{\vc{u}}:=
	\begin{bmatrix}
		u_{t-s}
	\end{bmatrix}_{t=1,s=1}^{\nD,\nD}$. Then, we have
	\begin{equation}
		\Big[\inner{\phiu{i-1}}{\phiu{j-1}}_{\Hk}\Big]_{i=1,j=1}^{\nD,\nD} = \mx{T}_{\vc{u}}\mx{K}\mx{T}_{\vc{u}}^\tr.    
	\end{equation}
\end{proposition}	
\begin{proof}
	The claim is a straightforward result of \eqref{eqn:phi_u_ts} and the definition of matrices  $\mx{K}$ and $\mx{T}_{\vc{u}}$.
\end{proof}
%================================
%\section{Appendix:Utilizing Barrier Method} 
\subsection{Utilizing Barrier Method}
\label{sec:appendix_barrier_QCQP}
Let $p:\Rbb^m\to\Rbb$ and $p_j:\Rbb^m\to\Rbb$, for $j\in\IP$,
be respectively defined as
$p(\vcx):= \frac12\|\mxA\vcx-\vcy\|^2 + \frac{\lambda}{2} \vc{x}^\tr\Phi\vc{x}$ and
$p_j(\vcx):= \frac12\vcx^\tr(\vcb_j\vcb_j^\tr+\vcc_j\vcc_j^\tr)\vcx-\frac12(1-\epsilon)$, for all $\vcx\in\Rbb^m$.
Then, the optimization problem \eqref{eqn:opt_7_finite} can be written as 
%\begin{equation}
%\begin{split}
%\minOp_{\vcx\in\Rbb^m} &\ \frac{1}{2}\|\mxA\vcx-\vcy\|^2 + \frac{\lambda}{2} \vc{x}^\tr\Phi\vc{x} 
%\\&+ \sum_{j\in\IP}
%\Ical\Big(
%\frac12 \vcx^\tr(\vcb_j\vcb_j^\tr+\vcc_j\vcc_j^\tr)\vcx-\frac12(1-\epsilon)
%\Big),
%\end{split}
%\end{equation}
\begin{equation}\label{eqn:min_p_Ipj}
	\minOp_{\vcx\in\Rbb^m} \ p(\vcx)+ \sum_{j\in\IP}
	\Ical(p_j(\vcx)),
\end{equation}
where
$\Ical:\Rbb\to\Rbb_+\cup\{+\infty\}$ is defined as
$\Ical(x):=\delta_{(-\infty,0]}(x)$, for all $x\in\Rbb$.
The solution approach utilizing logarithmic barrier function approximates $\Ical$ in \eqref{eqn:min_p_Ipj}  successively and  solves the resulting convex program to obtain a sequence in $\Rbb^m$ converging to the solution of \eqref{eqn:min_p_Ipj}. 
More precisely, for $n\in\Nbb$, let $b_n:\Rbb\to\Rbb_+\cup\{+\infty\}$  be the function defined as $b_n(x):=-\theta_n\ln(-x)$, for $x\in\Rbb$, where $\{\theta_n\}_{n=1}^\infty$ is a decreasing sequence such that $\theta_n\downarrow 0$.
Define function  $f_n:\Rbb\to\Rbb_+\cup\{+\infty\}$ as following
\begin{equation}
	f_n(\vcx) := p(\vcx) - \theta_n \sum_{j\in\IP}
	\ln(-p_j(\vcx)),\quad \forall\vcx\in\Rbb^m,
\end{equation}
for any $n\in\Nbb$. 
The gradient of $f_n$ is
\begin{equation}\label{eqn:Grad_fn}
	\begin{split}
		&\nabla f_n(\vcx) = 
		\mx{A}^\tr\mx{A}\vcx+ \lambda \Phi\vcx + 
		\vc{q} 
		\\& \qquad -
		\theta_n\sum_{j=0}^{\nP}
		\frac{(\vc{b}_j^\tr\vcx) \vc{b}_j+(\vc{c}_j^\tr\vcx)\vc{c}_j}
		{r+\frac{1}{2}(\vc{b}_j^\tr\vcx)^2+\frac{1}{2}
			(\vc{c}_j^\tr\vcx)^2},
	\end{split}
\end{equation}
and, the Hessian of $f_n$ is 
\begin{equation}\label{eqn:Grad2_fn}
	\begin{split}\!\!\!\!\!\!\!\!
		&\!\!\!\nabla^2\! f_n(\vcx) \!=\! 
		\mx{A}^\tr\!\mx{A}\!+\! \lambda \Phi \!-\! 
		\theta_n
		\!\sum_{j=0}^{\nP}
		\frac{\vc{b}_j\vc{b}_j^\tr+\vc{c}_j\vc{c}_j^\tr}
		{r+\frac{1}{2}(\vc{b}_j^\tr\vcx)^2+\frac{1}{2}
			(\vc{c}_j^\tr\vcx)^2}
		\\ &\!\!\!+\! 
		\theta_n
		\!\sum_{j=0}^{\nP}
		\frac{
			\left((\vc{b}_j^\tr\vcx) \vc{b}_j+(\vc{c}_j^\tr\vcx)\vc{c}_j\right)
			\left((\vc{b}_j^\tr\vcx) \vc{b}_j+(\vc{c}_j^\tr\vcx)\vc{c}_j\right)^\tr\!\!
		}
		{(r+\frac{1}{2}(\vc{b}_j^\tr\vcx)^2+\frac{1}{2}
			(\vc{c}_j^\tr\vcx)^2)^2}.
	\end{split}
\end{equation}
Given the gradient and the Hessian of $f_n$ in \eqref{eqn:Grad_fn} and \eqref{eqn:Grad2_fn}, one may employ an iterative optimization scheme such as Newton method or (L-)BFGS method \cite{boyd2004convex} to find $\vcx_n$, the solution of the convex program $\min_{\vcx\in\Rbb^m} f_n(\vcx)$. Note that $\vcx=\zero$ is feasible in \eqref{eqn:opt_7_finite}.
Also, one can see that the feasible set in \eqref{eqn:opt_7_finite} is a convex set with non-empty interior, and consequently, it is guaranteed \cite{boyd2004convex} that $\{\vcx_n\}_{n=1}^\infty$ converges to the solution of \eqref{eqn:min_p_Ipj}.

Given $\Phi\vcx$, we have the value of $\mxA\vcx$ as well as $\vc{b}_j^\tr\vcx$ and $\vc{c}_j^\tr\vcx$, for $j\in\IP$. Accordingly, by obtaining the vector $\Phi\vcx$, one can calculate the value in \eqref{eqn:Grad_fn} efficiently.
Moreover, one can see from \eqref{eqn:Grad2_fn} that the matrices $\mx{A}^\tr\mx{A}+ \lambda \Phi$ and $\vc{b}_j\vc{b}_j^\tr+\vc{c}_j\vc{c}_j^\tr$, for $j\in\IP$,  can be calculated only once in the optimization procedure. Note that since (L-)BFGS only employs the gradients of $f_n$, it is preferred when $\nP$ is large. 
\fi 
%_____________________________________
\bibliographystyle{IEEEtran}
%{amsalpha} %siam %amsalpha %IEEEtranS %plain %amsplain %abbrv
\bibliography{mainbib}
%_____________________________________

\end{document}